\documentclass[a4paper,11pt]{amsart}
\usepackage[utf8]{inputenc}
\usepackage{amsmath}
\usepackage{amsfonts}
\usepackage{amssymb}
\usepackage{amsthm}
\usepackage{mathrsfs}
\usepackage{calrsfs}
\usepackage[english]{babel}
\usepackage{enumitem}
\usepackage{graphicx}
\usepackage{hyperref,tikz}
\usepackage{array}
\usepackage{multirow}
\usepackage{xcolor}
\usepackage{float}

\renewcommand{\geq}{\geqslant}
\renewcommand{\leq}{\leqslant}
\renewcommand{\epsilon}{\varepsilon}
\newcommand{\sign}{\text{sign}}

\hypersetup{colorlinks=true,    linkcolor=MyDarkBlue,
	citecolor=BrickRed,       filecolor=BrickRed,       urlcolor=darkgreen
}
\usepackage{color}
\definecolor{darkgreen}{rgb}{0,0.4,0}
\definecolor{MyDarkBlue}{rgb}{0,0.08,0.85}
\definecolor{BrickRed}{rgb}{0.8,0.08,0}

\newtheorem{theorem}{Theorem}
\newtheorem{lemma}[theorem]{Lemma}
\newtheorem{remark}[theorem]{Remark}

\newtheorem{proposition}[theorem]{Proposition}

\numberwithin{equation}{section}

\title[Discrete harmonic functions for non-symmetric Laplace operators]{Discrete harmonic functions for non-symmetric Laplace operators in the quarter plane}
\author{Viet Hung Hoang}
\address{CNRS \and Institut Denis Poisson, UMR CNRS 7013, Universit\'e de Tours et Universit\'e d'Orl\'eans, Parc de Grandmont, 37200 Tours, France, Institut f\"ur Mathematische Stochastik, Universit\"at M\"unster, Orl\'eans-Ring 10, 48149 M\"unster, Germany}
\email{viet.hung-hoang@lmpt.univ-tours.fr}
\thanks{This project has received funding from the European Research Council (ERC) under the European Union's Horizon 2020 research and innovation programme under the Grant Agreement No.\ 759702.}

\date{\today}


\begin{document}

\maketitle
	
\begin{abstract}
	We construct harmonic functions in the quarter plane for discrete Laplace operators. In particular, the functions are conditioned to vanish on the boundary and the Laplacians admit coefficients associated with transition probabilities of non-symmetric random walks. By solving a boundary value problem for generating functions of harmonic functions, we deduce explicit expressions for the generating functions in terms of conformal mappings. These mappings are yielded from a conformal welding problem with quasisymmetric shift and contain information about the growth of harmonic functions. Further, we describe the set of harmonic functions as a vector space isomorphic to the space of formal power series.
\end{abstract}
	
\setcounter{tocdepth}{1}
\tableofcontents
\section{Introduction}\label{sec: intro}

\subsection*{Discrete harmonic functions}

Let us first recall some definitions. A real-valued function $h$ defined on a domain $D\subset\mathbb{Z}^2$ is called discrete harmonic (or preharmonic) with respect to the discrete Laplace operator
\begin{equation}\label{eq: discrete Laplacian}
	\Delta(h)(i,j) = \sum_{k,\ell} p_{k,\ell} h(i+k,j+\ell)- h(i,j),
\end{equation}
with $\{p_{k,\ell}\}$ fixed, if $\Delta(h)(i,j)=0$ for all $(i,j)\in D$. For the sake of brevity, the words ``harmonic function" and ``Laplace operator" will always be understood in the discrete context. The word ``discrete" will only be used in case where we want to emphasize it in relation to the continuous counterpart.

During the second half of the 20\textsuperscript{th} century, queuing theory was of interest to many probabilists. Discrete harmonicity appeared vastly under the form of balance equations to model the stationary distribution of many quantities (see the book \cite{Kl-75} of Kleinrock for an introduction of various models or the book \cite{Co-82} of Cohen for single server queues). Another important application is the connection to the concept of Doob's transform. Biane in \cite{Bi-91}, Eichelsbacher and K\"onig in \cite{EiKo-08} constructed processes conditioned to stay in cones from positive harmonic functions. Later, Denisov and Wachtel in \cite{DeWa-15} used the functions to find the tail asymptotics for the exit time of such conditioned walks. Harmonic functions are also found in some problems of population dynamics to model the extinction probabilities of the population (see \cite{LaRaTr-13,AlRa-19}).

The references listed above mostly concern positive harmonic functions on account of their evident links to probabilistic quantities. However, some authors continued to extend the problem further to signed harmonic functions (namely, the sign may vary in some domains). Almansi in \cite{Al-1899} proved that in the continuous context, polyharmonic functions can be decomposed into signed harmonic functions. The discrete analog of this result was achieved in \cite{CoCoGoSi-02,SaWo-21,ChFuRa-20}. Further in \cite{ChFuRa-20}, Chapon, Fusy and Raschel unveiled some connections between discrete polyharmonic functions and complete asymptotic expansions in walk enumeration problems. Therefore, studies of polyharmonic and/or signed harmonic functions become more relevant. Last but not least, one of our goals is to study the space of harmonic functions (see \cite{HoRaTa-20}), which gives us a general viewpoint on this discrete potential theory problem.

\subsection*{Overview of our main model and some studies of large jump random walks}
Hoang, Raschel and Tarrago in \cite{HoRaTa-20} studied discrete harmonic functions for symmetric Laplace operators in the quarter plane. More precisely, they considered the Laplacian \eqref{eq: discrete Laplacian} with $p_{k,\ell}=p_{\ell,k}$ for all $k,\ell\in\mathbb{Z}$. The present work aims at extending their results to a non-symmetric context. We therefore reintroduce our main problem and point out key differences between the models of this work and of \cite{HoRaTa-20}.

In this article, we consider the problem of finding discrete harmonic functions in the quarter plane (i.e., $D=\mathbb{N}^2=\{1,2,3,...\}^2$) which vanish on the boundary axes. The weights $\{p_{k,\ell}\}$ of the Laplacian \eqref{eq: discrete Laplacian} will be modeled by transition probabilities of random walks. This is motivated by the links between harmonic functions and various problems of queuing systems, population dynamics,... which can be reformulated as problems of random walks. In particular, we consider walks with arbitrary large negative jumps, small positive jumps, and moreover not being constrained by any symmetry conditions (see Table~\ref{table: walks}). This is the main difference of this paper's model in comparison with \cite{HoRaTa-20}. Due to the high order of the characteristic polynomial associated with large jumps, such walks are undoubtedly challenging to study. We refer the readers to some achievements in the subject: an analytic technique was developed in \cite{CoBo-83,Co-92} by Cohen and Boxma to study walks with two-dimensional state space; asymptotics of multidimensional walks are investigated in \cite{DeWa-15} by Denisov and Wachtel, \cite{DuWa-15} by Duraj and Wachtel; and recently in combinatorics, walks with steps in $\{-2,-1,0,1\}^2$ are systematically classified in \cite{BoBoMe-21} by Bostan, Bousquet-M\'elou and Melczer, alongside with some related conjectures. In brief, the study of large step walks is still largely open for further development.
\begin{table}[t]
	\begin{center}
		\begin{tabular}{| c | c |}
			\hline
			Symmetric small jump walks & Non-symmetric small jump walks \\
			\hline
			\begin{tikzpicture}[scale=.7] 
			\draw[->,white] (1.5,1.5) -- (-1.5,-1.5);
			\draw[->,white] (1.5,-1.5) -- (-1.5,1);
			\draw[->] (0,0) -- (-1,1);
			\draw[->] (0,0) -- (1,-1);
			\draw[->] (0,0) -- (1,0);
			\draw[->] (0,0) -- (-1,0);
			\draw[->] (0,0) -- (0,-1);
			\draw[->] (0,0) -- (0,1);
			\end{tikzpicture}
			\begin{tikzpicture}[scale=.7] 
			\draw[->,white] (1.5,1.5) -- (-1.5,-1.5);
			\draw[->,white] (1.5,-1.5) -- (-1.5,1);
			\draw[->] (0,0) -- (-1,0);
			\draw[->] (0,0) -- (1,0);
			\draw[->] (0,0) -- (0,1);
			\draw[->] (0,0) -- (0,-1);
			\end{tikzpicture}
			\begin{tikzpicture}[scale=.7] 
			\draw[->,white] (1.5,1.5) -- (-1.5,-1.5);
			\draw[->,white] (1.5,-1.5) -- (-1.5,1);
			\draw[->] (0,0) -- (1,1);
			\draw[->] (0,0) -- (-1,0);
			\draw[->] (0,0) -- (0,-1);
			\end{tikzpicture}
			&
			\begin{tikzpicture}[scale=.7] 
			\draw[->,white] (1.5,1.5) -- (-1.5,-1.5);
			\draw[->,white] (1.5,-1.5) -- (-1.5,1);
			\draw[->] (0,0) -- (-1,1);
			\draw[->] (0,0) -- (1,0);
			\draw[->] (0,0) -- (0,-1);
			\end{tikzpicture}
			\begin{tikzpicture}[scale=.7] 
			\draw[->,white] (1.5,1.5) -- (-1.5,-1.5);
			\draw[->,white] (1.5,-1.5) -- (-1.5,1);
			\draw[->] (0,0) -- (-1,0);
			\draw[->] (0,0) -- (1,0);
			\draw[->] (0,0) -- (-1,1);
			\draw[->] (0,0) -- (1,-1);
			\end{tikzpicture}
			\begin{tikzpicture}[scale=.7] 
			\draw[->,white] (1.5,1.5) -- (-1.5,-1.5);
			\draw[->,white] (1.5,-1.5) -- (-1.5,1);
			\draw[->] (0,0) -- (1,1);
			\draw[->] (0,0) -- (-1,1);
			\draw[->] (0,0) -- (0,-1);
			\end{tikzpicture}  \\
			\hline
			Symmetric large negative jump walks & Non-symmetric large negative jump walks\\
			\hline
			\begin{tikzpicture}[scale=.7] 
			\draw[->,white] (1.5,1.5) -- (-2,-2.5);
			\draw[->,white] (1.5,-2.5) -- (-2,1.5);
			\draw[->] (0,0) -- (1,1);
			\draw[->] (0,0) -- (0,1);
			\draw[->] (0,0) -- (1,0);
			\draw[->] (0,0) -- (-2,-2);
			\draw[->] (0,0) -- (-2,0);
			\draw[->] (0,0) -- (0,-2);
			\draw[->] (0,0) -- (1,-2);
			\draw[->] (0,0) -- (-2,1);
			\end{tikzpicture}
			&
			\begin{tikzpicture}[scale=.7] 
			\draw[->,white] (1.5,1.5) -- (-2,-2.5);
			\draw[->,white] (1.5,-2.5) -- (-2,1.5);
			\draw[->] (0,0) -- (1,1);
			\draw[->] (0,0) -- (0,1);
			\draw[->] (0,0) -- (1,0);
			\draw[->] (0,0) -- (-1,-2);
			\draw[->] (0,0) -- (-2,-2);
			\draw[->] (0,0) -- (-2,0);
			\draw[->] (0,0) -- (1,-2);
			\end{tikzpicture}\\
			\hline
			Symmetric large positive jump walks & Non-symmetric large positive jump walks\\
			\hline
			\begin{tikzpicture}[scale=.7] 
			\draw[->,white] (-1.5,-1.5) -- (2.5,2.5);
			\draw[->,white] (-1.5,-1.5) -- (2.5,2.5);
			\draw[->] (0,0) -- (-1,-1);
			\draw[->] (0,0) -- (0,-1);
			\draw[->] (0,0) -- (-1,0);
			\draw[->] (0,0) -- (2,2);
			\draw[->] (0,0) -- (2,0);
			\draw[->] (0,0) -- (0,2);
			\draw[->] (0,0) -- (-1,2);
			\draw[->] (0,0) -- (2,-1);
			\end{tikzpicture}
			&
			\begin{tikzpicture}[scale=.7] 
			\draw[->,white] (-1.5,-1.5) -- (2.5,2.5);
			\draw[->,white] (-1.5,-1.5) -- (2.5,2.5);
			\draw[->] (0,0) -- (-1,-1);
			\draw[->] (0,0) -- (0,-1);
			\draw[->] (0,0) -- (-1,0);
			\draw[->] (0,0) -- (1,2);
			\draw[->] (0,0) -- (2,2);
			\draw[->] (0,0) -- (2,0);
			\draw[->] (0,0) -- (-1,2);
			\end{tikzpicture}\\
			\hline
			Symmetric transition probabilities & Non-symmetric transition probabilities\\
			\hline
			\begin{tikzpicture}[scale=.7] 
			\draw[->,white] (2,2) -- (-2,-2);
			\draw[->,white] (2,-2) -- (-2,2);
			\draw[->] (0,0) -- (-1,0) node[left] {$1/4$};
			\draw[->] (0,0) -- (1,0) node[right] {$1/4$};
			\draw[->] (0,0) -- (0,1) node[above] {$1/4$};
			\draw[->] (0,0) -- (0,-1) node[below] {$1/4$};
			\end{tikzpicture}
			\begin{tikzpicture}[scale=.7] 
			\draw[->,white] (1,2) -- (-1,-2);
			\draw[->,white] (1,-2) -- (-1,1);
			\draw[->] (0,0) -- (-1,1) node[left] {$1/6$};
			\draw[->] (0,0) -- (1,-1) node[right] {$1/6$};
			\draw[->] (0,0) -- (1,0) node[right] {$1/6$};
			\draw[->] (0,0) -- (-1,0) node[left] {$1/6$};
			\draw[->] (0,0) -- (0,-1) node[below] {$1/6$};
			\draw[->] (0,0) -- (0,1) node[above] {$1/6$};
			\end{tikzpicture}
			&
			\begin{tikzpicture}[scale=.7] 
			\draw[->,white] (2,2) -- (-2,-2);
			\draw[->,white] (2,-2) -- (-2,2);
			\draw[->] (0,0) -- (-1,0) node[left] {$1/6$};
			\draw[->] (0,0) -- (1,0) node[right] {$1/6$};
			\draw[->] (0,0) -- (0,1) node[above] {$1/3$};
			\draw[->] (0,0) -- (0,-1) node[below] {$1/3$};
			\end{tikzpicture}
			\begin{tikzpicture}[scale=.7] 
			\draw[->,white] (1,2) -- (-1,-2);
			\draw[->,white] (1,-2) -- (-1,1);
			\draw[->] (0,0) -- (-1,1) node[left] {$1/8$};
			\draw[->] (0,0) -- (1,-1) node[right] {$1/8$};
			\draw[->] (0,0) -- (1,0) node[right] {$1/6$};
			\draw[->] (0,0) -- (-1,0) node[left] {$1/6$};
			\draw[->] (0,0) -- (0,-1) node[below] {$1/8$};
			\draw[->] (0,0) -- (0,1) node[above] {$1/8$};
			\end{tikzpicture}\\
			\hline
		\end{tabular}
	\end{center}
	\caption{Various examples of symmetric and non-symmetric walks}
	\label{table: walks}
\end{table}

\subsection*{Analytic approach to find harmonic functions}
We present here a complex analytic approach which will be the backbone of our work. The technique was pioneered by Malyshev in \cite{Ma-72}, Fayolle and Iasnogorodski in \cite{FaIa-79}, to study the stationary distribution of small jump random walks in a two-dimensional state space. The global idea is as follows. The balance equations for the stationary probabilities lead to a functional equation for their generating functions. An important bivariate polynomial characterized by the transition probabilities of the walk appears in the equation and is usually called the kernel function. By studying the Riemann surface associated with this kernel, then constructing a uniformizing variable, one can solve the equation in certain regions, then continue meromorphically the solutions into a bigger region. In another direction proposed in \cite{FaIa-79} by Fayolle and Iasnogorodski, the algebraic curve associated with the kernel offers a large choice of branches, and with an appropriate subset, the functional equation can be reduced to boundary value problems (BVP). Such a technique is also discussed in deeper insights by the same authors in the book \cite{FaIaMa-17}.

Since then, the technique has been successfully applied to various areas, such as queuing systems \cite{FlHa-84,Fl-85,KuSu-03}, potential theory \cite{Ra-11,Ra-14,LeRa-16} and enumerative combinatorics (counting walks in the quarter plane) \cite{BMMi-10,KuRa-12,DrHaRoSi-18}, leading to highly precise results (both exact and asymptotic results).

We also mention here the work \cite{BaFl-02} by Banderier and Flajolet. The authors considered combinatorial aspects of one-dimensional large jump walks as directed (large jump) two-dimensional walks. By the analytic approach, the counting generating functions are shown to have algebraic forms constructed by the branches of the associated algebraic curve.

And lastly, an alternative approach introduced by Cohen and Boxma in \cite{CoBo-83} contributes largely to the development of this technique (see also \cite{Co-92} by Cohen). The overall idea is the same as Fayolle and Iasnogorodski's. However, the subset of the kernel's zero set is chosen differently, in a way that this choice does not depend on the branches of the algebraic curve. Their construction overcomes the obstacle of studying precisely the Riemann surface, therefore it can cover a wide class of large jump random walks. This approach will also be the central point of our present analysis.

\subsection*{Non-symmetric walks with large negative jumps and harmonic functions}

Consider a lattice walk on $\mathbb{Z}^2$ with transition probabilities (or non-negative weights) $\{p_{k,\ell}\}_{(k,\ell)\in\mathbb Z^2}$ summing to $1$ such that:
\begin{enumerate}[label=(H\arabic{*}),ref={\rm (H\arabic{*})}]
	\item\label{assumption: H1} The walk has small positive jumps (at most $1$ unit) and arbitrary large  negative jumps, i.e., $p_{k,\ell} = 0$ if $k\geq 2$ or $\ell \geq 2$ (see Table~\ref{table: walks});
	
	\item\label{assumption: H2} If $\big\vert \sum p_{k,\ell}x^ky^\ell \big\vert=1$ and $\vert x\vert=\vert y\vert=1$, then $x=y=1$ (which is a necessary condition for the irreducibility of the walk);
	
	\item\label{assumption: H3.2} $\{p_{k,\ell}\}_{(k,\ell)\in\mathbb Z^2}$ have finite moments of order $2$, i.e., $\sum(k^2+\ell^2)p_{k,\ell} <\infty$;
	
	\item\label{assumption: H3} The drift is zero, i.e., $\sum k p_{k,\ell}=\sum \ell p_{k,\ell}=0$.
\end{enumerate}

Now let $\Delta$ be the discrete Laplacian \eqref{eq: discrete Laplacian} associated with $\{p_{k,\ell}\}_{(k,\ell)\in\mathbb Z^2}$ and acting on functions $h:\mathbb{Z}^2\to \mathbb{R}$. Our goal is to describe real-valued functions $\{h(i,j)\}_{i,j\in\mathbb{Z}^2}$ satisfying the following conditions:
\begin{enumerate}[label=(H\arabic{*}),ref={\rm (H\arabic{*})}]\setcounter{enumi}{4}
	\item\label{assumption: H4} $h(i,j)$ is harmonic with respect to the operator $\Delta$ on the positive quadrant, i.e., for all $i,j\geq 1$, $\Delta(h)(i,j)=0$;
	
	\item\label{assumption: H5} $h(i,j)$ vanishes elsewhere, i.e., for all $(i,j)\in\mathbb Z^2$ with $i\leq0$ and/or $j\leq0$, $h(i,j)=0$.
\end{enumerate}

Random walks with such non-symmetric step sets have appeared throughout the literature. In \cite{BMMi-10,BoRaSa-14}, the authors study the problem of counting walks for numerous small jump random walks. In \cite{Bi-91,EiKo-08,KoSc-10,Ra-11}, walks staying in Weyl chambers are investigated. For example, walks in the duals of SU(3) and Sp(4) can be viewed as non-symmetric walks in the quadrant.

\subsection*{Continuous counterpart of the main problem}
For future use, we recall some facts about continuous harmonic functions in cones. A real-valued function $f$ defined on a domain $D\subset\mathbb{R}^2$ is called harmonic with respect to the (continuous) Laplace operator
\begin{equation}\label{eq: continuous laplacian}
	\nabla^2 := \partial_{xx} + \partial_{yy},
\end{equation}
if $\nabla^2 f = 0$ on $D$.

Consider the following Dirichlet problem on the cone $\mathcal{E}:=\{(r\cos\phi,r\sin\phi):r> 0,0< \phi< \theta\}$ of angle $\theta\in(0,\pi]$: find functions $f:\mathcal{E}\to \mathbb{R}$ such that
\begin{enumerate}[label=\textnormal{(\roman{*})},ref=\textnormal{(\roman{*})}]
	\item $f$ is harmonic with respect to $\nabla^2$ on $\mathcal{E}$;
	
	\item $f$ is continuous on the closure $\overline{\mathcal{E}}$ of $\mathcal{E}$ and vanishes on the boundary $\partial\mathcal{E}$.
\end{enumerate}

If the cone is the upper half-plane $\mathbb{R}\times\mathbb{R}^+$, then the problem can easily be solved by complex analysis: by Schwarz reflection principle, the function $f$ can be extended to a harmonic function on $\mathbb{R}^2$ by the formula $f(x,y) = - f(x,-y)$ for all $(x,y)\in\mathbb{R}\times\mathbb{R}_-$. Moreover, a function $f:\mathbb{R}^2\to\mathbb{R}$ is harmonic on $\mathbb{R}^2$ if and only if there exists uniquely a complex analytic function $g$ (up to an additive constant) on $\mathbb{C}$ such that $f(x,y) = \Im g(x+yi)$ (with $\Im$ denoting the imaginary part) for any $(x,y)\in\mathbb{R}^2$. Therefore, $g(\overline{z}) = \overline{g(z)}$ for $z\in\mathbb{C}$ by Schwarz reflection principle. In other words, $g(z)$ admits the form $g(z)=\sum_{n\geq 0}a_nz^n$ with $(a_n)_n\subset\mathbb{R}$ and $\vert a_n\vert^{1/n}\to 0$ (which is a necessary and sufficient condition for $g(z)$ to be analytic on $\mathbb{C}$). One can describe the set of its solutions as the vector space
\begin{align*}
	H(\mathbb{R}\times\mathbb{R}^+)= \{ \sum_{n\geq 1} a_n \Im\big((x+iy)^{n}\big): (a_n)_n\subset\mathbb{R}\text{ and }\vert a_n\vert^{1/n}\to 0\}.
\end{align*}
By the mapping $z\mapsto z^{\theta/\pi}$ which maps conformally the upper half-plane to a cone of angle $\theta$ (in the complex plane), we then obtain a similar description for the solution set of the problem on the cone $\mathcal{E}$ of angle $\theta$ as
\begin{align}
	H(\mathcal{E})= \{ \sum_{n\geq 1} a_n \Im\big((x+iy)^{n\pi/\theta}\big): (a_n)_n\subset\mathbb{R}\text{ and }\vert a_n\vert^{1/n}\to 0\}.
\end{align}

Through the linear transform $(x,y)\mapsto (x\sin\theta-y\cos\theta,y)$ which maps the cone of angle $\theta$ to the first quadrant $\mathcal{Q}:=\mathbb{R}_+^2$, the problem on the cone can be reformulated into the one on the first quadrant: find functions $f:\mathcal{Q}\to \mathbb{R}$ such that
\begin{enumerate}[label=\textnormal{(\roman{*})},ref=\textnormal{(\roman{*})}]
	\item $f$ is analytic and satisfies the equation
	\begin{equation}
		\partial_{xx}f - 2\cos(\theta)\partial_{xy}f + \partial_{yy}f =0
	\end{equation}
	on the first quadrant $\mathcal{Q}$;
	\item $f$ is continuous on $\overline{\mathcal{Q}}$ and vanishes on the boundary $\partial\mathcal{Q}$.
\end{enumerate}
As a result, the corresponding solution set of this problem is also isomorphic to the solution set $H(\mathcal{E})$ and is described as:
\begin{equation}\label{eq: set of continuous harmonic funcs}
	H(\mathcal{Q}) = \{ \sum_{n\geq 1} a_n h_{\theta,n}: (a_n)_n\subset\mathbb{R}\text{ and }\vert a_n\vert^{1/n}\to 0\},
\end{equation}
where
\begin{equation}\label{eq: h_n^sigma}
	h_{\theta,n} (x,y) := \Im \bigl((x/\sin\theta+y\cot\theta+{i}y)^{n\pi/\theta}\bigr).
\end{equation}

The reason for mentioning these facts about the continuous problem is twofold: on one hand, we will see later the analogous description to \eqref{eq: set of continuous harmonic funcs} of the discrete counterpart; on the other hand, we will also emphasize a convergence of the discrete harmonic functions to the continuous ones.

\subsection*{Generating functions and the kernel}
Let $h(i,j)$ be any harmonic function satisfying the conditions \ref{assumption: H4} and \ref{assumption: H5}. We now introduce its generating function
\begin{equation}\label{eq: H(x,y)-def}
	H(x,y)=  \sum_{i,j\geq 1} h(i,j) x^{i-1}y^{j-1},
\end{equation}
and its sections
\begin{equation}\label{eq: H(x,0)-H(0,y)-def}
	H(x,0)= \sum_{i\geq 1} h(i,1) x^{i-1}
	\quad\text{and}\quad H(0,y)= \sum_{j\geq 1} h(1,j) y^{j-1}.
\end{equation}
From the Assumptions~\ref{assumption: H1}, \ref{assumption: H4}, and \ref{assumption: H5}, one has
\begin{align*}
	H(x,y)&=\sum_{k,\ell}p_{k,\ell}x^{-k}y^{-\ell}\sum_{i,j\geq 1} h(i+k,j+\ell) x^{i+k-1}y^{j+\ell-1}\\
	&= \sum_{k,\ell\leq 1}p_{k,\ell }x^{-k}y^{-\ell}H(x,y) - \sum_{k\leq 1}p_{k,1} H(x,0) - \sum_{\ell\leq 1}p_{1,\ell} H(0,y) + p_{1,1}x^{-1}y^{-1}H(0,0).
\end{align*}
In other words, $H(x,y)$ satisfies the functional equation
\begin{equation}\label{eq: functional-eq}
	K(x,y)H(x,y) = K(x,0) H(x,0) +K(0,y) H(0,y)-K(0,0) H(0,0),
\end{equation}
where $K(x,y)$ is called the kernel and is characterized by the weights $\{p_{k,\ell}\}_{(k,\ell)\in\mathbb Z^2}$:
\begin{equation}\label{eq: K(x,y)-def}
	K(x,y)= xy\left(1-\sum p_{k,\ell}x^{-k}y^{-\ell}\right) =xy-\sum p_{k,\ell}x^{-k+1}y^{-\ell+1}.
\end{equation}

It is easily verified that solutions of Equation~\eqref{eq: functional-eq} take the form
\begin{equation}\label{eq: F+G/K}
	H(x,y)=\frac{F(x)+G(y)}{K(x,y)},
\end{equation}
for any power series $F,G\in\mathbb{C}[[t]]$ (formal ring of power series with coefficients in $\mathbb{C}$). However, not every quotient $\big(F(x)+G(y)\big)/K(x,y)$ defines a bivariate power series around the center $(0,0)$. This raises the task of finding appropriate $F,G\in\mathbb{C}[[t]]$ such that the quotient is a power series, which also leads to the study of the kernel's zero set, particularly near the point $(0,0)$. 

\subsection*{A subset of the kernel's zero set}
We now introduce the subset of the kernel's zero set under the approach of Cohen and Boxma in \cite{CoBo-83}:
\begin{equation}\label{eq: mathcal K}
	\mathcal{K}:=\{ (x,y)\in\mathbb{C}^2:K(x,y)=0\text{ and }\vert x\vert =\vert y\vert \leq 1\}.
\end{equation}

As will be shown in Section~\ref{sec: kernel}, $\mathcal{K}$ behaves differently in three subcases and consists of a unique connected set, except in the case $p_{1,1}=0$, $p_{1,0}\neq p_{0,1}$, where it also consists of an isolated point $(0,0)$. The projections of its connected subsets along the first and second variables are respectively denoted by $\mathcal{S}_1$ and $\mathcal{S}_2$. Let $X:\mathcal{S}_2\to\mathcal{S}_1$ be the mapping such that for all $y\in\mathcal{S}_2$, $(X(y),y)\in\mathcal{K}$. Under Assumption~\ref{assumption: monotonic}, $\mathcal{S}_1$ and $\mathcal{S}_2$ are Jordan curves, $X(y)$ is a well defined and one-to-one mapping from $\mathcal{S}_2$ onto $\mathcal{S}_1$. Further, as $y$ moves on $\mathcal{S}_2$, $X(y)$ moves on $\mathcal{S}_1$ in different orientation. Let $Y:\mathcal{S}_1\to\mathcal{S}_2$ be the inverse of $X$. Under Assumption\ref{assumption: monotonic}, we can first deduce further insights of the kernel's zero set in the bidisk (Proposition~\ref{prop: extend-solutions-intro}), and later define conformal mappings from the interior of $\mathcal{S}_1$ and $\mathcal{S}_2$ onto disjoint, complementary domains, which help constructing appropriate $F(x)$ and $G(y)$ in \eqref{eq: F+G/K}.

\subsection*{Main results}
In order to state our theorems properly, we introduce here some new notations. The interior (resp.~exterior) of the Jordan curve $\mathcal{S}_1$ is denoted by $\mathcal S_1^+$ (resp.~$\mathcal S_1^-$). $\mathcal{S}_2^+$ and $\mathcal{S}_2^-$ are denoted similarly.  By analytic continuation, we first have the following result concerning solutions of the kernel in the bidisk.
\begin{proposition}\label{prop: extend-solutions-intro}
	Under Assumption \ref{assumption: monotonic},
	\begin{enumerate}[label=\textnormal{(\roman{*})},ref=\textnormal{(\roman{*})}]
		\item\label{item: Prop1-item1} The function $X:\mathcal{S}_2\to\mathcal{S}_1$, defined such that $(X(y),y)\in\mathcal{K}$ for all $y\in\mathcal{S}_2$, can be continued analytically and is uni-valued on $\mathcal{S}_2^-\cap\mathcal{C}^+$. In other words, the extension of $X(y)$ has no branch point on $\mathcal{S}_2^-\cap\mathcal{C}^+$. We have an analogous statement for $Y(x)$, which is the inverse of $X(y)$;
		
		\item\label{item: Prop1-item2} $K(x,y)$ has no root on $\mathcal{S}_1^+\times\mathcal{S}_2^+$.
	\end{enumerate}
\end{proposition}

Proposition~\ref{prop: extend-solutions-intro} may seem technical but it is crucial to indicate domains where $H(x,y)$ in \eqref{eq: F+G/K} is well defined. Particularly in the case $p_{1,1}=0$, we will show that the generating function $H(x,y)$, initially defined on $\mathcal{S}_1^+\times\mathcal{S}_2^+$, also admits an extension around $(0,0)$. It is also worth mentioning that its item \ref{item: Prop1-item1} is a new statement, whereas its item \ref{item: Prop1-item2} has been stated for the symmetric cases of \cite{HoRaTa-20}, and will be proven in Section~\ref{sec: kernel} for non-symmetric cases.

Our main theorems also are the main results in \cite{HoRaTa-20}, but now applied to non-symmetric cases under two technical assumptions \ref{assumption: theta_1,theta_2 not=0} and \ref{assumption: alpha'(z)}. Roughly speaking, \ref{assumption: theta_1,theta_2 not=0} ensures that $\mathcal{S}_1$ and $\mathcal{S}_2$ do not admit a cusp at $1$, and \ref{assumption: alpha'(z)} ensures that a ``shift" function, appearing in a conformal welding problem (Section~\ref{sec: conformal welding}), has power growth derivatives. Under the assumptions \ref{assumption: monotonic}--\ref{assumption: alpha'(z)}, $\mathcal{S}_1^+$ and $\mathcal{S}_2^+$ are mapped conformally onto two disjoint complementary domains, respectively by two mappings $\psi_1$ and $\psi_2$ such that:
\begin{enumerate}[label=\textnormal{(\roman{*})},ref=\textnormal{(\roman{*})}]
	\item $\psi_1^+(X(y)) = \psi_2^+(y)$ for all $y\in\mathcal{S}_2$, where $\psi_1^+$ and $\psi_2^+$ are respectively the continuous extension of $\psi_1$ and $\psi_2$ to $\mathcal{S}_1$ and $\mathcal{S}_2$;
	
	\item $\psi_1^+(1)=\psi_2^+(1)=\infty$;
	
	\item $\psi_1(\overline{z}) = \overline{\psi_1(z)}$ for all $z\in\mathcal{S}_1^+$, and $\psi_2(\overline{z}) = \overline{\psi_2(z)}$ for all $z\in\mathcal{S}_2^+$.
\end{enumerate}
$\psi_1$ and $\psi_2$ will be showed to be well defined around $0$ whether or not $0$ is in $\mathcal{S}_1^+,\, \mathcal{S}_2^+$.

We also introduce a related angle:
\begin{equation}\label{eq: theta}
	\theta:= \arccos\frac{-\sum k\ell p_{k,\ell}}{\sqrt{\sum k^2 p_{k,\ell}}\sqrt{\ell^2 p_{k,\ell}}}.
\end{equation}
It turns out that  $\theta$ is the arithmetic mean of the angles at $1$ of $\mathcal{S}_1$ and $\mathcal{S}_2$, and contains information about the growth of harmonic functions.

Finally, we introduce the following polynomials forming a basis of the space $\mathbb{R}[X]$:
\begin{equation}\label{eq: family-polynomials}
\begin{split}
	& P_{2n+1} = (X-\psi_1(0))^{n+1}  (X-\psi_2(0))^n, n\geq 0,\\
	& P_{2n} = (X-\psi_1(0))^n  (X-\psi_2(0))^n, n\geq 1.
\end{split}
\end{equation}

We then have our first theorem.
\begin{theorem}\label{thm: main_intro-1}
	Under the assumptions \ref{assumption: monotonic}--\ref{assumption: alpha'(z)}, for any $n\geq 1$, the function
	\begin{equation}\label{eq: H_n(x,y)}
		H_n(x,y):=\frac{P_n(\psi_{1}(x))-P_n(\psi_{2}(y))}{K(x,y)}
	\end{equation}
	defines the generating function of a harmonic function $h_n(i,j)$ satisfying \ref{assumption: H4} and \ref{assumption: H5}. Moreover, $H_n$ is analytic in $\mathcal{S}_1^+\times\mathcal{S}_2^+$.
	
	In particular, the discrete Laplace transform of normalized discrete harmonic functions converges pointwise to the continuous Laplace transform of continuous harmonic functions, that is,
	\begin{equation}
		\lim_{m\to\infty}\frac{c}{m^{n\pi/\theta+1}}\mathcal{L} h_{n}(\lfloor mx\rfloor ,\lfloor my\rfloor)=  \mathcal{L}h_{\theta,n}(x,y),
	\end{equation}
	for any $x,y>0$, where $c$ is a non-vanishing constant and $h_{n,\theta}$ is defined in \eqref{eq: h_n^sigma}, with $\theta$ in \eqref{eq: theta}, and the Laplace transform $\mathcal{L}$ is defined as follows
\begin{equation}
\mathcal{L}f(x,y):=\begin{cases}
\sum_{u,v=0}^\infty f(u,v)e^{-(ux+vy)}, & f:\mathbb{N}^2\to\mathbb{C},\\
\int_{0}^{\infty}\int_{0}^{\infty} f(u,v)e^{-(ux+vy)}dudv, & f:\mathcal{Q}\to\mathbb{C}.
\end{cases}
\end{equation}
\end{theorem}

From the harmonic functions constructed in Theorem~\ref{thm: main_intro-1}, we can describe the set of harmonic functions as a vector space.

\begin{theorem}\label{thm: main_intro-2}
	Under the assumptions \ref{assumption: monotonic}--\ref{assumption: alpha'(z)}, the space of discrete harmonic functions is isomorphic to the vector space of formal power series $\mathbb{R}_{0}[[t]]$ with vanishing constant term, through the isomorphism 
	\begin{equation*}
	\Phi:\left\lbrace\begin{matrix}
	\mathbb{R}[[t]]&\longrightarrow&H(\mathbb N^2)\\
	\sum_{n\geq 1}a_{n}t^n&\longmapsto& \sum_{n\geq 1}a_{n}h_n.
	\end{matrix}\right.
	\end{equation*}
\end{theorem}

Among the additional assumptions \ref{assumption: monotonic}--\ref{assumption: alpha'(z)}, the condition~\ref{assumption: alpha'(z)} is particularly technical and difficult to verify, since it concerns the growth of a function appearing in the analysis. Therefore, we also introduce a sufficient condition asserting the moments of order 4 of the weights $\{p_{k,\ell}\}_{k,\ell}$ (Assumption~\ref{assumption: H7 moment order 4}), under which the assumption~\ref{assumption: alpha'(z)} holds true. Then Theorems~\ref{thm: main_intro-1} and \ref{thm: main_intro-2} are also valid under Assumptions~\ref{assumption: H1}--\ref{assumption: H7 moment order 4}, \ref{assumption: monotonic} and \ref{assumption: theta_1,theta_2 not=0}.

Although the results are similar to that of \cite{HoRaTa-20}, the arguments for non-symmetric cases require more assumptions and the analysis of quasiconformal mapping. We therefore dedicate a table of comparison between the two models to emphasize their differences (see Table~\ref{table: comparison}).
\begin{table}[t]
	\begin{center}
		\begin{tabular}{| c | c |}
			\hline
			Symmetric random walk in \cite{HoRaTa-20} & Non-symmetric random walk in this paper\\
			\hline
			$\mathcal{S}_1$ and $\mathcal{S}_2$ are non-self-intersecting.&$\mathcal{S}_1$ and $\mathcal{S}_2$ can be self-intersecting.\\
			 & (requires Assumption \ref{assumption: monotonic} to exclude\\
			 &the cases of self-intersecting curves)\\
			 \hline
			$X(y)\in\mathcal{S}_1$ and $y\in\mathcal{S}_2$ & $X(y)\in\mathcal{S}_1$ and $y\in\mathcal{S}_2$\\
			have different orientations. & can have the same orientations.\\
			 & (requires Assumption \ref{assumption: monotonic} to exclude\\
			 &the cases of  the same orientations)\\
			\hline
			Angles at $1$ of $\mathcal{S}_1$ and $\mathcal{S}_2$ are equal & Angles at $1$ of $\mathcal{S}_1$ and $\mathcal{S}_2$ can be different\\
			and different from $0$& or equal to $0$.\\
			 & (requires assumption \ref{assumption: theta_1,theta_2 not=0} to exclude\\
			 &the cases of the cusps at $1$)\\
			\hline
			\multicolumn{2}{|c|}{Case $p_{1,1}=0$: $\mathcal{S}_1^+$ and $\mathcal{S}_2^+$ contain $0$.}\\
			\multicolumn{2}{|c|}{Case $p_{1,1}=0$, $p_{1,0}=p_{0,1}$: $\mathcal{S}_1$ and $\mathcal{S}_2$ contain $0$.}\\
			\multicolumn{2}{|c|}{(same proof of Proposition~\ref{prop: extend-solutions-intro})}\\
			\hline
			& Case $p_{1,1}=0$, $p_{1,0}\not= p_{0,1}$:\\
			(Case $p_{1,1}=0$, $p_{1,0}\not= p_{0,1}$&$0$ lies in the interior of one curve\\
			does not exist)&and the exterior of the other curve.\\
			 &(requires analytic continuations\\
			 &to properly define functions at $0$)\\
			\hline
			$\mathcal{S}_1$ and $\mathcal{S}_2$ coincide, & $\mathcal{S}_1$ and $\mathcal{S}_2$ do not coincide.\\
			$X(y)=\overline{y}$ for all $y\in\mathcal{S}_2$.&\\
			(Conformal welding is naturally&(Conformal welding is obtained\\
			obtained on the unit circle)&by solving a BVP with quasisymmetric shift\\
			&under Assumption~\ref{assumption: alpha'(z)})\\
			\hline
		\end{tabular}
	\end{center}
	\caption{Comparison between symmetric and non-symmetric cases}
	\label{table: comparison}
\end{table}

\subsection*{Acknowledgements}
We would like to thank Kilian Raschel and Pierre Tarrago for useful discussions. We also thank an anonymous referee for very constructive and detailed remarks: he/she made us realize that one of our technical assumptions can be replaced by a moment assumption.

\section{Study of the kernel}\label{sec: kernel}


Subsection~\ref{subsec: mathcal K-description} aims at describing the set $\mathcal{K}$ in \eqref{eq: mathcal K} through a parameterisation, and presenting its characteristics under Assumption~\ref{assumption: monotonic} about the path of its connected subset. We emphasize the behavior of $\mathcal{K}$ at $(1,1)$, as it admits a corner point, whose associated angles are strongly related to the growth of harmonic functions. These properties are roughly similar to those in \cite{HoRaTa-20} and admit vastly the same proof schemes. In Subsection~\ref{subsec: neighborhood of mathcal K}, we show some behaviors of the kernel's solutions near $\mathcal{K}$ by analytic continuation. Finally, Subsection~\ref{subsec: proof-Prop_extend-solutions} will present the proof of Proposition~\ref{prop: extend-solutions-intro}.

We introduce some notations. Let $\mathcal{C}$, $\mathcal{C}^+$, and $\overline{\mathcal{C}^+}$ respectively denote the unit circle, the open and closed unit disk. We also denote $\overline{z}$ as the complex conjugate of any point $z\in\mathbb{C}$. Finally, the abbreviation ``$\text{const}$" will indicate a non-vanishing constant in $\mathbb{C}$.

\subsection{Characteristics of $\mathcal{K}$}\label{subsec: mathcal K-description}

Adapting the analysis of Boxma and Cohen in \cite[Sec.~II.3.2]{CoBo-83}, we first parameterise $(x,y)$ with $\vert x\vert =\vert y\vert \leq 1$ as
\begin{equation}\label{eq: x_y_eta_s}
	(x,y) = (\eta s,\eta s^{-1} ),
\end{equation}
with $\eta,s\in\mathbb{C}$, $\vert \eta\vert \leq 1$ and $\vert s\vert=1$. This interpretation comes in handy since the trajectory of the variable $s$ is already known. The kernel \eqref{eq: K(x,y)-def} then becomes
\begin{equation}\label{eq: K-eta-s}
	K(\eta s,\eta s^{-1}) = \eta ^2-\sum p_{k,\ell}\eta ^{-k-\ell+2}s^{-k+\ell}.
\end{equation}

Before presenting our analysis, we make two remarks as they will be needed for future use:
\begin{enumerate}[label=\textnormal{(\roman{*})},ref=\textnormal{(\roman{*})}]
	
	\item Since $\sum p_{k,\ell}=1$, then $K(x,y)$ converges for all $\vert x\vert ,\vert y\vert \leq 1$, and thus it is analytic in the open bidisk $\mathcal{C}^+\times\mathcal{C}^+$. In particular, $K(x,y)$ is a polynomial if the jumps are bounded;
	
	\item Assumption~\ref{assumption: H2} (irreducible random walks) implies that $p_{1,1}$, $p_{0,1}$ and $p_{1,0}$ cannot simultaneously vanish.
\end{enumerate}


\begin{figure}[t]
	\centering
	\includegraphics[scale=0.2]{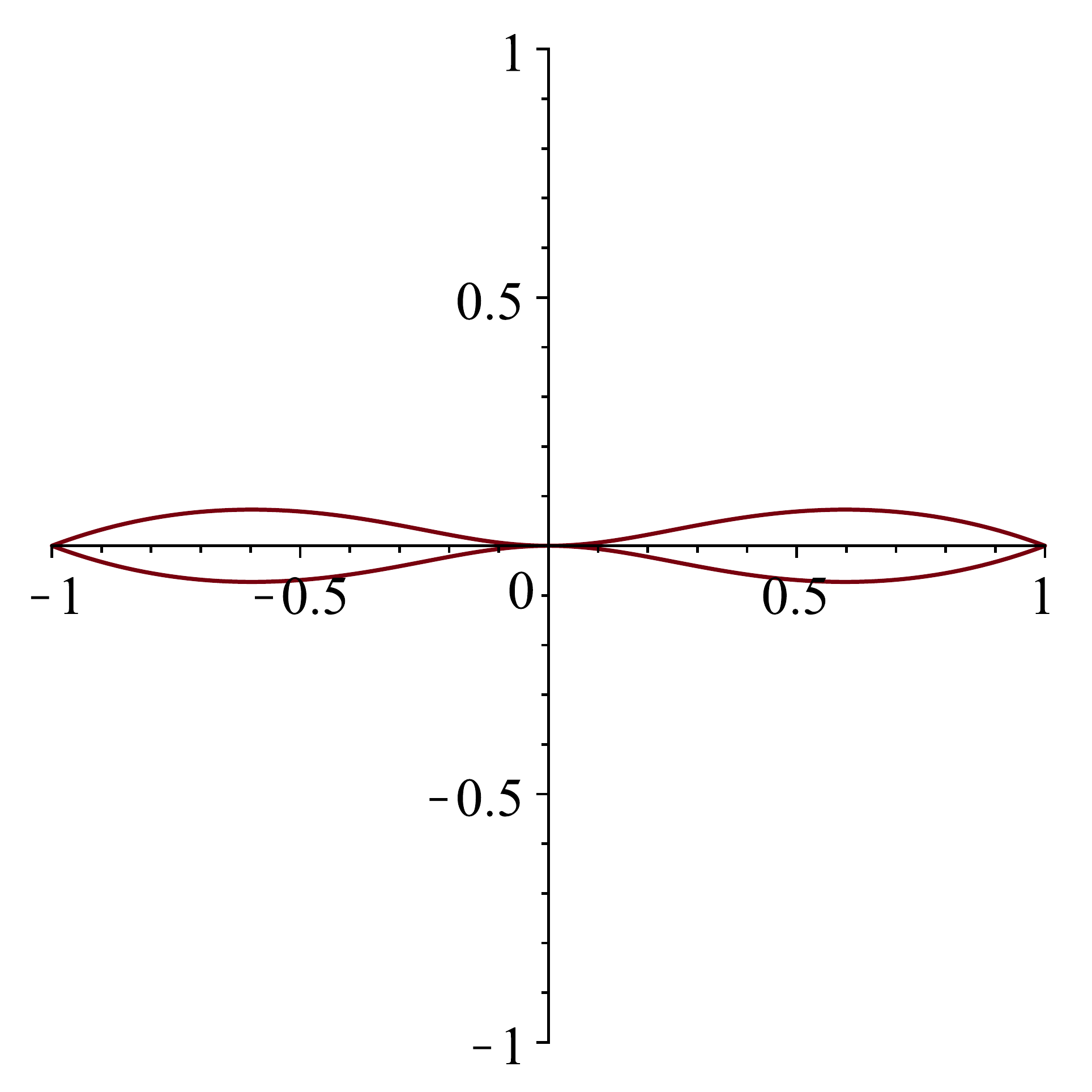}
	\includegraphics[scale=0.2]{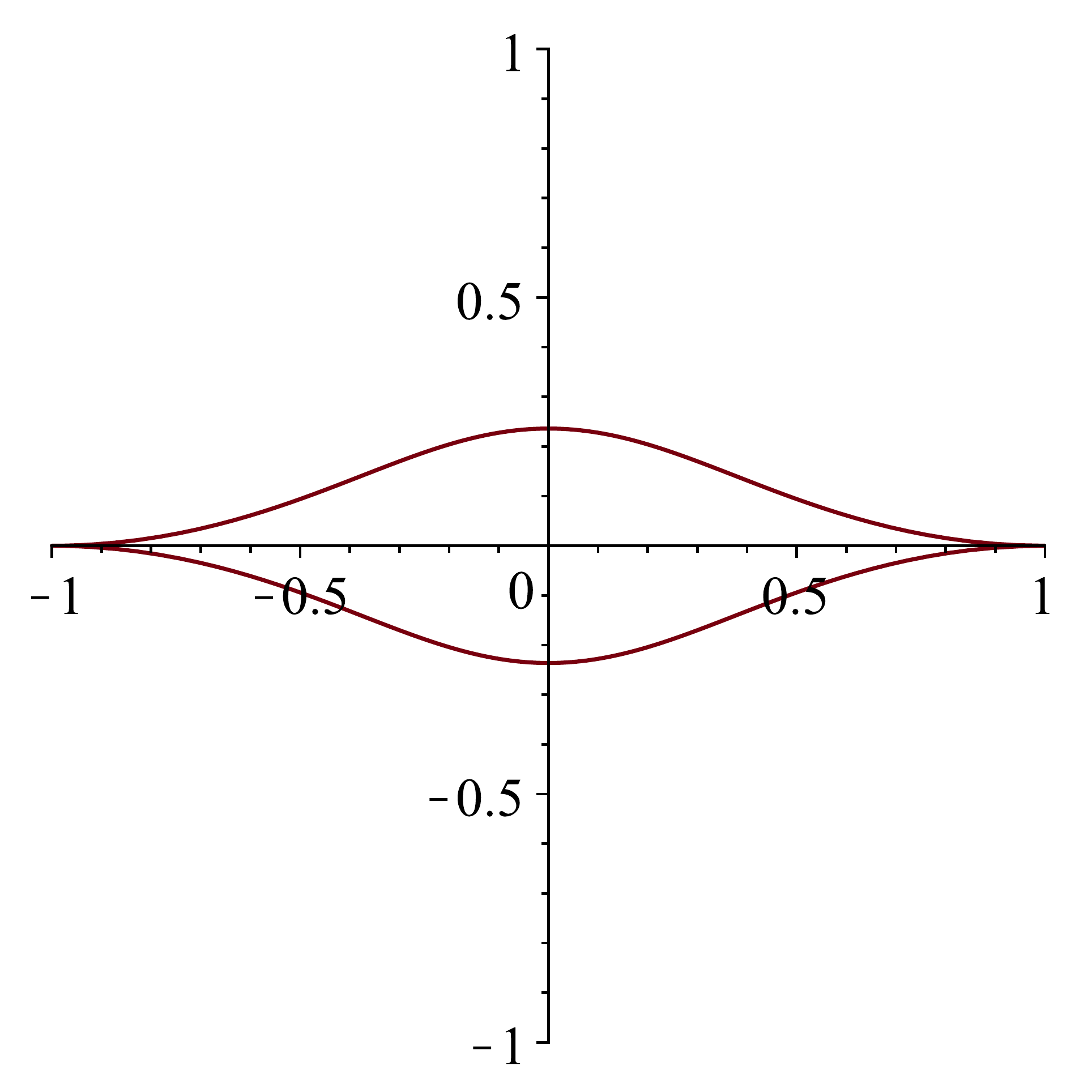}
	\includegraphics[scale=0.2]{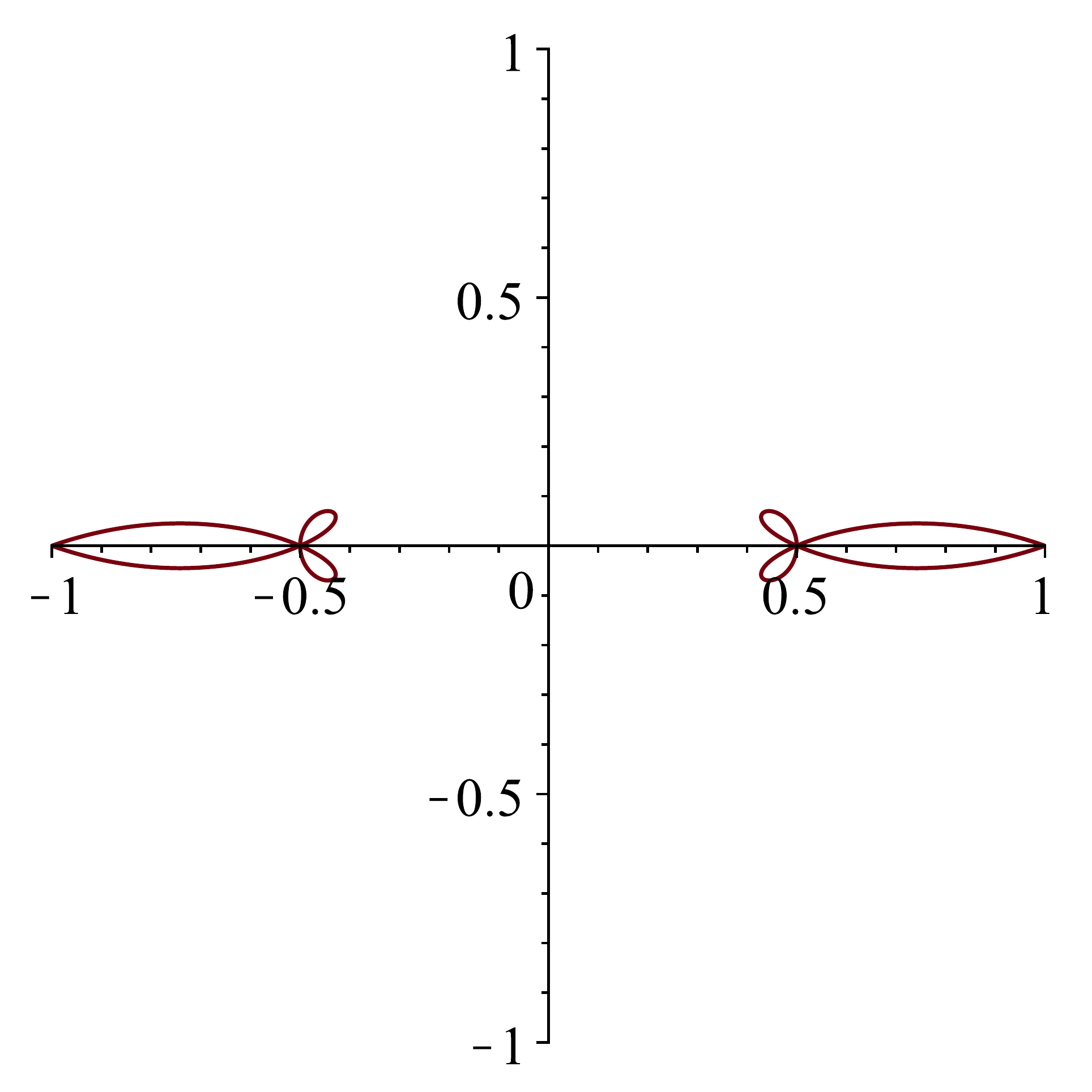}
	\caption{Solutions of $K(\eta s,\eta s^{-1})=0$, $\vert s\vert =1$ for some walks, from left to right: $p_{0,1}=p_{1,0}=p_{-1,1}=p_{0,-1}/2=1/5$; $p_{0,1}=p_{-1,0}=p_{1,-1}=1/3$; $p_{1,1}=p_{-1,1}=p_{0,-1}/2=1/4$.}
	\label{fig: eta(s)}
\end{figure}


The following lemma concerns the number of solutions of $K(\eta s,\eta s^{-1})$ as a function of $\eta$ with parameter $s\in\mathcal{C}$.

\begin{lemma}\label{lem: solutions-mathcal-K}
	Under Assumptions~\ref{assumption: H1}--\ref{assumption: H3}, for any $s\in \mathcal{C}$, the map $\eta\mapsto K(\eta s,\eta s^{-1})$ admits exactly two roots in $\overline{\mathcal{C}^+}$ (see Figure~\ref{fig: eta(s)}). In particular,
	\begin{enumerate}[label=\textnormal{(\roman{*})},ref=\textnormal{(\roman{*})}]
		
		\item\label{item: solutions-kernel-p11=0} If $p_{1,1}=0$, one root is identically zero, whereas the other one, denoted by $\eta(s)$, is a continuous function of $s$ on $\mathcal{C}$ and varies as follows: $\eta(s)\in\mathcal{C}^+$ for all $s\in\mathcal{C}\setminus\{\pm 1\}$ and $\eta(1) = -\eta(-1) = 1$. Besides, $\eta(-s)=\overline{\eta(-\overline{s})}=-\eta(s)$ for all $s\in\mathcal{C}$;
		
		\item\label{item: solutions-kernel-p11not=0} If $p_{1,1}\neq 0$, the two roots, denoted by $\eta(s)$ and $\eta_2(s)$, are continuous functions of $s$ on $\mathcal{C}$ and vary in $\mathcal{C}^+$ for all $s\in\mathcal{C}\setminus\{1\}$. Further, $K(\eta,1)$ has two distinct roots in $[-1,1]$, which are $1$ and another root in $(-1,0)$. Let $\eta(s)$ be the root such that $\eta(1)=1$. Besides, $\eta(s)=\overline{\eta(\overline{s})} = -\eta_2(-s)$ for all $|s|=1$.
	\end{enumerate}
\end{lemma}

We refer the readers to the proof of \cite[Lem.\,1 and Lem.\,2]{HoRaTa-20}, in which the arguments are based on Rouch\'e's theorem and can be applied for Lemma~\ref{lem: solutions-mathcal-K}. The only difference is that in \cite{HoRaTa-20}, $\eta(s)$ and $\eta_2(s)$ are real and vary in $[-1,1]$, whereas in the present work, we can only state that they vary in $\overline{\mathcal{C}^+}$.


Notice that if $p_{1,1}=0$, then one can factorize \eqref{eq: K-eta-s} as $K(\eta s,\eta s^{-1}) = \eta\widetilde{K}(\eta,s)$, where
\begin{equation}\label{eq: K-tilde}
	\widetilde{K}(\eta,s) := \eta - \sum p_{k,\ell}\eta^{-k-\ell+1}s^{-k+\ell},
\end{equation}
and $\widetilde{K}(\cdot,s)$ admits $\eta(s)$ as the unique root for any $\vert s\vert =1$. Since $\widetilde{K}(0,s) = - p_{1,0}s - p_{0,1}s^{-1}$, then $\eta(s)$ never meets $0$ if $p_{1,1}=0$, $p_{1,0}\not=p_{0,1}$, whereas $\eta(s)$ meets $0$ only at $s=\pm i$ if $p_{1,1}=0$, $p_{1,0}=p_{0,1}$.

In the case $p_{1,1}\not= 0$, since $(\eta_2(s)s,\eta_2(s)s^{-1})=(\eta(-s)(-s),\eta(-s)(-s)^{-1})$ for all $s\in\mathcal{C}$, then $\mathcal{K}$ can be described by either $\eta(s)$ or $\eta_2(s)$. Moreover, at this step, it is not clear that $\eta(s)$ and $\eta_2(s)$ are distinct roots for all $\vert s\vert =1$. In fact, the assertion holds true if and only if $\partial_\eta K(\eta s,\eta s^{-1})\not= 0$ for all $\vert s\vert = 1$ and $\eta=\eta(s)$, i.e.,
\begin{equation*}
	\sum p_{k,\ell}(-k-\ell)\eta(s)^{-k-\ell+1}s^{-k+\ell}\not= 0,
\end{equation*}
for all $\vert s\vert = 1$.

We can now properly define $\mathcal{K}$ as follows:
\begin{equation}\label{eq: mathcal K description}
	\mathcal{K} = 
	\begin{cases}
	\{ (\eta(s)s,\eta(s)s^{-1}):\vert s\vert=1 \}, &\text{ if } p_{1,1}=0,\, p_{1,0}=p_{0,1},\\
	\{ (\eta(s)s,\eta(s)s^{-1}):\vert s\vert=1 \} \cup \{(0,0)\}, &\text{ if } p_{1,1}=0,\, p_{1,0}\neq p_{0,1},\\
	\{ (\eta(s)s,\eta(s)s^{-1}):\vert s\vert=1 \}, &\text{ if } p_{1,1}\neq 0.
	\end{cases}
\end{equation}
We can also properly describe the projections of the connected subset of $\mathcal{K}$ along the first and second variables:
\begin{equation*}
	\mathcal{S}_1=\{\eta(s)s : \vert s\vert=1\} \quad \text{and}\quad 	\mathcal{S}_2=\{\eta(s)s^{-1}:\vert s\vert=1\}.
\end{equation*}

We now introduce the following assumption:
\begin{enumerate}[label=(K\arabic{*}),ref={\rm (K\arabic{*})}]
		\item\label{assumption: monotonic} $\mathcal{S}_1$ and $\mathcal{S}_2$ are non-self-intersecting; $\eta(s)s$ and $\eta(s)s^{-1}$ move in opposite orientations and admit non-vanishing derivatives for any $s\in\mathcal{C}\setminus\{\pm 1\}$ (resp.~$s\in\mathcal{C}\setminus\{1\}$) in the case $p_{1,1}=0$ (resp.~$p_{1,1}\not= 0$), given that the derivatives exist.
\end{enumerate}
For some examples not satisfying Assumption~\ref{assumption: monotonic}, see Figure~\ref{fig: S1-S2-not-satisfy-K1}.  Assumption \ref{assumption: monotonic} ensures that some characteristics of the curves $\mathcal{S}_1$ and $\mathcal{S}_2$ match those of the symmetric case in \cite{HoRaTa-20}, which suggests that certain of their results and arguments may be reapplied in the present work.

\begin{figure}[t]
	\centering
	\includegraphics[scale=0.5]{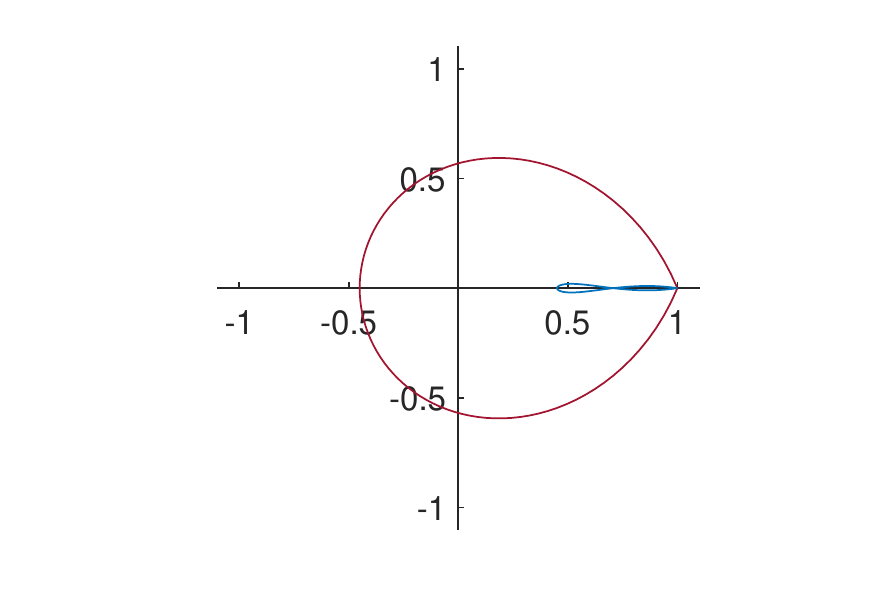}
	\includegraphics[scale=0.5]{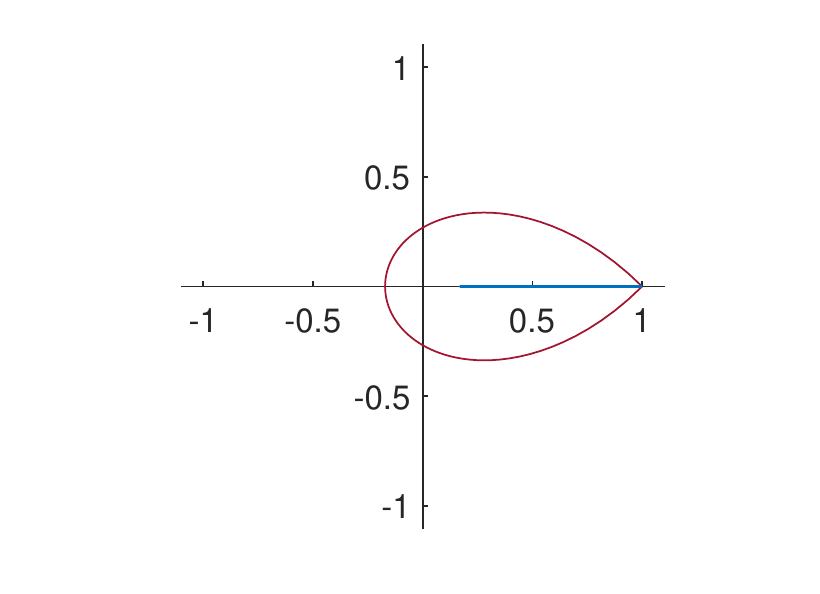}
	\includegraphics[scale=0.5]{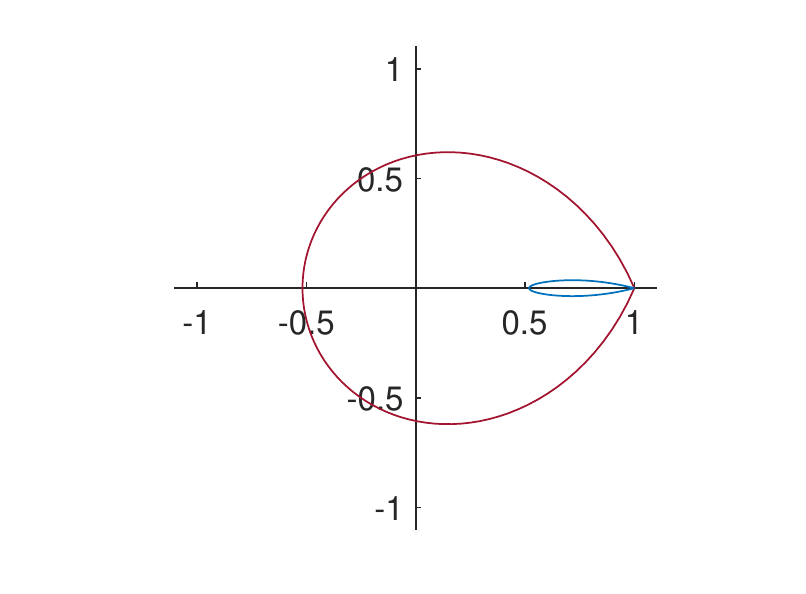}
	\caption{$\mathcal{S}_1$ (red) and $\mathcal{S}_2$ (blue) for some models not satisfying Assumption~\ref{assumption: monotonic}, from left to right: $p_{0,1}/8=p_{1,0}=p_{-1,0}/3=p_{1,-4}/2=1/14$ ($\mathcal{S}_2$ is self-intersecting); $p_{0,1}=p_{0,-1}=p_{-1,1}=p_{1,-1}=1/4$ ($\mathcal{S}_2$ is a segment); $p_{-1,0}=p_{1,-3}=p_{0,1}/3=1/5$ ($\eta(s)s$ and $\eta(s)s^{-1}$ move in the same orientations).}
	\label{fig: S1-S2-not-satisfy-K1}
\end{figure}

Lemma~\ref{lem: S1-S2} below presents some fundamental properties of $\mathcal{S}_1$ and $\mathcal{S}_2$, and emphasizes differences between the subcases of the description \eqref{eq: mathcal K description}.


\begin{figure}[t]
	\centering
	\includegraphics[scale=0.2]{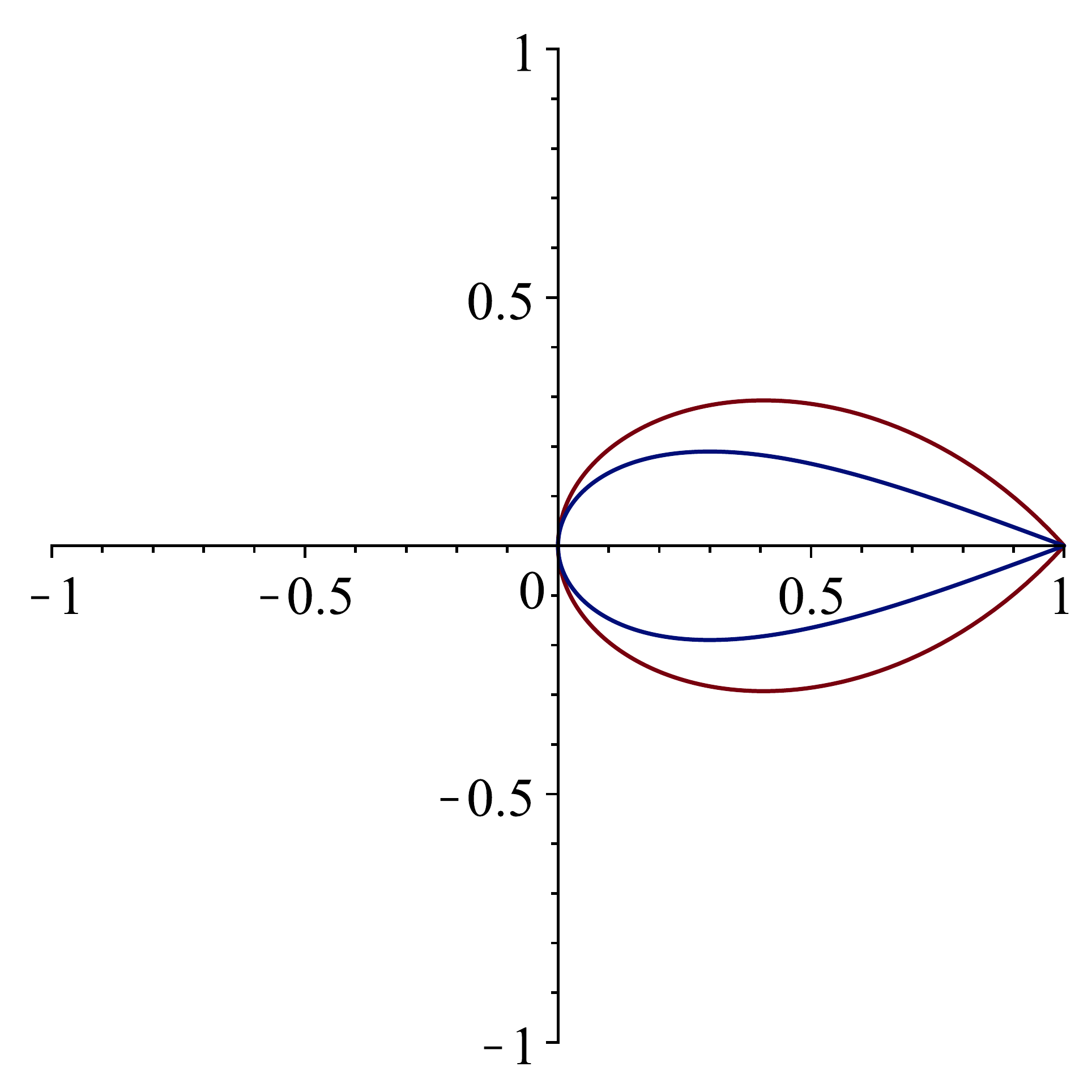}
	\includegraphics[scale=0.2]{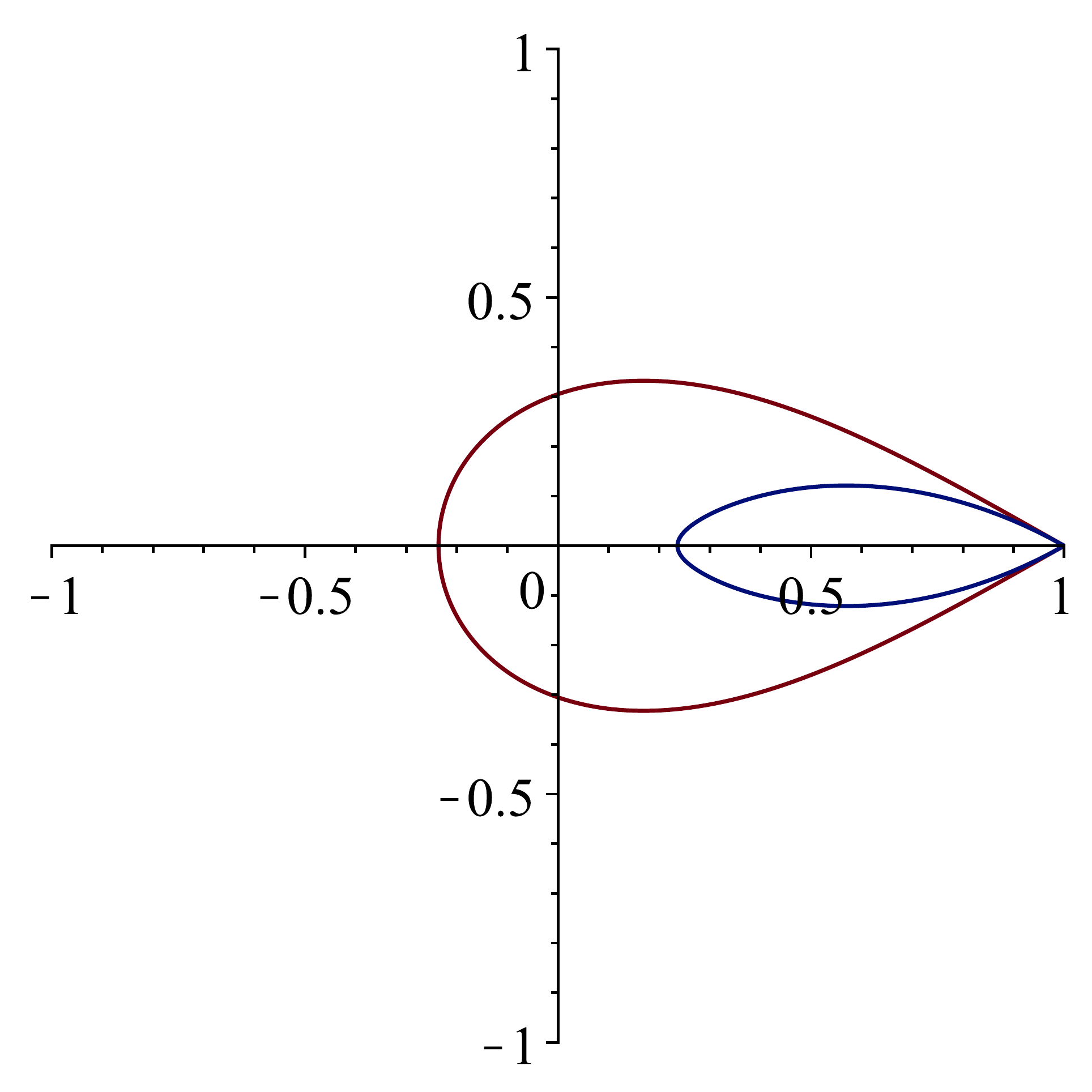}
	\includegraphics[scale=0.2]{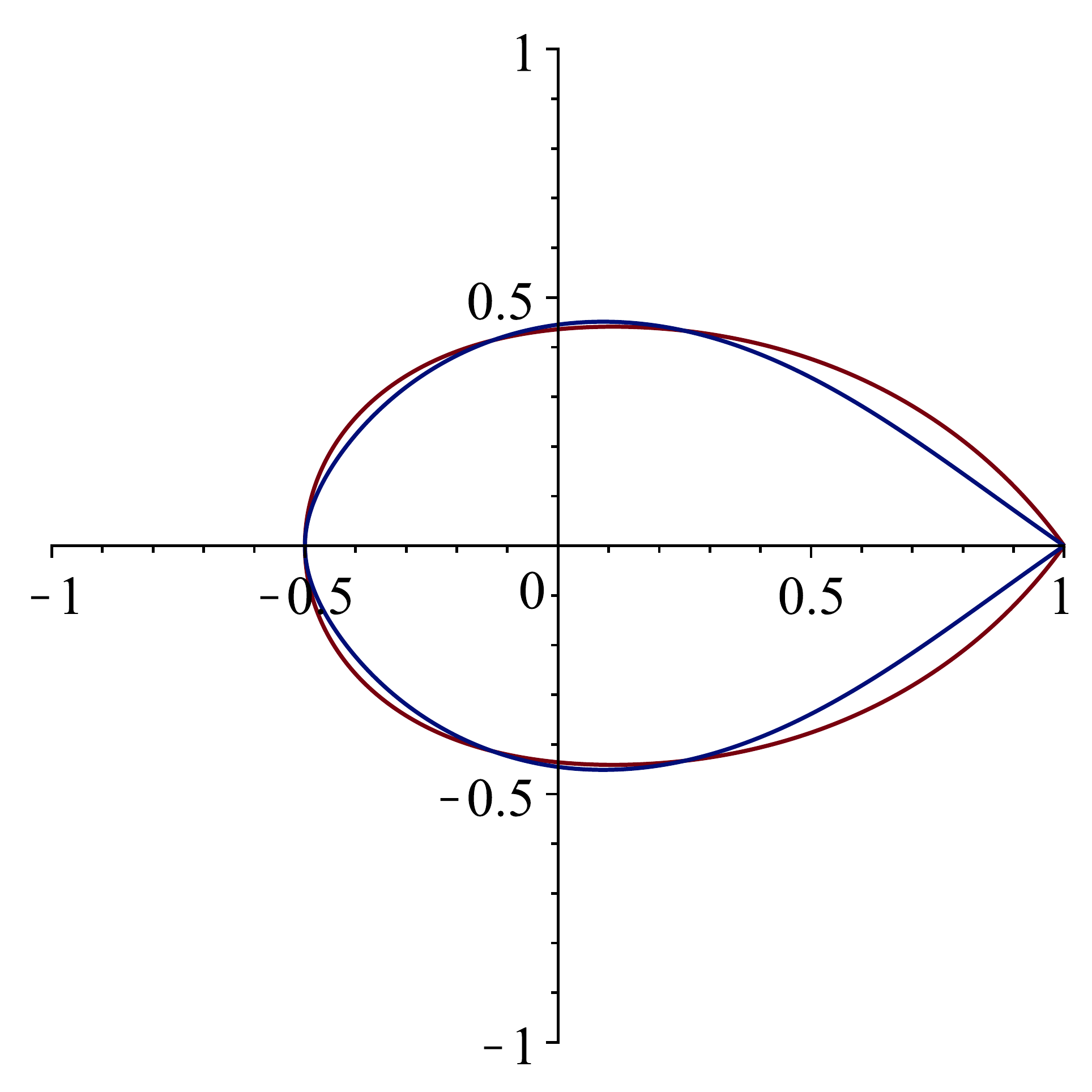}
	\caption{$\mathcal{S}_1$ (red) and $\mathcal{S}_2$ (blue) for some walks, from left to right: $p_{0,1}=p_{1,0}=p_{-1,1}=p_{0,-1}/2=1/5$; $p_{0,1}=p_{-1,0}=p_{1,-1}=1/3$; $p_{1,1}=p_{-1,1}=p_{0,-1}/2=1/4$. The right figure also shows that one of the curve does not necessarily lie inside the other.}
	\label{fig: S1-S2}
\end{figure}


\begin{lemma}\label{lem: S1-S2}
	Under Assumptions~\ref{assumption: H1}--\ref{assumption: H3} and \ref{assumption: monotonic}, the following assertions hold (see Figure~\ref{fig: S1-S2}):
	\begin{enumerate}[label=\textnormal{(\roman{*})},ref=\textnormal{(\roman{*})}]		
		\item\label{item: one-to-one} If $p_{1,1}=0$ (resp.~$p_{1,1}\neq 0$), then the mappings $s\mapsto \eta(s)s$ and $s \mapsto \eta(s)s^{-1}$ are two-to-one (resp.~one-to-one) from $\mathcal{C}$ onto $\mathcal{S}_1$ and $\mathcal{S}_2$;
		
		\item\label{item: S1-S2-symmetric} $\mathcal{S}_1$ and $\mathcal{S}_2$ are symmetric with respect to the real axis;
		
		\item\label{item: direction-S1-S2} If $s$ traverses $\mathcal{C}$ counterclockwise, then $\eta(s)s$ traverses $\mathcal{S}_1$ counterclockwise and $\eta(s)s^{-1}$ traverses $\mathcal{S}_2$ clockwise;
		
		\item\label{item: position-0-S1-S2} The relative position of $0$ with respect to $\mathcal{S}_1$ and $\mathcal{S}_2$ varies in different cases:
		\begin{enumerate}
			\item\label{item: position-of-0-p11=0-p01=p10} If $p_{1,1}=0$, $p_{1,0}=p_{0,1}$, then $0\in\mathcal{S}_1,\,\mathcal{S}_2$;
			
			\item\label{item: position-of-0-p11=0-p01not=p10} If $p_{1,1}=0$, $p_{0,1}>p_{1,0}$ (resp.~$p_{0,1}<p_{1,0}$), then $0\in\mathcal{S}_1^+,\,\mathcal{S}_2^-$ (resp.~$0\in\mathcal{S}_1^-,\,\mathcal{S}_2^+$);
			
			\item\label{item: position-of-0-p11not=0} If $p_{1,1}\neq 0$, then $0\in\mathcal{S}_1^+,\,\mathcal{S}_2^+$.
		\end{enumerate}
	\end{enumerate}
\end{lemma}

\begin{proof}
	By Lemma~\ref{lem: solutions-mathcal-K}, $\eta(e^{i\phi})e^{i\phi}$ and $\eta(e^{i\phi})e^{-i\phi}$ are periodic functions of $\phi$ with period $\pi$ in the case $p_{1,1}=0$ or period $2\pi$ in the case $p_{1,1}\not=0$. Since $\mathcal{S}_1$ and $\mathcal{S}_2$ are non-self-intersecting (by Assumption~\ref{assumption: monotonic}), then these functions are one-to-one mappings from $[0,\pi)$ (resp.~$[0,2\pi)$) onto $\mathcal{S}_1$ and $\mathcal{S}_2$ in the case $p_{1,1}=0$ (resp.~$p_{1,1}\not=0$), which also implies Item~\ref{item: one-to-one}. Further, recall that for all $\vert s\vert=1$, $\eta(s)=-\overline{\eta(-\overline{s})}$ if $p_{1,1}=0$, whereas $\eta(s)=\overline{\eta(\overline{s})}$ if $p_{1,1}\not=0$. This infers Item~\ref{item: S1-S2-symmetric}. Item~\ref{item: direction-S1-S2} follows easily from Assumption~\ref{assumption: monotonic}.
	
	We now prove Item~\ref{item: position-0-S1-S2} by studying the intersection of $\mathcal{S}_1\setminus\{1\}$ (or $\mathcal{S}_2\setminus\{1\}$) and $\mathbb{R}$. In the case $p_{1,1}=0$, this leads us to study the roots of $K(x,-x)$ in $[-1,1]$. One can write $K(x,-x) = -x\widehat{K}(x)$, where $\widehat{K}$ is the series
	\begin{equation}\label{eq: K-hat}
		\widehat{K}(x)=x +\sum p_{k,\ell}(-1)^{-\ell+1}x^{-k-\ell+1}.
	\end{equation}
	Notice that $\widehat{K}'(x)>0$ on $(-1,1)$, and by \ref{assumption: H2} one has $\widehat{K}(-1)=-1+\sum p_{k,\ell}(-1)^{-k}<0$, $\widehat{K}(0)=p_{0,1}-p_{1,0}$, $\widehat{K}(1) = 1 + \sum p_{k,\ell}(-1)^{-\ell+1}>0$. Hence, if $p_{0,1}=p_{1,0}$, then $0$ is the unique root of $\widehat{K}(x)$ on $[-1,1]$. Thus $\mathcal{S}_1$ and $\mathcal{S}_2$ pass through $0$, which implies Item~\ref{item: position-of-0-p11=0-p01=p10}; If $p_{0,1}>p_{1,0}$, then $\widehat{K}(x)$ has a root in $(-1,0)$, i.e., $\mathcal{S}_1$ (resp.~$\mathcal{S}_2$) passes through a negative (resp.~positive) point; If $p_{0,1}<p_{1,0}$, then $\widehat{K}(x)$ has a root in $(0,1)$, i.e., $\mathcal{S}_1$ (resp.~$\mathcal{S}_2$) passes through a positive (resp.~negative) point. This implies Item~\ref{item: position-of-0-p11=0-p01not=p10}.
	
	We move to the case $p_{1,1}\not =0$ and study the root of $K(x,x)$ in $[-1,0]$. Since $K(-1,-1)=1-\sum p_{k,\ell}(-1)^{-k-\ell}>0$ and $K(0,0) = -p_{1,1}<0$, then $K(x,x)$ has a root in $(-1,0)$, and $\mathcal{S}_1$ and $\mathcal{S}_2$ pass through a same negative point, which implies Item~\ref{item: position-of-0-p11not=0}.
\end{proof}


The next lemma concerns the smoothness of $\mathcal{S}_1$ and $\mathcal{S}_2$, particularly their shapes at $1$. Let us denote by $V[K]$ the zero set of the kernel on the closed bidisk:
\begin{equation*}
	V[K] := \{(x,y)\in\mathbb{C}^2:K(x,y)=0,\vert x\vert\leq 1,\vert y\vert \leq 1\}.
\end{equation*}
Let us recall some key definitions from complex analysis. Firstly, a curve parametrized by a function $z:[a,b]\to\mathbb{C}$ is smooth if $z' (t)$ exists, is continuous and does not vanish on $[a,b]$. Secondly, a singular point of $V[K]$ is defined as a solution of the system $\partial_x K(x,y) = \partial_y K(x,y) = K(x,y) = 0$. Otherwise, any point of $V[K]$ is called non-singular or regular. Finally, the angle at a corner point formed by the two tangents will be simply referred to as the angle (of the curve) at that point.

\begin{lemma}\label{lem: S1-S2-smoothness-angle-at-1}
	Under Assumptions~\ref{assumption: H1}--\ref{assumption: H3} and \ref{assumption: monotonic}, the curves $\mathcal{S}_1$ and $\mathcal{S}_2$ are smooth everywhere except at $1$, where they admit corner points (see Figure~\ref{fig: S1-S2}) with the angles
	\begin{equation}\label{eq: theta_1 theta_2}
		\theta_1 = \arccos \frac{AC-A^2-2B^2+2AB}{A(A-2B+C)}\quad\text{and}\quad
		\theta_2 = \arccos \frac{AC-A^2-2B^2-2AB}{A(A+2B+C)},
	\end{equation}
	where $A = \sum p_{k,\ell}(k+\ell)^2$, $B = \sum p_{k,\ell}(k^2-\ell^2)$ and $C=\sum p_{k,\ell}(k-\ell)^2$.
	Moreover, $\mathcal{K}\setminus\{(1,1)\}$ consists of non-singular points of $V[K]$.
\end{lemma}

\begin{proof}
	We first study the differentiability of $\eta(s)$. By Assumption~\ref{assumption: H3}, $\partial_x K$ and $\partial_y K$ are well defined for all $\vert x\vert, \vert y\vert < 1$, then $\partial_\eta K(\eta s,\eta s^{-1})$ and $\partial_s K(\eta s,\eta s^{-1})$ are also well defined for all $\vert\eta\vert< 1$, $\vert s\vert =1$.
	
	In the case $p_{1,1}=0$, recall that $K(\eta s,\eta s^{-1}) = \eta\widetilde{K}(\eta,s)$, where $\widetilde{K}$ is defined in \eqref{eq: K-tilde}. By Lemma~\ref{lem: solutions-mathcal-K}, for any $s_0\in\mathcal{C}\setminus\{\pm 1\}$, $\eta(s_0)$ is the unique root of $\widetilde{K}(\eta,s_0)$ in $\mathcal{C}^+$, i.e., there exists a continuous function $u(\eta,s)$ defined in a neighborhood of $(\eta(s_0),s_0)$ such that $\widetilde{K}(\eta,s) = (\eta-\eta(s_0))u(\eta,s)$ and $u(\eta(s),s)\not=0$. This is equivalent to $\partial_\eta \widetilde{K}(\eta(s_0),s_0)=u(\eta(s_0),s_0)\not= 0$. Thus, $\eta'(s)$ is well defined for all $s\in\mathcal{C}\setminus\{\pm 1\}$ under the form
	\begin{equation}\label{eq: eta'(s)}
		\eta'(s)=-\frac{\partial_s\widetilde{K}(\eta(s),s)}{\partial_\eta\widetilde{K}(\eta(s),s)},
	\end{equation}
	which is obtained by differentiating the identity $\widetilde{K}(\eta(s),s)=0$. Since $(\eta(s)s)'$ and $(\eta(s)s^{-1})'$ are non-vanishing on $\mathcal{C}\setminus\{\pm 1\}$ (by Assumption~\ref{assumption: monotonic}), then $\mathcal{S}_1$ and $\mathcal{S}_2$ are smooth everywhere except at $1$.
	
	In the case $p_{1,1}\not= 0$, notice that for all $s\in\mathcal{C}$, $\eta(s)$ and $\eta_2(s)$ are distinct roots of $K(\eta s,\eta s^{-1})$. Indeed, if there exists $s_0\in\mathcal{C}$ such that $\eta(s_0)=\eta_2(s_0)$, then Lemma~\ref{lem: solutions-mathcal-K} implies $\eta(s_0)s_0 = \eta(-s_0)(-s_0)$, i.e., $\mathcal{S}_1$ is self-intersecting, which contradicts Assumption~\ref{assumption: monotonic}. Hence, for any $s\in\mathcal{C}\setminus\{1\}$, $\eta(s)$ is the unique root of $K(\eta s,\eta s^{-1})$ in a neighborhood of $\eta(s)$, which is equivalent to $\partial\eta K(\eta(s)s,\eta(s)s^{-1})\not= 0$ by similar arguments as in the case $p_{1,1}=0$. Thus, $\eta' (s)$ is well defined for all $s\in\mathcal{C}\setminus\{1\}$. By Assumption~\ref{assumption: monotonic}, then $\mathcal{S}_1$ and $\mathcal{S}_2$ are smooth everywhere except at $1$.
	
	We now prove that $\mathcal{K}\setminus\{(1,1)\}$ consists of non-singular points of $V[K]$. Consider first the point $(0,0)\in\mathcal{K}$, which can only happen in the case $p_{1,1}=0$. Since $\partial_x K(0,0)$ and $\partial_y K(0,0)$, which are respectively equal to $p_{0,1}$ and $p_{1,0}$, cannot simultaneously vanish, then $(0,0)$ is a non-singular point of $V[K]$. Now move to the other points of $\mathcal{K}\setminus\{(1,1)\}$ and let $(\eta,s)$ be a root of $K(\eta s,\eta s^{-1})$  such that $\eta\in\overline{\mathcal{C}}\setminus\{0\}$ and $s\in\mathcal{C}$, then as shown in the beginning of the proof, $\partial_\eta K(\eta,s)\neq 0$. Since
	\begin{equation*}
		\partial_\eta K = \partial_\eta x \partial_x K + \partial_\eta y \partial_y K = s \partial_x K + s^{-1} \partial_y K,
	\end{equation*}
	we must have either $\partial_x K(\eta,s)$ or $\partial_y K(\eta,s)$ non-zero. The claim then follows.
	
	We now investigate the shapes of $\mathcal{S}_1$ and $\mathcal{S}_2$ at $1$. In the case $p_{1,1}=0$, notice that one cannot evaluate the limit of $\eta'(s)$ as $s\to 1$ directly from \eqref{eq: eta'(s)} since the quotient is undefined at $1$. We then differentiate the identity $\widetilde{K}(\eta(s),s)=0$ twice and evaluate its limit as $s\to 1$. The limit $\lim_{s\to 1}\eta'(s)$ satisfies the equation
	\begin{equation}\label{eq: eqn of eta'(1)}
		AX^2 +2BX+ C = 0,
	\end{equation}
	where $A$, $B$ and $C$ are defined in Lemma \ref{lem: S1-S2-smoothness-angle-at-1} and are finite by Assumption~\ref{assumption: H3.2}. Since the above equation has two distinct roots, which indicates that $\eta(s)$ is semi-differentiable at $s=1$ and admits left and right derivatives	
	\begin{equation*}
		\partial_\pm\eta(1):= \lim_{\substack{s= \exp(i\phi)\\\phi\to 0^\pm}}\frac{\eta(s)-1}{s-1}\in \left\{\frac{-B\pm{i}\sqrt{AC-B^2}}{A}\right\}.
	\end{equation*}
	In order to know which root of \eqref{eq: eqn of eta'(1)} corresponds to $\partial_+\eta(1)$, we take a look at the following form:
	\begin{equation*}
		\partial_+(\eta(s)s)\big|_{s=1} = \partial_+\eta(1) + \eta(1) \in \left\{ \frac{A-B\pm{i}\sqrt{AC-B^2}}{A} \right\}.
	\end{equation*}
	Since $\arg \partial_+(\eta(s)s)\big|_{s=1}$ indicates the oriented angle between the tangent of $\mathcal{S}_1$ and $\mathcal{C}$ at $1$, which is non-negative and smaller than or equal to $\pi/2$ by Assumption~\ref{assumption: monotonic}, then
	\begin{equation*}
		\partial_+(\eta(s)s)\big|_{s=1} = \frac{A-B+{i}\sqrt{AC-B^2}}{A}\quad\text{and}\quad \partial_-(\eta(s)s)\big|_{s=1} = \frac{A-B-{i}\sqrt{AC-B^2}}{A}.
	\end{equation*}
	The angle $\theta_1$ of $\mathcal{S}_1$ at $1$ is implied from
	\begin{equation*}
		\cos\theta_1 = 2\cos^2\left(\frac{\theta_1}{2}\right) -1 = 2\frac{AC-B^2}{(A-B)^2+(AC-B^2)}-1.
	\end{equation*}
		Similarly, we have:
	\begin{align*}
		&\partial_\pm(\eta(s)s^{-1})\big|_{s=1} = \frac{-A-B\pm{i}\sqrt{AC-B^2}}{A},\\
		&\cos\theta_2 = 2\cos^2\left(\frac{\theta_2}{2}\right) -1 = 2\frac{AC-B^2}{(A+B)^2+(AC-B^2)}-1.
	\end{align*}
	In the case $p_{1,1}\not= 0$, by differentiating the identity $K(\eta(s)s,\eta(s)s^{-1})=0$ twice and evaluating its limit as $s\to 1$, we also obtain Eq.~\eqref{eq: eqn of eta'(1)}. The rest of the proof then follows the case $p_{1,1}=0$.
\end{proof}


The following remark shows the relation between the angles $\theta_1$, $\theta_2$ in Lemma~\ref{lem: S1-S2-smoothness-angle-at-1} and $\theta$ introduced in \eqref{eq: theta}. We also show how the weights $\{p_{k,\ell}\}_{k,\ell}$ impact upon $\theta_1$ and $\theta_2$.

\begin{remark}\label{rmk: angles theta1, theta2}
	The angle $\theta$ defined in \eqref{eq: theta} is the arithmetic mean of $\theta_1$ and $\theta_2$ in \eqref{eq: theta_1 theta_2}. Further, under Assumption~\ref{assumption: monotonic}, we have the following trichotomy:
	\begin{enumerate}[label=\textnormal{(\roman{*})},ref=\textnormal{(\roman{*})}]
		\item $\theta_1,\,\theta_2\not= 0$ (see Figure~\ref{fig: S1-S2}) if
		\begin{equation*}
			\begin{cases}
				-\sum p_{k,\ell}k\ell < \sum p_{k,\ell}k^2,\\
				-\sum p_{k,\ell}k\ell < \sum p_{k,\ell}\ell^2;
			\end{cases}
		\end{equation*}
		
		\item $\theta_1\not= 0$, $\theta_2=0$ (i.e., $\mathcal{S}_2$ has a cusp at $1$, see Figure~\ref{fig: theta_2 = 0}) if
		\begin{equation*}
			\sum p_{k,\ell}k^2 = -\sum p_{k,\ell}k\ell <\sum p_{k,\ell}\ell^2;
		\end{equation*}
		
		\item $\theta_1=0$ (i.e., $\mathcal{S}_1$ has a cusp at $1$), $\theta_2\not= 0$ if
		\begin{equation*}
			\sum p_{k,\ell}\ell^2 = -\sum p_{k,\ell}k\ell <\sum p_{k,\ell}k^2.
		\end{equation*}
	\end{enumerate}
	The above statements may not hold true if the model does not satisfy Assumption~\ref{assumption: monotonic}. For example, the model $p_{-1,0}=p_{1,-3}=p_{0,1}/3=1/5$ (see Figure~\ref{fig: S1-S2-not-satisfy-K1}) does not satisfy any of the statements.
\end{remark}

\begin{proof}
	We have:
	\begin{align*}
		\cos\frac{\theta_1+\theta_2}{2} &= \cos\frac{\theta_1}{2} \cos\frac{\theta_2}{2} - \sin\frac{\theta_1}{2} \sin\frac{\theta_2}{2}\\
		&=\frac{(AC-B^2)-(A-B)(A+B)}{\sqrt{(A^2-2AB+AC)(A^2+2AB+AC)}}\\
		&=\frac{C-A}{\sqrt{(A-2B+C)(A+2B+C)}}=\cos \theta,
	\end{align*}
	which implies that $\theta=(\theta_1+\theta_2)/2$.
	
	Moving to the proof of the trichotomy, we first recall from Lemma~\ref{lem: S1-S2}\ref{item: direction-S1-S2} that $\eta(s)s$ and $\eta(s)s^{-1}$ move respectively clockwise and counterclockwise as $s$ moves clockwise on $\mathcal{C}$ under Assumption~\ref{assumption: monotonic}. This implies that
	\begin{equation*}
		\Re \left( \partial_+(\eta(s)s)\big|_{s=1} \right) \geq 0 \quad\text{and}=\frac{A-B}{A}\quad \Re \left( \partial_+(\eta(s)s^{-1})\big|_{s=1} \right)\frac{-A-B}{A} \leq 0,
	\end{equation*}
	where $\Re$ denotes the real part of a complex number. Hence, Assumption~\ref{assumption: monotonic} implies $-A\leq B\leq A$. The case $A=B=0$ cannot hold. Otherwise, we have: $\theta_1,\theta_2\not=0$ if $-A<B<A$; $\theta_1\not=0$, $\theta_2=0$ if $-A<B=A$; $\theta_1=0$, $\theta_2\not=0$ if $-A=B<A$.
\end{proof}

\begin{figure}[t]
	\centering
	\includegraphics[scale=0.75]{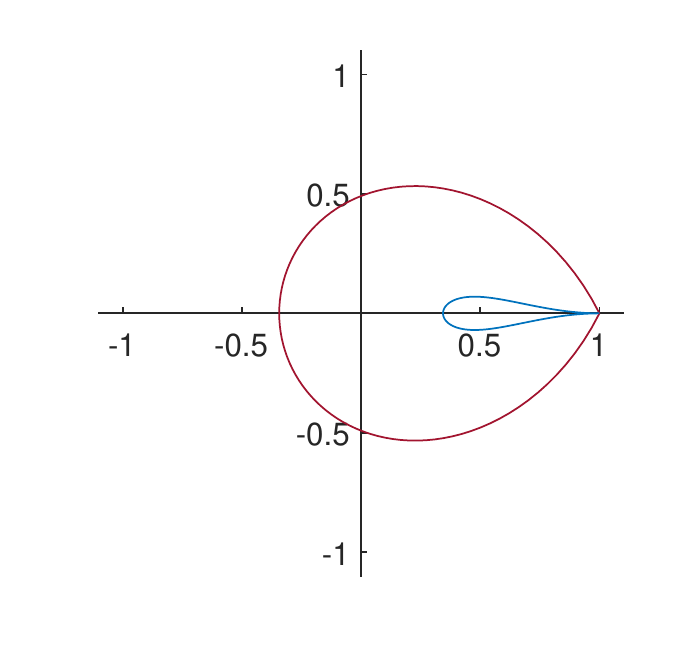}
	\caption{$\mathcal{S}_1$ (red) and $\mathcal{S}_2$ (blue) for the walk $p_{0,1}=1/2$, $p_{-1,0}=1/4$, $p_{1,0}=p_{1,-4}=1/8$. Such a model has $\sum p_{k,\ell}k^2 = -\sum p_{k,\ell}k\ell (=1/2)<\sum p_{k,\ell}\ell^2(=1)$, and thus the angle $\theta_2=0$.}
	\label{fig: theta_2 = 0}
\end{figure}

\subsection{Solutions of the kernel in a neighborhood of $\mathcal{K}$}\label{subsec: neighborhood of mathcal K}

We first properly define the functions $X$ and $Y$ mentioned in Proposition~\ref{prop: extend-solutions-intro}. Let $X:\mathcal{S}_2\to\mathcal{S}_1$ be the mapping $s\mapsto \eta(s)s$ composed with the compositional inverse of $s\mapsto\eta(s)s^{-1}$. Similarly, $Y:\mathcal{S}_1\to\mathcal{S}_2$ denotes the mapping $s\mapsto \eta(s)s^{-1}$ composed with the inverse of $s\mapsto\eta(s)s$. In other words, $Y:\mathcal{S}_1\to\mathcal{S}_2$ is the inverse of $X:\mathcal{S}_2\to\mathcal{S}_1$. The following lemma presents some crucial properties of these functions.

\begin{lemma}\label{lem: X(y)-Y(x)-on-S1-S2}
	Under Assumptions~\ref{assumption: H1}--\ref{assumption: H3} and \ref{assumption: monotonic}, we have:
	\begin{enumerate}[label=\textnormal{(\roman{*})},ref=\textnormal{(\roman{*})}]
		\item\label{item: X(y)-one-to-one} $X(y)$ is a one-to-one mapping from $\mathcal{S}_2$ onto $\mathcal{S}_1$;
		
		\item\label{item: X(y)-solution-kernel} For all $y\in\mathcal{S}_2$, $(X(y),y)\in\mathcal{K}$;
		
		\item\label{item: X(y)-derivative} $X(y)$ is admits non-vanishing derivatives on $\mathcal{S}_2\setminus\{1\}$.
		
		\item\label{item: X(y)-derivative-at-1} $X(y)$ admits non-vanishing left and right derivatives at $1$.
	\end{enumerate}
	We have analogous statements for $Y(x)$.
\end{lemma}

\begin{proof}
	Since $s\mapsto \eta(s)s$ and $s\mapsto \eta(s)s^{-1}$ are one-to-one mappings from the half-circle $\{{e}^{{i}\phi}:\phi\in[0,\pi)\}$ (resp.~$\mathcal{C}$) onto $\mathcal{S}_1$ and $\mathcal{S}_2$ in the case $p_{1,1}=0$ (resp.~$p_{1,1}\neq 0$) (see Lemma~\ref{lem: S1-S2}), then Item~\ref{item: X(y)-one-to-one} follows. By the construction of $X(y)$, for all $y\in\mathcal{S}_2$, there exists $s\in\mathcal{C}$ such that $(X(y),y)=(\eta(s)s,\eta(s)s^{-1})\in\mathcal{K}$, which implies Item~\ref{item: X(y)-solution-kernel}. Moreover, $\eta(s)s$ and $\eta(s)s^{-1}$ admit non-vanishing derivatives on $\mathcal{C}\setminus\{\pm 1\}$ (resp.~$\mathcal{C}\setminus\{1\}$) and admit non-vanishing left/right derivatives at $\pm 1$ (resp.~$1$) in the case $p_{1,1}=0$ (resp.~$p_{1,1}\neq 0$) (see Lemma~\ref{lem: S1-S2-smoothness-angle-at-1}), then Items~\ref{item: X(y)-derivative} and \ref{item: X(y)-derivative-at-1} follow.
\end{proof}

We formulate the next lemma to extend the domains of definition of $X(y)$ and $Y(x)$.
\begin{lemma}\label{lem: solution-neigborhood-S2}
	Under Assumptions~\ref{assumption: H1}--\ref{assumption: H3} and \ref{assumption: monotonic}, $X(y)$ can be continued analytically on a neighborhood $V\subset\mathbb{C}\setminus\{1\}$ of $\mathcal{S}_2\setminus\{1\}$ such that:
	\begin{enumerate}[label=\textnormal{(\roman{*})},ref=\textnormal{(\roman{*})}]
		\item $K(X(y),y) = 0$ for all $y\in V$;
		\item $X(y)\in\mathcal{S}_1^-$ and $\vert X(y)\vert > \vert y\vert$ for all $y\in V\cap\mathcal{S}_2^+$;
		\item $X(y)\in\mathcal{S}_1^+$ and $\vert X(y)\vert < \vert y\vert$ for all $y\in V\cap\mathcal{S}_2^-$.
	\end{enumerate}
	We have an analogous statement for $Y(x)$.
\end{lemma}

\begin{proof}
	We only prove the lemma for $X(y)$. The result for $Y(x)$ is deduced similarly. For any $y\in\mathcal{S}_2\setminus\{1\}$, since $(X(y),y)$ is a non-singular point of $V[K]$ (Lemma~\ref{lem: S1-S2-smoothness-angle-at-1}), $X'(y)\not= 0$ (Lemma~\ref{lem: X(y)-Y(x)-on-S1-S2}), and $\partial_x K(X(y),y) X'(y)+\partial_y K(X(y),y) =0$ (by differentiating $K(X(y),y)=0$), then $\partial_xK(X(y),y)$ and $\partial_y K(X(y),y)$ are simultaneously non-vanishing. Thus, by the implicit function theorem (see \cite[Sec.\,B4]{GuRo-65}), $X(y)$ can be extended as an analytic bijection in a neighborhood $V\subset\mathbb{C}\setminus\{1\}$ small enough of $\mathcal{S}_2\setminus\{1\}$ such that $K(X(y),y)=0$. We remark that with $V$ small enough, $\mathcal{S}_2$ divides $V$ into two halves, on which the sign of $(\vert X(y)\vert - \vert y\vert)$ does not change and $X(y)$ does not pass through $\mathcal{S}_1$. Hence, it suffices to prove the assertions in the lemma for some $y$ close to the intersection of $\mathcal{S}_2\setminus\{1\}$ and $\mathbb{R}$.
	
	Consider first the case $p_{1,1}=0$, $p_{0,1}=p_{1,0}$, where $0$ is the intersection of $\mathcal{S}_2$ and $\mathbb{R}$. Differentiating once and twice the identity $K(X(y),y)=0$ at $y=0$, one has:
	\begin{align*}
	& X'(0) = -\frac{\partial_yK(0,0)}{\partial_xK(0,0)} = -1,\\
	&X''(0) = -\frac{\partial_{xx}K(0,0)X'(0)^2+2\partial_{xy}K(0,0)X'(0)+\partial_{yy}K(0,0)}{\partial_xK(0,0)} = -\frac{2(1+p_{-1,1}+p_{1,-1})}{p_{0,1}}.
	\end{align*}
	We then obtain the asymptotics of $X(y)$ around $0$:
	\begin{equation*}
	X(y) = -y - \frac{1+p_{-1,1}+p_{1,-1}}{p_{0,1}}y^2 + o(y^2).
	\end{equation*}
	We further point out that any high order derivative of $X(y)$ at $0$ also admits a real value, since it is formed by high order partial derivatives of $K(x,y)$ at $(0,0)$. This indicates that $X(y)$ is real on a neighborhood $V_0\subset\mathbb{R}$ of $0$. Hence, for all $y>0$ close to $0$ (that is, $y\in\mathcal{S}_2^+$), one has $X(y)<X(0)$ (that is, $X(y)\in\mathcal{S}_1^-$) and $\vert X(y)\vert >\vert y\vert $. On the other hand, for all $y<0$ close to $0$, $X(y)>X(0)$ and $\vert X(y)\vert <\vert y\vert $. The lemma is then proven in this case.
	
	We move to the other cases and let $y_0$ be the intersection of $\mathcal{S}_2\setminus\{1\}$ and $\mathbb{R}$. In the case $p_{1,1}=0$, $p_{0,1}>p_{1,0}$, one has $X(y_0)=-y_0<0$ and the asymptotics of $X(y)$ around $y_0$:
	\begin{equation*}
	X(y) = -y_0 + X'(y_0)(y-y_0) + o(y-y_0),
	\end{equation*}
	where
	\begin{equation*}
	X'(y_0) = -\frac{\partial_yK(-y_0,y_0)}{\partial_xK(-y_0,y_0)},
	\end{equation*}
	obtained by differentiating $K(X(y),y)=0$ at $y=y_0$. We now prove that $X'(y_0)<-1$. One has:
	\begin{equation*}
	\partial_y K(-y_0,y_0) = -p_{1,0} - y_0 \left( 1+ \sum_{(k,\ell)\neq (1,0)} p_{k,\ell}(-\ell+1) (-1)^{-k+1} y_0^{-k-\ell} \right).
	\end{equation*}
	Since
	\begin{equation*}
	\left| \sum_{(k,\ell)\neq (1,0)} p_{k,\ell}(-\ell+1) (-1)^{-k+1} y_0^{-k-\ell} \right| < \sum_{(k,\ell)\neq (1,0)} p_{k,\ell}(-\ell+1) = 1-p_{1,0} \leq 1,
	\end{equation*}
	then $\partial_y K(-y_0,y_0) <0$. We now have:
	\begin{align*}
	\partial_x K(-y_0,y_0) - \partial_y K(-y_0,y_0)
	&=\left(y_0- \sum p_{k,\ell}(-k-\ell+1)(-1)^{-k}y_0^{-k-\ell+1}\right) - \frac{K(-y_0,y_0)}{y_0}\\
	&= y_0\left(1 - \sum p_{k,\ell}(-k-\ell+1)(-1)^{-k}y_0^{-k-\ell}\right).
	\end{align*}
	Since
	\begin{equation*}
	\left| \sum p_{k,\ell}(-k-\ell+1) (-1)^{-k} y_0^{-k-\ell} \right| < \sum p_{k,\ell}(-k-\ell+1) = 1,
	\end{equation*}
	then $\partial_x K(-y_0,y_0) - \partial_y K(-y_0,y_0)>0$. Now notice that by Lemma~\ref{lem: S1-S2}\ref{item: direction-S1-S2}, one has $X'(y_0)<0$. Thus, $\partial_y K(-y_0,y_0) < \partial_x K(-y_0,y_0) <0$ and $X'(y_0)<-1$. We also notice that $X(y)$ is real for any $y\in\mathbb{R}$ close enough to $y_0$, since the Taylor expansion of $X(y)$ at $y_0$ admits all real coefficients. This implies that for all $y>y_0$ close to $y_0$, one has $X(y)<X(y_0)$ and $\vert X(y)\vert >\vert y\vert$. Similarly, for all $y<y_0$ close to $y_0$, then $X(y)>X(y_0)$ and $\vert X(y)\vert < \vert y\vert$. The lemma is proven for the case $p_{1,1}=0$, $p_{0,1}>p_{1,0}$.
	
	Now in the case $p_{1,1}=0$, $p_{0,1}<p_{1,0}$, one has $X(y_0)=-y_0>0$ and the asymptotics of $X(y)$ around $y_0$:
	\begin{equation*}
	X(y) = -y_0 + X'(y_0)(y-y_0) + o(y-y_0),
	\end{equation*}
	where $X'(y_0)=1/Y'(-y_0)\in (-1,0)$. The rest of the proof in this case follows the case $p_{1,1}=0$, $p_{0,1}>p_{1,0}$.
	
	In the case $p_{1,1}\not=0$, it is seen that $X(y_0)=y_0<0$ and as $y$ is close to $y_0$,
	\begin{equation*}
		X(y) = y_0 + X'(y_0)(y-y_0) + o(y-y_0),
	\end{equation*}
	where $X'(y_0)<0$ by Lemma~\ref{lem: S1-S2}\ref{item: direction-S1-S2}. Moreover, $X(y)$ is real for any $y\in\mathbb{R}$ close enough to $y_0$, since the Taylor expansion of $X(y)$ at $y_0$ admits all real coefficients. Hence, for all $y>y_0$ close to $y_0$, one has $X(y)<X(y_0)$ and $\vert X(y)\vert >\vert y\vert$. Similarly, for all $y<y_0$ close to $y_0$, then $X(y)>X(y_0)$ and $\vert X(y)\vert < \vert y\vert$. The proof is then complete.
\end{proof}

\begin{lemma}\label{lem: X(y) around 0}
	Under Assumptions~\ref{assumption: H1}--\ref{assumption: H3} and \ref{assumption: monotonic}, in the case $p_{1,1}=0$, $p_{0,1}>p_{1,0}$, where $0$ lies in $\mathcal{S}_1^+$ and $\mathcal{S}_2^-$, the roots of $K(x,y)$ around $(0,0)$ can be represented by an analytic continuation of $X(y)$ along a path in $\overline{\mathcal{S}_2^-}\cap\mathcal{C}^+$ under the form $(X(y),y)$. In particular for $y$ close to $0$,
	\begin{equation}\label{eq: X(y) near 0}
	X(y) = -\frac{p_{1,0}}{p_{0,1}}y-\frac{p_{1,0}/p_{0,1}+p_{-1,1}+p_{1,-1}}{p_{0,1}}y^2 +o(y^2).
	\end{equation}
\end{lemma}

The statement in Lemma~\ref{lem: X(y) around 0} may not hold for $Y(x)$ in the case $p_{1,1}=0$, $p_{0,1}>p_{1,0}$,  since $p_{1,0}$ may be equal to $0$, and $x=0$ then becomes a branch point of $Y(x)$.

\begin{proof}
	Construct a path $\gamma:[0,1]\to \overline{\mathcal{S}_2^-}\cap\mathcal{C}^+$ starting at $\gamma(0)\in\mathcal{S}_2\setminus\{1\}$ and ending at $\gamma(1)=0$. Further, $\gamma((0,1))$ is included in $\mathcal{S}_2^+$ and avoids any point in $V[K]$ such that $\partial_x K=0$. Then, $X(y)$ can be analytically extended along $\gamma$. By Lemma~\ref{lem: solution-neigborhood-S2}, $\vert X(y)\vert<\vert y\vert $ for any $y\in\gamma$ close enough to $\gamma(0)$. Since $\gamma((0,1))$ does not pass through $\mathcal{S}_2$, then we also have $\vert X(y)\vert<\vert y\vert $ for all $y\in\gamma((0,1))$. Thus, $X(y)$ converges to $0$ as $y$ tends to $0$ along $\gamma$. Further, since $\partial_x K(0,0)=-p_{0,1}\not= 0$, then $y=0$ is not a branch point of $X(y)$.

	Now differentiating the identity $K(X(y),y)=0$ once and twice at $y=0$, we have:
	\begin{equation*}
	X'(0) = -\frac{p_{1,0}}{p_{0,1}}\quad \text{and}\quad
	X'' (0) = -\frac{2(p_{1,0}/p_{0,1}+p_{-1,1}+p_{1,-1})}{p_{0,1}},
	\end{equation*}
	and \eqref{eq: X(y) near 0} then follows.
\end{proof}


\subsection{Proof of Proposition~\ref{prop: extend-solutions-intro}}\label{subsec: proof-Prop_extend-solutions}

The proof scheme relies heavily on analytic continuations along paths and comparison between the modulus of two coordinates of the kernel's solutions.

\begin{proof}[Proof of Proposition~\ref{prop: extend-solutions-intro}\ref{item: Prop1-item1}]
	Consider first the case $p_{1,1}=0$, $p_{0,1}=p_{1,0}$, in which $0\in\mathcal{S}_1,\,\mathcal{S}_2$. Construct a closed path $\gamma:[0,1]\to \overline{\mathcal{S}_2^-}\cap\mathcal{C}^+$ such that $\gamma(0)=\gamma(1)=0$, $\gamma\big((0,1)\big)\in\mathcal{S}_2^-$. We further assume that $\gamma((0,1))$ stays in $\mathcal{S}_2^-$ and avoids any branch point of $X(y)$. Extend $X(y)$ along $\gamma$, then by Lemma~\ref{lem: solution-neigborhood-S2}, $\vert X(y)\vert < \vert y\vert$ for any $y\in \gamma\big((0,1)\big)$ close enough to $\gamma(0)$. Since $\gamma((0,1))$ never passes through $\mathcal{S}_2$, then we also have $X(y)\vert < \vert y\vert$ for all $y\in \gamma\big((0,1)\big)$. Hence, as $y$ tends to $\gamma(1)$ along $\gamma$, $X(y)$ also returns to $0$, which implies that $X(y)$ is analytic in a neighborhood of $\gamma$. The extension of $X(y)$ on the interior of $\gamma$ can be then expressed by Cauchy's integral formula. Further, Morera's theorem implies that $X(y)$ is analytic in the interior of $\gamma$, since $\oint_\gamma X(y)dy = 0$. Therefore, $X(y)$ does not have any branch point in $\mathcal{S}_2^-\cap \mathcal{C}$ and similarly for $Y(x)$ on $\mathcal{S}_1^-\cap \mathcal{C}$.

	We move to the case $p_{1,1}=0$, $p_{0,1}>p_{1,0}$, as we shall study separately $X(y)$ and $Y(x)$. Construct a closed path $\gamma:[0,1]\to \mathcal{S}_2^-\cap\mathcal{C}^+$ such that $\gamma(0)=\gamma(1)=0$ and $\gamma$ avoids any branch point of $X(y)$. One can extend $X(y)$ along $\gamma$ since $X(y)$ is already well defined around $0$ by Lemma~\ref{lem: X(y) around 0}. Further, \eqref{eq: X(y) near 0} implies that $\vert X(y)\vert < \vert y\vert$ for any $y\in \gamma\big((0,1)\big)$ close enough to $\gamma(0)$. Thus, $\vert X(y)\vert < \vert y\vert$ for all $y\in \gamma\big((0,1)\big)$, since $\gamma$ does not pass through $\mathcal{S}_2$. Then, as $y$ tends to $0$ along $\gamma$, $X(y)$ must return to $0$. By the same argument as in the preceding case, $X(y)$ can be analytically extended on the interior of $\gamma$. It follows that $X(y)$ does not have any branch point in $\mathcal{S}_2^-\cap \mathcal{C}$.
	
	We proceed with the proof for $Y(x)$ still in the case $p_{1,1}=0$, $p_{0,1}>p_{1,0}$. Let $x_0$ (resp.~$y_0=-x_0$) be the intersection of $\mathcal{S}_1\setminus\{1\}$ (resp.~$\mathcal{S}_2\setminus\{1\}$) and $\mathbb{R}$. Construct a closed path $\gamma:[0,1]\to \overline{\mathcal{S}_1^-}\cap\mathcal{C}^+$ such that $\gamma(0)=\gamma(1)=x_0$, $\gamma\big((0,1)\big)\in\mathcal{S}_1^-$. Extend $Y(x)$ along $\gamma$, then by Lemma~\ref{lem: solution-neigborhood-S2} and the construction of $\gamma$ ($\gamma$ does not pass through $\mathcal{S}_1$), $\vert Y(x)\vert < \vert x\vert$ for all $x\in \gamma\big((0,1)\big)$. We notice that as $x$ tends to $x_0$ along $\gamma$, then $Y(x)$ must return to $y_0$. Indeed, assuming that $\vert Y(\gamma(1))\vert < \vert y_0 \vert$, we continue to extend $Y(x)$ along a path in $\mathcal{S}_1^+$ from $x_0$ to $0$. Then as $x$ tends to $0$, $\vert Y(x)\vert < \vert x\vert$ (since the path stays in $\mathcal{S}_1^+$), and thus $Y(x)$ must converge to $0$. However, the behavior of the kernel's solutions around $(0,0)$ (see Lemma~\ref{lem: X(y) around 0}) suggests that $\vert Y(x)\vert > \vert x\vert$ as $x$ close to $0$, which yields a contradiction. Hence, $\vert Y(\gamma(1))\vert$ must be equal to $\vert y_0 \vert$, and thus $\left(\gamma(1),Y(\gamma(1))\right)$ is in $\mathcal{K}$ and coincides with $(x_0,y_0)$, thanks to the injectivity of the mappings $s\mapsto(\eta(s)s,\eta(s)s^{-1})$ and $Y:\mathcal{S}_1\to\mathcal{S}_2$ (see Lemma~\ref{lem: S1-S2}\ref{item: one-to-one} and Lemma~\ref{lem: X(y)-Y(x)-on-S1-S2}). We then have $Y(\gamma(1)) = y_0$. By the reasoning in the first case, $Y(x)$ does not have any branch point in $\mathcal{S}_1^-\cap \mathcal{C}^+$.
	
	We move to the case $p_{1,1}\not= 0$. Let $x_0$ (resp.~$y_0=x_0$) be the intersection of $\mathcal{S}_1\setminus\{1\}$ (resp.~$\mathcal{S}_2\setminus\{1\}$) and $\mathbb{R}$. Construct a closed path $\gamma:[0,1]\to \overline{\mathcal{S}_2^-}\cap\mathcal{C}^+$ such that $\gamma(0)=\gamma(1)=y_0$, $\gamma\big((0,1)\big)\in\mathcal{S}_2^-$ and $\gamma$ avoids any branch point of $X(y)$. Extend $X(y)$ along $\gamma$, then by Lemma~\ref{lem: solution-neigborhood-S2} and the construction of $\gamma$ ($\gamma$ does not pass through $\mathcal{S}_2$), $\vert X(y)\vert < \vert y\vert$ for all $y\in \gamma\big((0,1)\big)$. We notice that as $y$ tends to $y_0$ along $\gamma$, then $X(y)$ must return to $x_0$. Indeed, assuming that $\vert X(\gamma(1))\vert < \vert x_0 \vert$, we continue to extend $X(y)$ along a path in $\mathcal{S}_2^+$ from $y_0$ to $0$. Then as $y$ tends to $0$, $\vert X(y)\vert < \vert y\vert$ (since the path stays in $\mathcal{S}_1^+$), and thus $Y(x)$ must converge to $0$, which causes a contradiction since $(0,0)$ is not a root of $K(x,y)$. Hence, $\vert X(\gamma(1))\vert$ must be equal to $\vert x_0 \vert$, and thus $X(\gamma(1)) = x_0$. By the reasoning in the first case, $X(y)$ does not have any branch point in $\mathcal{S}_2^-\cap \mathcal{C}$ and similarly for $Y(x)$ on $\mathcal{S}_1^-\cap \mathcal{C}$.
\end{proof}

Hereafter, the function $X(y)$ (resp.~$Y(x)$), originally defined on $\mathcal{S}_2$ (resp.~$\mathcal{S}_1$), will also indicate its extension on $\mathcal{S}_2^-\cap\mathcal{C}^+$ (resp.~$\mathcal{S}_1^-\cap\mathcal{C}^+$).

\begin{proof}[Proof of Proposition~\ref{prop: extend-solutions-intro}\ref{item: Prop1-item2}]
	We recall that $0\in\mathcal{S}_1^+,\mathcal{S}_2^+$ in the case $p_{1,1}\not=0$, and $0\in\mathcal{S}_1,\mathcal{S}_2$ in the case $p_{1,1}=0$, $p_{1,0}=p_{0,1}$. Since these properties are similar to that of the symmetric models of \cite{HoRaTa-20}, we then refer the readers to \cite[Lem.\,4]{HoRaTa-20} for the proof of these cases. We now prove the lemma for the case $p_{1,1}=0$, $p_{1,0}<p_{0,1}$.

	We first introduce some notations. Let $\mathcal{V}_x$, $\mathcal{V}_y$ respectively denote $\{(x,y)\in V[K]:\vert y\vert <\vert x\vert <1\}$, $\{(x,y)\in V[K]:\vert x\vert <\vert y\vert <1\}$, and $P_x,P_y:\mathbb{C}^2\to\mathbb{C}$ respectively denote the projections along the first and second variable.
	
	Reasoning by contradiction, we assume that there exists $(x_1,y_1)\in\mathcal{V}_x$ such that $x_1\in\mathcal{S}_1^+$. Construct a path $\gamma:[0,1]\to \mathcal{S}_1^+$ having two endpoints $\gamma(0)=x_1$ and $\gamma(1)=0$, and avoiding any points of $V[K]$ such that $\partial_x K=0$. There thus exists a path $\gamma':[0,1]\to V[K]$ such that $\gamma'(0)=(x_1,y_1)$ and $P_x\circ\gamma'=\gamma$. Since $\gamma$ stays in $\mathcal{S}_1^+$, then $\gamma'$ does not meet $\mathcal{K}\setminus\{(0,0)\}$. Consequently, $\gamma'((0,1))$ stays in $ \mathcal{V}_x$ and $P_y\circ\gamma'(1)=0$. On the other hand, by \eqref{eq: X(y) near 0}, one must have $\vert P_x\circ\gamma'(t)\vert < \vert P_y\circ\gamma'(t)\vert$ for $t$ close to $1$, which generates a contradiction. Hence, there is no $(x_1,y_1)\in\mathcal{V}_x$ such that $x_1\in\mathcal{S}_1^+$.
	
	Now let $y_0$ be the intersection of $\mathcal{S}_2\setminus\{1\}$ and $\mathbb{R}$. Assuming that there exists $(x_1,y_1)\in\mathcal{V}_y$ such that $y_1\in\mathcal{S}_2^+$, we construct a path $\gamma:[0,1]\to \mathcal{C}^+$ such that:
	\begin{itemize}
		\item It has two endpoints $\gamma(0)=y_1$ and $\gamma(1)=0$, and passes through $\gamma(\tau)=y_0$ with $\tau\in (0,1)$;
		\item $\gamma((0,\tau))\in\mathcal{S}_2^+$ and $\gamma((\tau,1))\in\mathcal{S}_2^-\cap\mathcal{C}^+$;
		\item $\gamma$ avoids any point of $V[K]$ such that $\partial_yK=0$.
	\end{itemize}
	There thus exists a path $\gamma':[0,1]\to V[K]$ such that $\gamma'(0)=(x_1,y_1)$ and $P_y\circ\gamma'=\gamma$. Since $\gamma((0,\tau))$ stays in $\mathcal{S}_2^+$, then $\gamma'((0,\tau))$ does not meet $\mathcal{K}$ and thus stays in $\mathcal{V}_y$.
	
	We prove that with the above constructions, the path $\gamma'$ has to pass through $\mathcal{K}$ at $(X(y_0),y_0)$. Indeed, we assume the opposite, that $\gamma'$ does not meet $\mathcal{K}\setminus \{(0,0)\}$, then $\gamma'((0,1))\subset \mathcal{V}_y$ and thus $\gamma'(1) = (0,0)=(X(0),0)$. The analytic continuation of $X(y)$ implies that $\gamma'([\tau,1]) = \{(X(y),y):y\in\gamma([\tau,1])\}$ and thus, $\gamma'(\tau) = (X(y_0),y_0)$, which contradicts the assumption.

	Since $\gamma'$ passes through $(X(y_0),y_0)$, then by Lemma~\ref{lem: solution-neigborhood-S2}, for any $t\in (0,\tau)$ close enough to $\tau$, $\vert P_x\circ \gamma'(t) \vert > \vert P_y\circ \gamma'(t) \vert $, which contradicts the assertion that $\gamma'((0,\tau))$ stays in $\mathcal{V}_y$. We thus conclude that there is no $(x_1,y_1)\in\mathcal{V}_y$ such that $y_1\in\mathcal{S}_2^+$. The proof is then complete.
\end{proof}

\section{Conformal welding}\label{sec: conformal welding}
This section aims at constructing conformal mappings from $\mathcal{S}_1^+$ and $\mathcal{S}_2^+$ onto two disjoint, complementary domains such that the two coordinates of any point in $\mathcal{K}$ are joined together by the boundary value of these mappings. The problem can be reduced to that of finding two conformal mappings from the upper and lower half-planes onto disjoint, complementary domains such that the extensions of the mappings on $\mathbb{R}$ differ by a quasisymmetric homeomorphism (also called a shift function). Such a problem is usually referred to as conformal welding.

We introduce some notations. Let $\mathcal{A}$ denote a Jordan curve (that is, a non-self-intersecting closed curve) or an infinite non-self-intersecting curve dividing the complex plane into two disjoint domains $\mathcal{A}^+$ and $\mathcal{A}^-$ (if $\mathcal{A}$ is a Jordan curve, then $\mathcal{A}^-$ conventionally indicates the domain containing infinity). Let also $f$ be a function defined on $\mathcal{A}^+$  (resp.~$\mathcal{A}^-$), then for any $t\in\mathcal{A}$, $f^+(t)$ (reps.\ $f^-(t)$) will denote the limit value of $f(x)$ as $x$ tends to $t$, provided it exists.

\subsection{Preliminary construction of conformal mappings}

In this subsection, we construct conformal mappings from $\mathcal{S}_1^+$ and $\mathcal{S}_2^+$ respectively onto the upper and lower half-planes, and introduce the corresponding shift functions on the real line.

\begin{figure}[t]
	\centering
	\begin{tikzpicture}
	\node at (0,0) {\includegraphics[scale = 0.2]{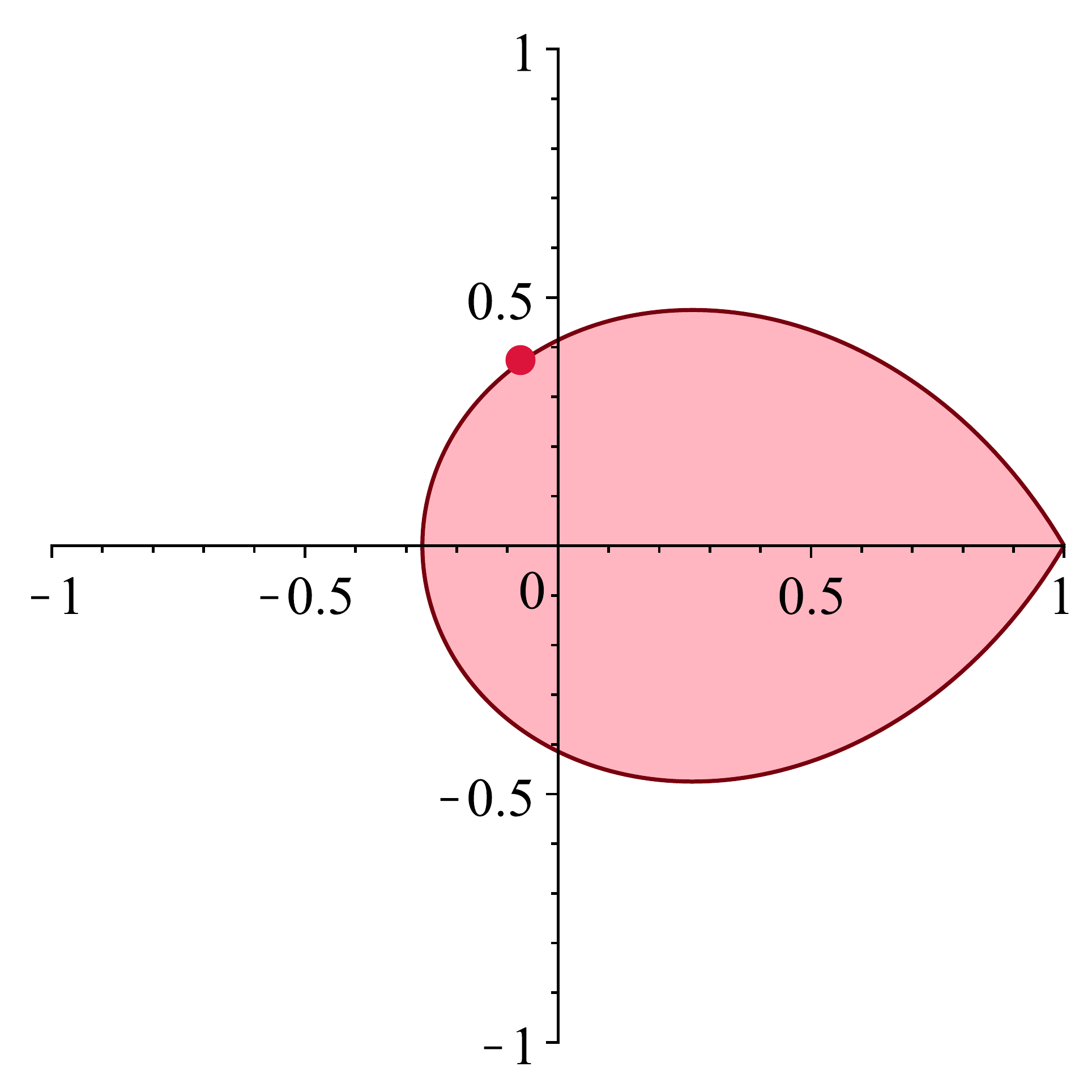}};
	\node at (0,-4) {\includegraphics[scale = 0.2]{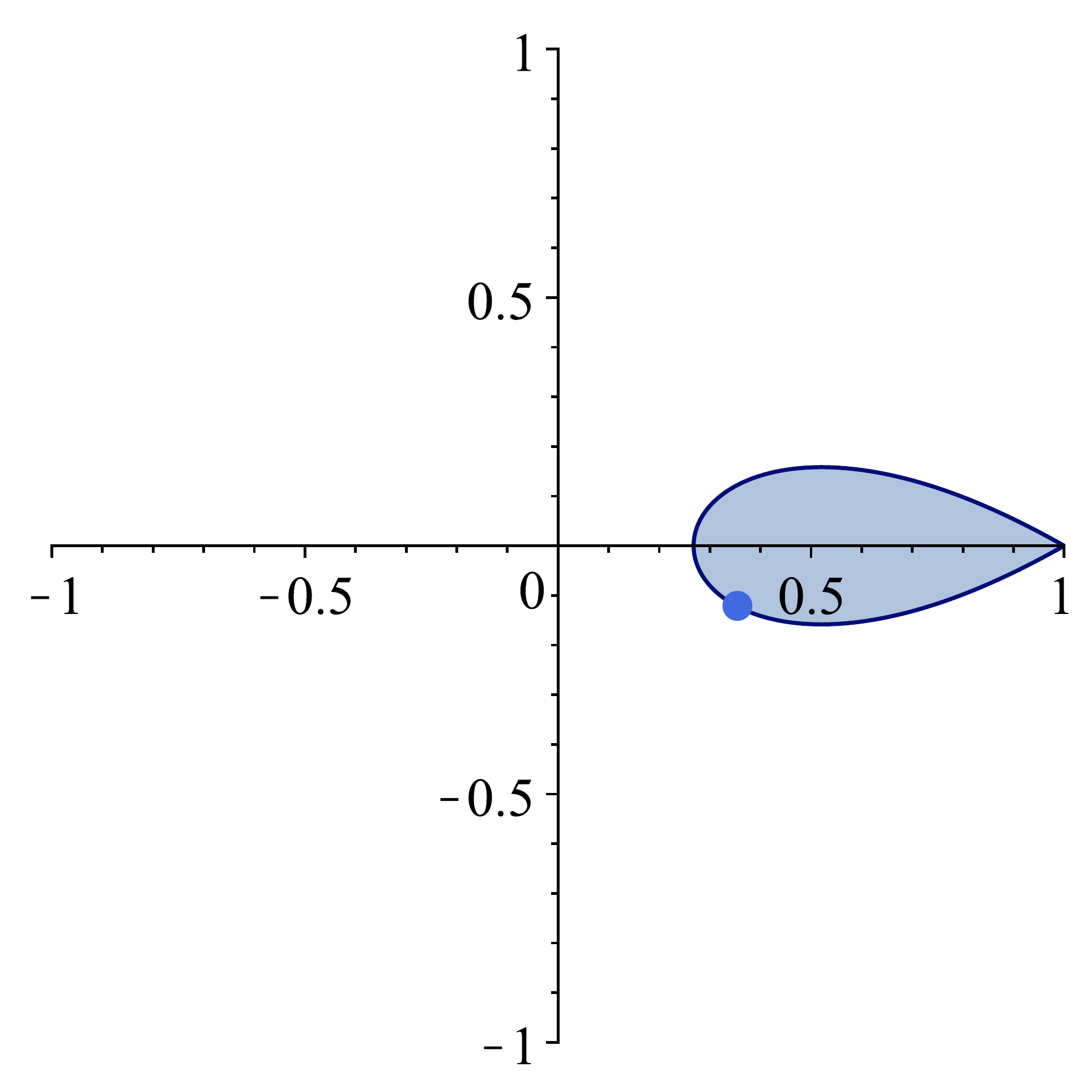}};
	\node at (5,0) {\includegraphics[scale = 0.2]{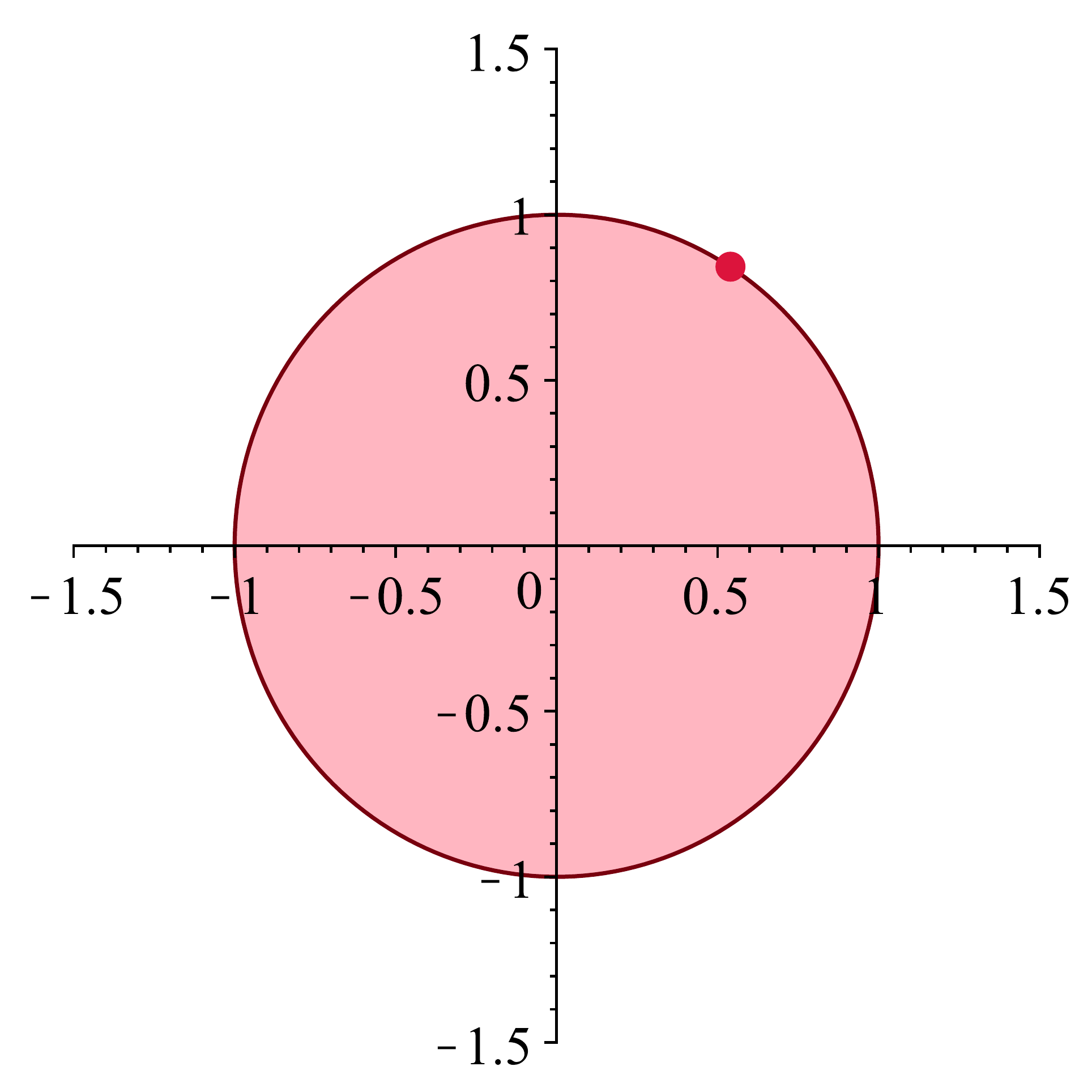}};
	\node at (5,-4) {\includegraphics[scale = 0.2]{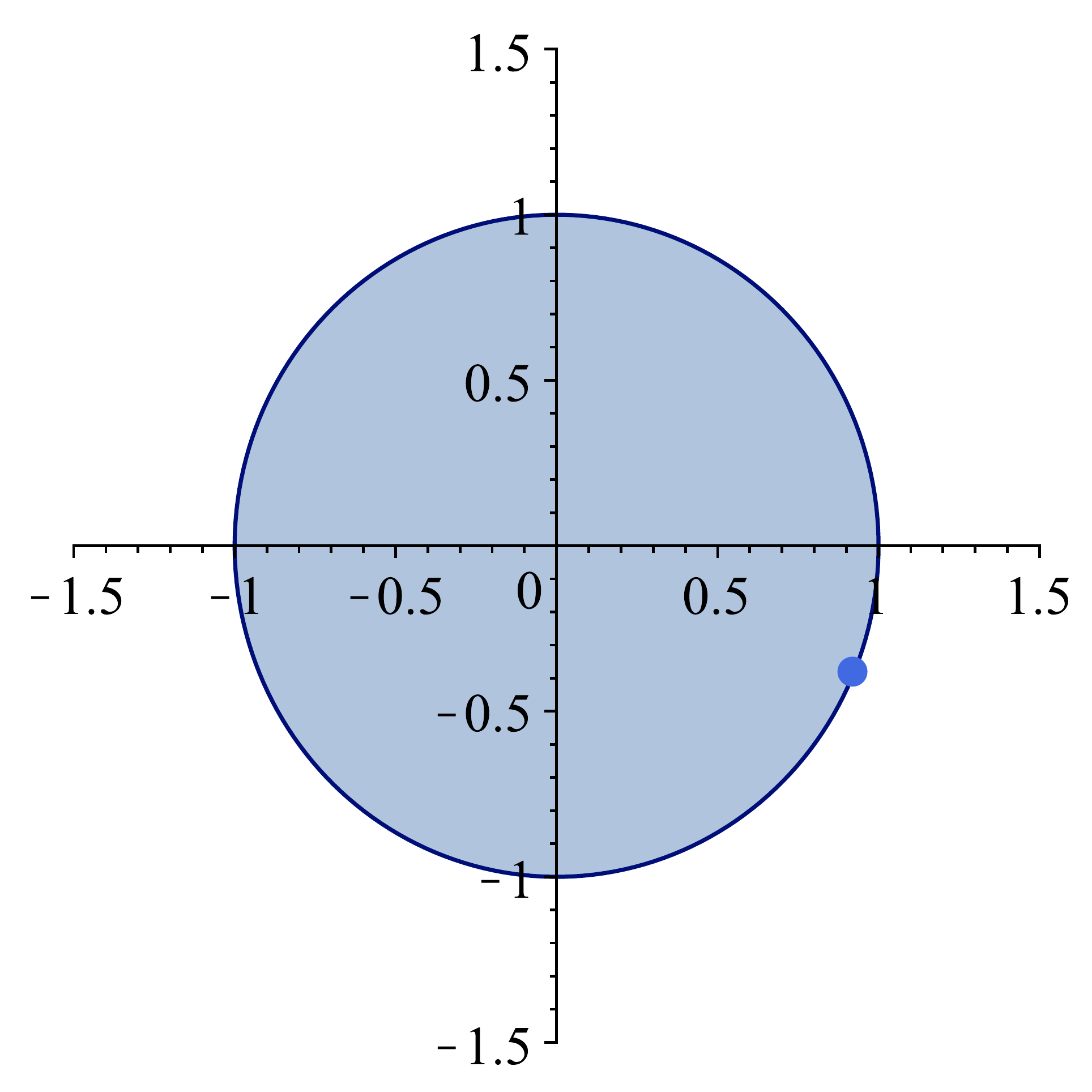}};
	\node at (10,-2) {\includegraphics[scale = 0.2]{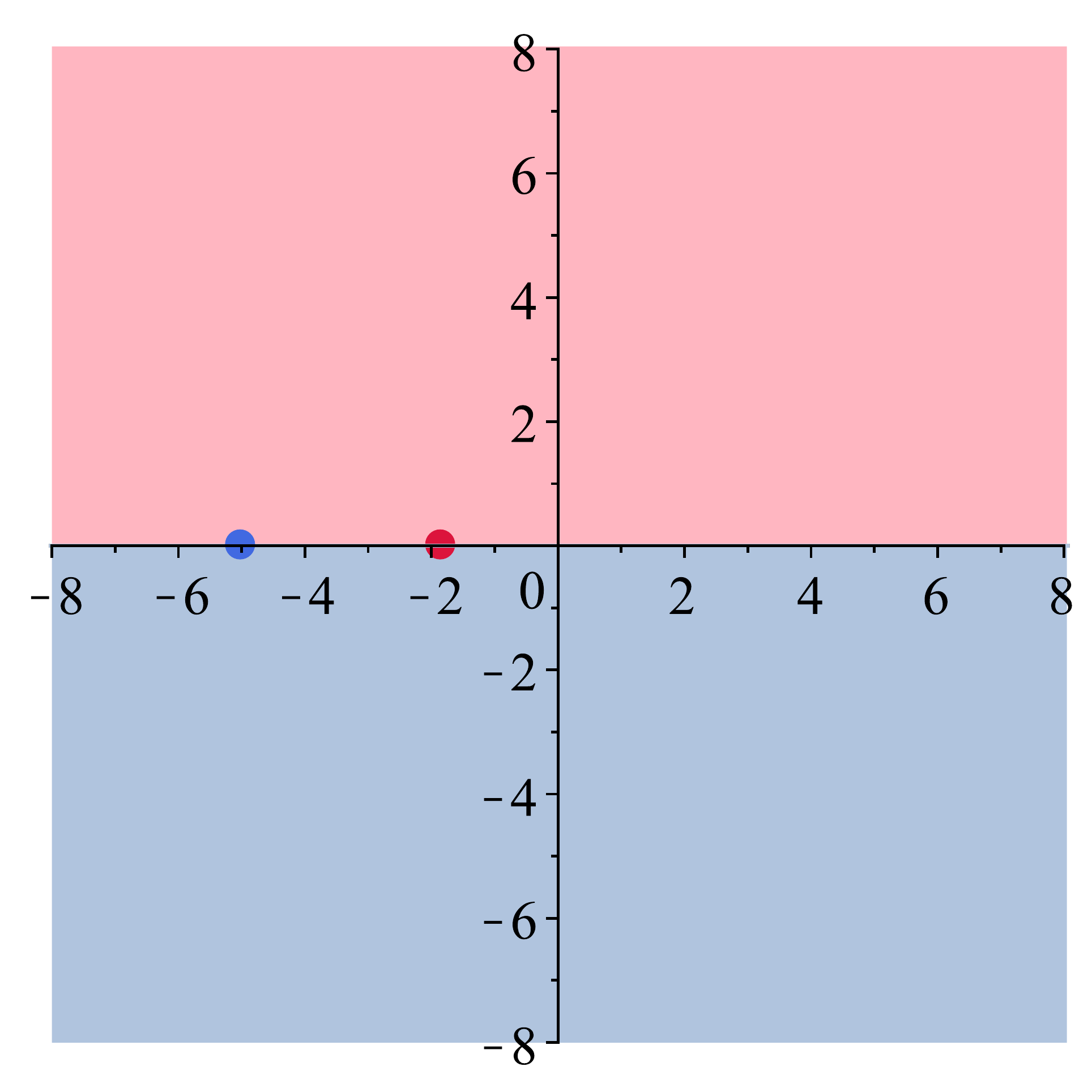}};
	
	\draw[->,thick] (2.1,0)--(3.1,0);
	\draw[->,thick] (2.1,-4)--(3.1,-4);
	\draw[->,thick] (7,0)--(8,-0.5);
	\draw[->,thick] (7,-4)--(8,-3.5);
	
	\node at (2.6,-0.25) {$\pi_1(z)$};
	\node at (2.6,-3.75) {$\pi_2(z)$};
	\node at (7.3,-0.5) {$\phi(z)$};
	\node at (7.1,-3.5) {$-\phi(z)$};
	
	\draw[->,dashed,gray] (8.95,-2.1) .. controls (10,-4) and (8,-6) .. (6.2,-4.5);
	\draw[->,dashed,gray] (6.03,-4.53) .. controls (5,-6) and (2,-6) .. (0.75,-4.3);
	\draw[->,dashed,gray] (0.6,-4.2) .. controls (-1.5,-3) and (-1.5,-1).. (-0.2,0.65);
	\draw[->,dashed,gray] (-0.05,0.77) .. controls (1.5,2) and (3.7,2.3) .. (5.6,1.1);
	\draw[->,dashed,gray] (5.75,1.1) .. controls (7.5,2) and (10,0.5) .. (9.65,-1.9);
	\draw[->,dashed,thick] (10-1.1,0.1-2) .. controls (10-0.9,0.3-2) and (10-0.6,0.3-2)..  (10-0.4,0.1-2);
	
	\node at (9.4,-4.7) {\textcolor{gray}{$\phi^{-1}(-z)$}};
	\node at (2.5,-5.8) {\textcolor{gray}{$(\pi_2^{-1})^+(z)$}};
	\node at (-1.5,-2) {\textcolor{gray}{$X(y)$}};
	\node at (2.5,2.1) {\textcolor{gray}{$\pi_1^+(z)$}};
	\node at (8.9,1) {\textcolor{gray}{$\phi(z)$}};
	\node at (10-0.75,0.45-2) {$\alpha$(z)};
	\end{tikzpicture}
	\caption{Some curves and conformal mappings in the model $p_{0,1}=p_{0,-1}=3/8$, $p_{1,0}=p_{-1,0}=1/8$: $\pi_1(z)$ maps conformally $\mathcal{S}_1^+$ (red domain in the top left figure) onto the unit disk; $\pi_2(z)$ maps conformally $\mathcal{S}_2^+$ (blue domain in the bottom left figure) onto the unit disk; $\phi(z)$ (resp.~$-\phi(z)$) maps conformally the unit disk onto the upper (resp.~lower) half-plane; The right figure shows the shift function $\alpha(z)$ constructed by composing $\phi^{-1}(-z)$, $(\pi_2^{-1})^+(z)$, $X(y)$, $\pi_1^+(z)$ and $\phi(z)$ (for example, along the gray dashed arrows in the figure).}
	\label{fig: preliminary conformal mappings}
\end{figure}
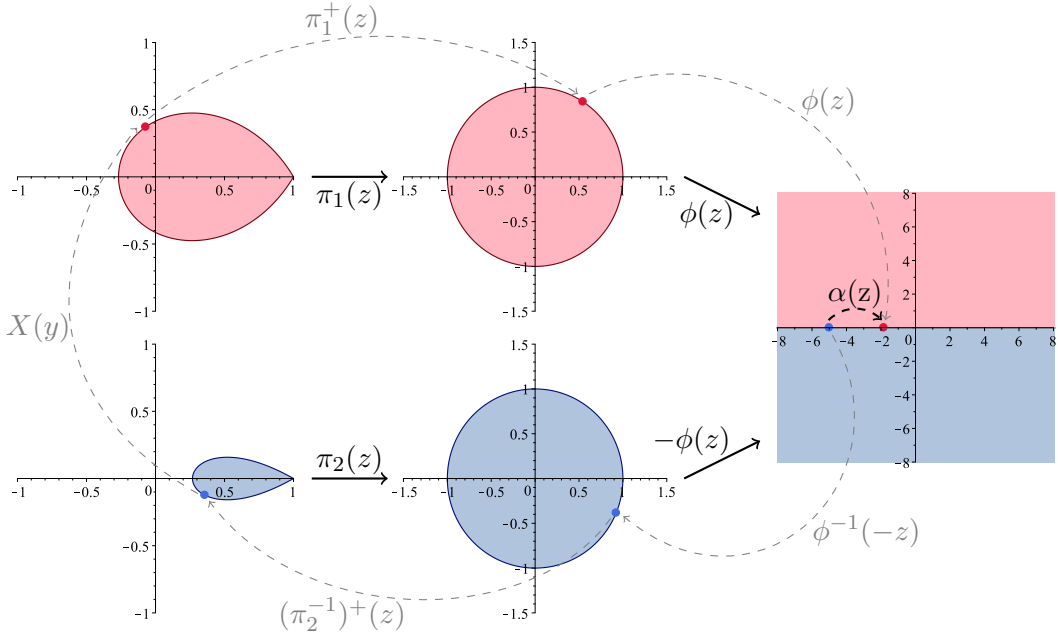

The following lemma is implied by the classical Riemann mapping theorem.

\begin{lemma}\label{lem: pi_1}
	Under Assumption~\ref{assumption: monotonic}, for any $r\in\mathcal{S}_1^+\cap\mathbb{R}$, there exists uniquely a conformal mapping $\pi_1$ from $\mathcal{S}_1^+$ onto $\mathcal{C}^+$ such that $\pi_1(r)=0$ and $\pi_1'(r)>0$. 
	Further, for all $z\in\mathcal{S}_1^+$, $\pi_1(z) = \overline{\pi_1(\overline{z})}$ and $\pi_1^+(1)=1$.
\end{lemma}

A detailed proof of Lemma~\ref{lem: pi_1} can be found in \cite[Lem.\,5]{HoRaTa-20}, since under Assumption~\ref{assumption: monotonic}, $\mathcal{S}_1$ is a Jordan curve and symmetric with respect to the real axis, which is similar to the symmetric case $p_{k,\ell}=p_{\ell,k}$ for all $k,\ell$.

We now construct a conformal mapping $\pi_2$ from $\mathcal{S}_2^+$ onto $\mathcal{C}^+$ with analogous properties to $\pi_1$ and introduce a M\"obius transformation:
\begin{equation*}\label{eq: phi}
	\phi(z) := -i\frac{z+1}{z-1},
\end{equation*}
which maps conformally $\mathcal{C}^+$ onto the upper half-plane $\mathcal{H}^+$. Thus, the composition $\phi\circ\pi_1$ (resp.~$-\phi\circ\pi_2$) maps conformally $\mathcal{S}_1^+$ (resp.~$\mathcal{S}_2^+$) onto $\mathcal{H}^+$ (resp.~$\mathcal{H}^-$).

Carath\'eodory's theorem (see \cite[Thm.\,3.1]{GaMa-05}) ensures that the limit values of such conformal mappings exist and form continuous homeomorphisms on the boundaries of the domains. We then introduce the shift function
\begin{equation}\label{eq: alpha(z)}
	\alpha(z):=\phi\circ\pi_1^+ \circ X\circ(\pi_2^{-1})^+\circ\phi^{-1}(-z),\quad z\in\mathbb{R},
\end{equation}
(see Figure~\ref{fig: preliminary conformal mappings}).

Hereafter in this section, we will consider the following assumption:
\begin{enumerate}[label=(K\arabic{*}),ref={\rm (K\arabic{*})},resume]
	\item\label{assumption: theta_1,theta_2 not=0} The set of weights $\{p_{k,\ell}\}_{k,\ell}$ satisfies \begin{equation*}
		\begin{cases}
		-\sum p_{k,\ell}k\ell < \sum p_{k,\ell}k^2,\\
		-\sum p_{k,\ell}k\ell < \sum p_{k,\ell}\ell^2.
	\end{cases}
	\end{equation*}
\end{enumerate}
We recall from Remark~\ref{rmk: angles theta1, theta2} that Assumption~\ref{assumption: theta_1,theta_2 not=0} is equivalent to $\theta_1,\,\theta_2\not= 0$. The following lemma presents some crucial properties of $\alpha(z)$.

\begin{lemma}\label{lem: alpha(z)}
	Under Assumptions~\ref{assumption: monotonic} and \ref{assumption: theta_1,theta_2 not=0}, $\alpha (z)$ is a one-to-one, strictly increasing and odd function from $\mathbb{R}\cup\{\infty\}$ onto itself. Further, it is analytic and possesses non-vanishing derivatives on $\mathbb{R}$ with the asymptotic behavior as $z\to +\infty$
	\begin{equation}\label{eq: alpha(z) at infty}
		\alpha(z) \sim  \text{const}\cdot z^{\theta_2/\theta_1},
	\end{equation}
	where $\theta_1$, $\theta_2$ are defined in \eqref{eq: theta_1 theta_2}.
\end{lemma}

\begin{proof}
	Since all the functions $\phi^{-1} |_\mathbb{R}$, $(\pi_2^{-1})^+$, $X|_{\mathcal{S}_2}$, $\pi_1^+$, and $\phi |_\mathcal{C}$ are one-to-one on their domains of definition, then $\alpha(z)$ is also one-to-one on $\mathbb{R}\cup\{\infty\}$. Further, $\pi_1^+(z)$ and $(\pi_2^{-1})^+(z)$  preserve (resp.~$X(y)$ inverses) the orientation as $z$ moves counterclockwise, then $\alpha(z)$ is strictly increasing on $\mathbb{R}$.
	
	Now notice that  $\phi(z)=-\phi(1/z)=-\phi(\overline{z})$ for all $z\in\mathcal{C}$, or equivalently, $\phi^{-1}(z) = 1/\phi^{-1}(-z) = \overline{\phi^{-1}(-z)}$ for all $z\in\mathbb{R}$. Besides, one has:
	\begin{equation*}
		\pi_1^+ \circ X\circ(\pi_2^{-1})^+(\overline{z}) = \overline{\pi_1^+ \circ X\circ(\pi_2^{-1})^+(z)},	
	\end{equation*}
	for all $z\in\mathcal{C}$. Thus, it is easily checked that $\alpha(z)$ is an odd function.
	
	Since $\pi_1^+:\mathcal{S}_1\to\mathcal{C}$ is the extension of the conformal mapping $\pi_1$ to the boundary, and $\mathcal{S}_1\setminus\{1\}$ and $\mathcal{C}\setminus\{1\}$ are analytic and smooth, then $\pi_1^+$ is analytic (see \cite[p.\,186]{Ne-52}) and possesses non-vanishing derivatives on $\mathcal{S}_1\setminus\{1\}$ (see \cite[Thm.\,3.9]{Po-92}). In particular, $\mathcal{S}_1$ admits a corner point of angle $\theta_1$ at $1$, whereas $\mathcal{C}$ is smooth at $1$, then by \cite[Thm.\,3.11]{Po-92}, one has the asymptotic behavior as $z\to 1$:
	\begin{equation*}
		\pi_1^+(z) = 1 + \text{const}\cdot (z-1)^{\pi/\theta_1} + o((z-1)^{\pi/\theta_1}).
	\end{equation*}
	Similarly, $(\pi_2^{-1})^+$ is analytic, possesses non-vanishing derivatives on $\mathcal{C}\setminus\{1\}$, and as $z\to 1$,
	\begin{equation*}
		(\pi_2^{-1})^+(z) = 1 + \text{const}\cdot (z-1)^{\theta_2/\pi} + o((z-1)^{\theta_2/\pi}).
	\end{equation*}
	We recall that by Lemma~\ref{lem: X(y)-Y(x)-on-S1-S2}, $X(y)$ is analytic, possesses non-vanishing derivatives on $\mathcal{S}_2\setminus\{1\}$, and left and right derivatives at $1$. Since $\alpha(z)$ is composed of the above functions and $\phi(z)$, then $\alpha(z)$ is also analytic, possesses non-vanishing derivatives on $\mathbb{R}$, and admits the asymptotic behavior \eqref{eq: alpha(z) at infty} as $z\to \infty$.
\end{proof}

We would like to specify some facts about the shift function $\alpha(z)$ in the symmetric case in \cite{HoRaTa-20}, where $\mathcal{S}_1$ and $\mathcal{S}_2$ are Jordan curves and coincide, and $X(y)=\overline{y}$ for all $y\in\mathcal{S}_2$. As a result, the shift function $\alpha(z)$ degenerates into the identity on $\mathbb{R}$ in the symmetric case. We notice that the assumptions \ref{assumption: monotonic} and \ref{assumption: theta_1,theta_2 not=0} always hold true in this case. Hence, by also considering the assumptions \ref{assumption: monotonic} and \ref{assumption: theta_1,theta_2 not=0} in the non-symmetric case, we restrict our analysis to models that share some critical properties with the symmetric one. 

\subsection{Conformal welding problem with quasisymmetric shift}\label{subsec: conformal welding}
This subsection aims at solving the following BVP: Find mappings $\chi_1$ and $\chi_2$ such that $\chi_1$ and $\chi_2$ map respectively $\mathcal{H}^+$ and $\mathcal{H}^-$ conformally onto two disjoint, complementary domains, satisfying the boundary condition
\begin{equation*}
	\chi_1^+(\alpha(z)) = \chi_2^-(z),
\end{equation*}
for all $z\in\mathbb{R}$. We will show that the shift function $\alpha(z)$ is quasisymmetric under some assumptions.

\begin{figure}[t]
	\centering
	\begin{tikzpicture}
	\node at (0,0) {\includegraphics[scale=0.2]{Fig-halfplan-eps-converted-to.pdf}};
	\node at (5,0) {\includegraphics[scale=0.2]{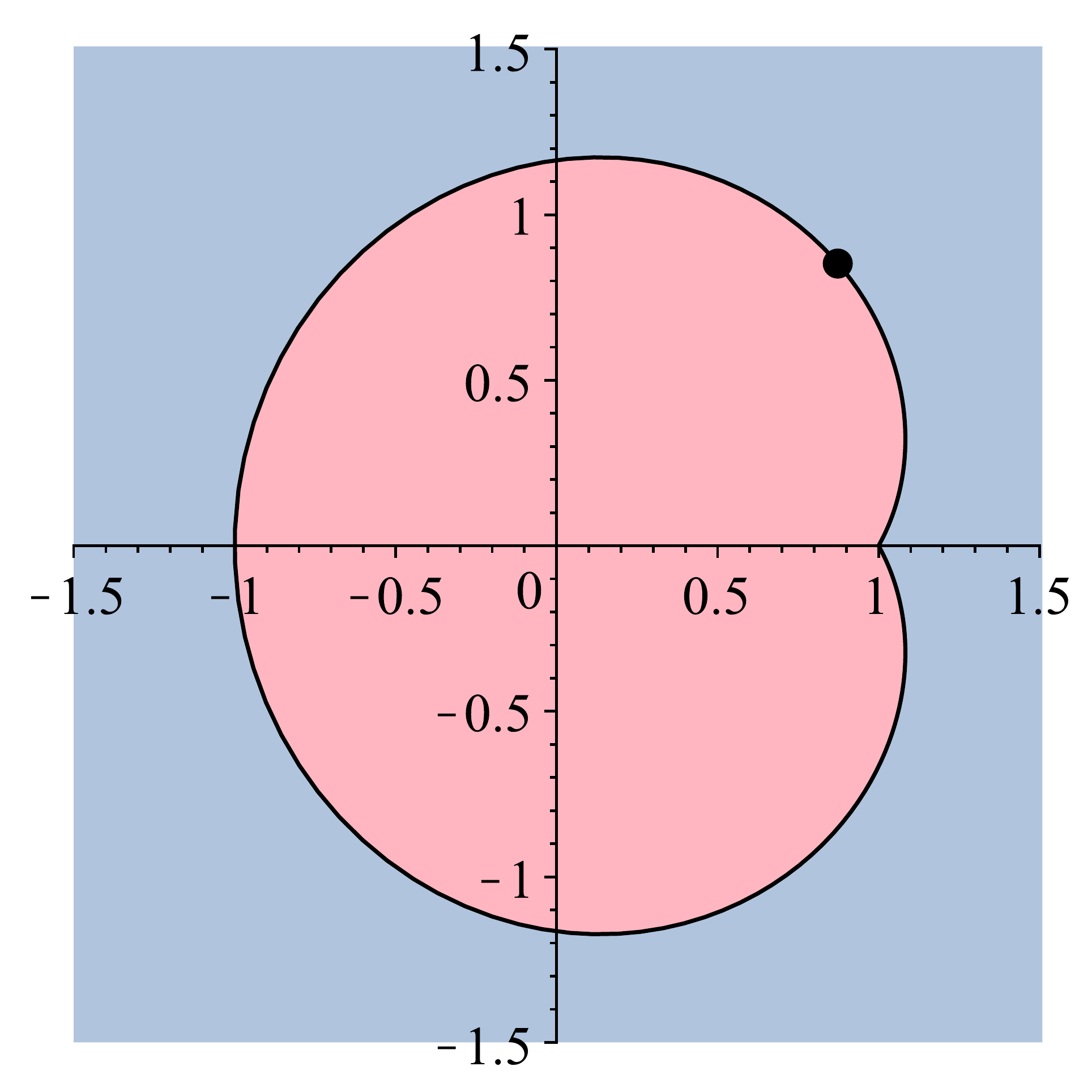}};
	\node at (10,0) {\includegraphics[scale=0.2]{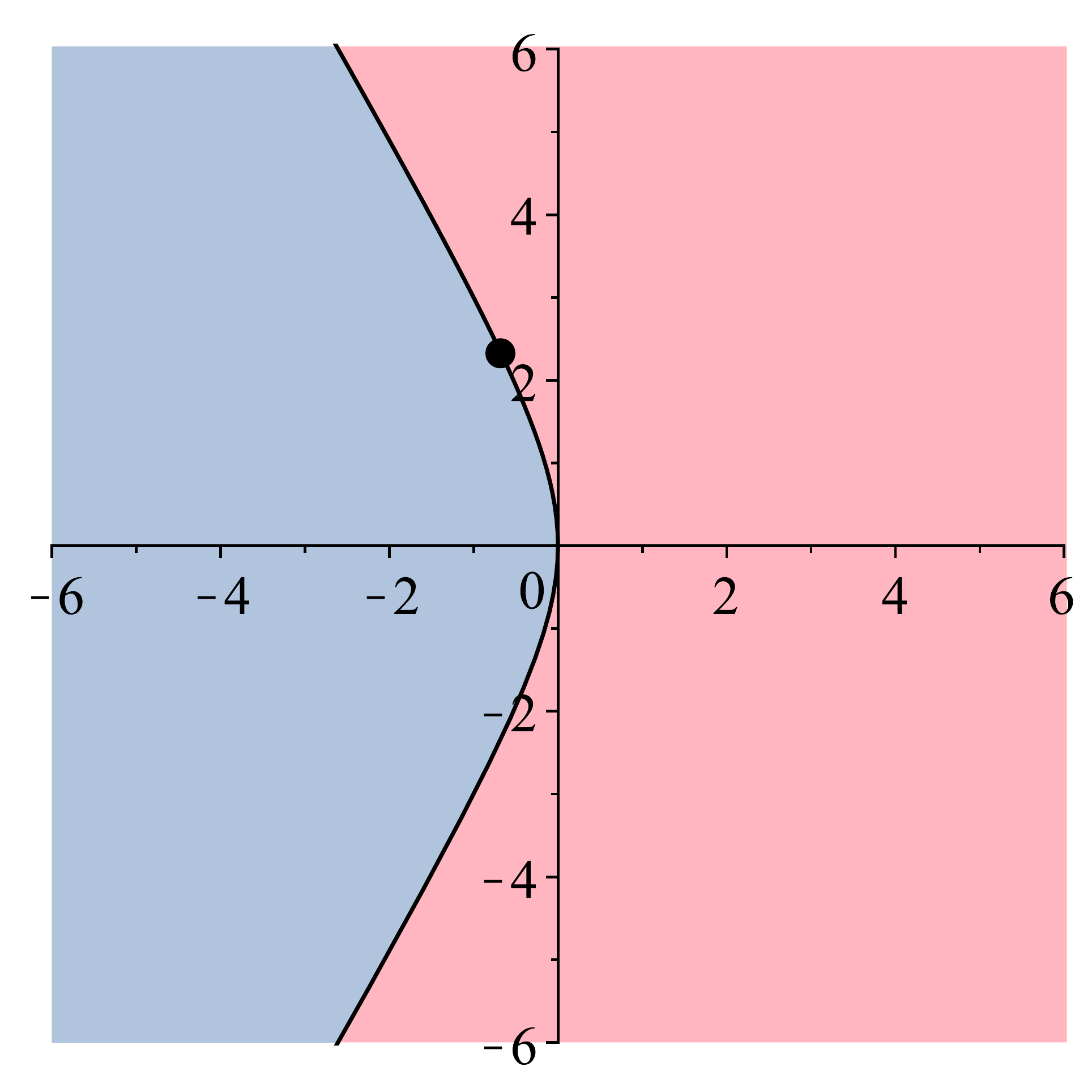}};
	
	\draw[->,thick] (1,-1.3).. controls (2,-2) and (3,-2).. (4,-1.3);
	\draw[->,thick] (1,1.3) .. controls (2.2,2) and (3.8,1.5).. (4.5,0.7);
	\draw[->,thick] (7,0) -- (8,0);
	\draw[->,dashed] (-1.1,0.1) .. controls (-0.9,0.3) and (-0.6,0.3)..  (-0.4,0.1); 
	
	\node at (2.7,1.9) {$\chi_1(z)$};
	\node at (2.7,-2.1) {$\chi_2(z)$};
	\node at (7.5,0.3) {$\phi(z)$};
	\node at (-0.75,0.45) {$\alpha$(z)};
	\end{tikzpicture}
	\caption{Some curves and conformal mappings concerning the BVP with shift in the model $p_{0,1}=p_{0,-1}=3/8$, $p_{1,0}=p_{-1,0}=1/8$: the first figure shows the shift function $\alpha(z)$ on $\mathbb{R}$; the second figure shows the Jordan curve $\mathcal{G}$ constructed in Lemma~\ref{lem: chi1-chi2-Jordan curve}; $\chi_1(z)$ maps conformally $\mathcal{H}^+$ onto $\mathcal{G}^+$, $\chi_2(z)$ maps conformally $\mathcal{H}^-$ onto $\mathcal{G}^-$; for any $z\in\mathbb{R}$ (blue point), its image (black point in the second figure) under $\chi_2^-$ coincides with the image of $\alpha(z)$ (red point) under $\chi_1^+$; the third figure shows the images of $\mathcal{G}$, $\mathcal{G}^+$ and $\mathcal{G}^-$ under the M\"obius transformation $\phi$ defined in \eqref{eq: phi}.}
	\label{fig: conformal welding}
\end{figure}

We first recall the definition of a quasisymmetric function on the real line. A strictly increasing function $f$ defined on the real line is quasisymmetric if and only if there exists $k>0$ such that
\begin{equation}\label{eq: quasisymmetric}
	\frac{1}{k} \leq \frac{f(x+t)-f(x)}{f(x)-f(x-t)} \leq k,
\end{equation}
for all $x\in\mathbb{R}$ and $t>0$.

The following lemma presents a result of quasisymmetric functions.
\begin{lemma}\label{lem: quasisymmetric-power-func}
	For all $p>0$, the function
	\begin{equation}\label{eq: power_func}
	P(x) := \sign(x)\vert x\vert ^p
	\end{equation}
	is quasisymmetric on $\mathbb{R}$.
\end{lemma}

\begin{proof}
	We first notice that
	\begin{equation*}
		\lim_{t\to 0} \frac{P(x+t)-P(x)}{P(x)-P(x-t)} = 1,
	\end{equation*}
	for all $x\in \mathbb{R}$. Now for all $(x,t)\in\mathbb{R}\times\mathbb{R}^+$,
	\begin{equation*}
		\frac{P(x+t)-P(x)}{P(x)-P(x-t)} = \frac{\sign(x/t+1)\vert x/t +1\vert^p - \sign(x/t)\vert x/t \vert^p }{\sign(x/t)\vert x/t \vert^p - \sign(x/t-1)\vert x/t -1\vert^p}.
	\end{equation*}
	Since the function
	\begin{equation*}
		\frac{\sign(z+1)\vert z +1\vert^p - \sign(z)\vert z \vert^p }{\sign(z)\vert z \vert^p - \sign(z-1)\vert z -1\vert^p}
	\end{equation*}
	is continuous, strictly positive on $\mathbb{R}$, and converges to $1$ as $z$ tends to $\pm\infty$, then it is bounded in $\mathbb{R}^+$ and thus, there exists a suitable number $k>0$ satisfying
	\begin{equation*}
		\frac{1}{k}\leq \frac{P(x+t)-P(x)}{P(x)-P(x-t)}\leq k,
	\end{equation*}
	for all $(x,t)\in\mathbb{R}\times\mathbb{R}^+$. The proof is then complete.
\end{proof}

To show that $\alpha(z)$ is quasisymmetric, we introduce the following assumption:
\begin{enumerate}[label=(K\arabic{*}),ref={\rm (K\arabic{*})},resume]
	\item\label{assumption: alpha'(z)} $\alpha'(z) = \text{const}\cdot \vert z\vert^q + o(z^q) $ as $z\to\pm\infty$, with $q\in\mathbb{R}$.
\end{enumerate}
Assumption~\ref{assumption: alpha'(z)} is particularly technical as it is stated for the shift function $\alpha(z)$, which is difficult to verify, instead of stated for the weights $\{p_{k,\ell}\}_{k,\ell}$ or the curves $\mathcal{S}_1$ and $\mathcal{S}_2$. Therefore, we also introduce in Lemma~\ref{lem: alpha'(z) moment order 4} a sufficient condition on the moments of the weights $\{p_{k,\ell}\}_{k,\ell}$, under which Assumption~\ref{assumption: alpha'(z)} always holds true.

\begin{lemma}\label{lem: alpha'(z) moment order 4}
	Assume that
		\begin{enumerate}[label=(H\arabic{*}),ref={\rm (H\arabic{*})}]\setcounter{enumi}{6}
			\item\label{assumption: H7 moment order 4} The set of weights $\{p_{k,\ell}\}_{k,\ell}$ has finite moments of order 4, i.e., $\sum p_{k,\ell}k^4 + \sum p_{k,\ell}\ell^4<\infty$, and  the one-sided second order derivatives of $X(y)$ at $1$, which are fully characterized by the moments of $\{p_{k,\ell}\}_{k,\ell}$ up to order 3, do not vanish.
		\end{enumerate}
	Then under Assumptions~\ref{assumption: monotonic}, \ref{assumption: theta_1,theta_2 not=0}, and \ref{assumption: H7 moment order 4}, we have
	\begin{equation*}
		\alpha'(z)\sim \text{const}\cdot\vert z\vert^{\theta_2/\theta_1-1},
	\end{equation*}
	as $z\to \pm\infty$, where $\theta_1$, $\theta_2$ are defined in \eqref{eq: theta_1 theta_2}.
\end{lemma}

\begin{proof}
	By differentiating three times the identity $K(X(y),y)=0$ and evaluating it as $y\to 1$, it is easily seen that the one-sided second order derivatives of $X(y)$ at $1$ are characterized by the moments of $\{p_{k,\ell}\}_{k,\ell}$ up to order 3. By the assumption, we then have
	\begin{equation*}
		X'(y)\sim \text{const}\cdot (y-1) + o(y-1),
	\end{equation*}
	as $y\to 1$.
		
	We now study the asymptotic behavior of the derivatives of $\pi_2^{-1}(z)$ as $z\to 1$. By the assumption on the moments of order 4, the parametrization of $\mathcal{S}_2$, which is $s\mapsto \eta(s)s^{-1}$ on $\mathcal{C}$, admits finite one-sided derivatives up to order 3 at $s=1$. Hence, the mapping $s\mapsto \eta(s)s^{-1}$ and its first and second derivatives are Lipschitz continuous on $\mathcal{C}$. Recall that as $z\to 1$,
	\begin{equation*}
		\pi_2^{-1}(z) = 1 + \text{const}\cdot (z-1)^{\theta_2/\pi} + o\left(  (z-1)^{\theta_2/\pi} \right).
	\end{equation*}
	Applying a result from \cite[Theorem~1]{Wi-65}, one may obtain the expansion of $(\pi_2^{-1})'$ by differentiating formally the expansion of $\pi_2^{-1}$:
	\begin{equation*}
		\pi_2^{-1}(z)' \sim \text{const}\cdot (z-1)^{\theta_2/\pi - 1} + o\left(  (z-1)^{\theta_2/\pi-1} \right),
	\end{equation*}
	as $z\to 1$. We remark that the result in \cite[Theorem~1]{Wi-65} is stated for conformal mappings from the upper half-plane onto a domain with a corner point at $0$, such that $0$ is mapped to $0$. However, it can be also applied to the current context, as one can use translation and M\"obius transformation to return to the original setting of the theorem. We also obtain a similar result for the mapping $\pi_1^{-1}$ and deduce the expansion of the derivatives of its inverse by \cite[Theorem~3]{Wi-65}:
	\begin{equation*}
		\pi_1(z)'\sim \text{const}\cdot (z-1)^{\pi/\theta_1 - 1} + o\left(  (z-1)^{\pi/\theta_1-1} \right),
	\end{equation*}
	as $z\to 1$. Since $\alpha(z)$ is the composition of $\phi(z)$, $\pi_1^+(z)$, $X(y)$, $(\pi_2^{-1})^+(z)$, and $\phi^{-1}(-z)$, all of which have the asymptotic expansions for the derivatives, then we also obtain the asymptotics for $\alpha'(z)$:
	\begin{equation*}
		\alpha'(z)\sim\text{const}\cdot \vert z\vert^{\theta_2/\theta_1-1},
	\end{equation*}
	as $z\to\pm\infty$.
\end{proof}

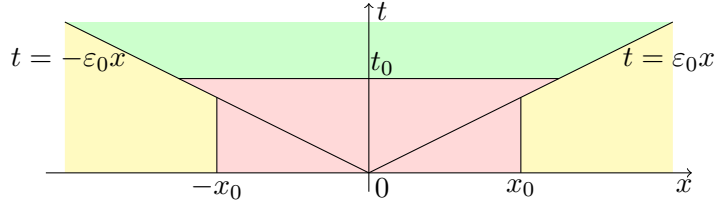
\begin{figure}[t]
	\centering
	\begin{tikzpicture}[scale=0.5]
	\fill[green!20] (-8,4)--(-5,2.5)--(5,2.5)--(8,4);
	\fill[yellow!30] (-8,4)--(-4,2)--(-4,0)--(-8,0);
	\fill[yellow!30] (8,4)--(4,2)--(4,0)--(8,0);
	\fill[red!15] (-5,2.5)--(-4,2)--(-4,0)--(4,0)--(4,2)--(5,2.5);
	
	\draw[->,color=black] (-8.5,0) -- (8.5,0);
	\draw[->,color=black] (0,-0.5) -- (0,4.5);
	
	\draw (0,0)--(8,4);
	\draw (0,0)--(-8,4);
	\draw (-5,2.5)--(5,2.5);
	\draw (-4,0)--(-4,2);
	\draw (4,0)--(4,2);
	
	\node at (8.3,-0.35) {$x$};
	\node at (0.35,4.3) {$t$};
	\node at (0.35,-0.4) {$0$};
	\node at (4,-0.4) {$x_0$};
	\node at (-4,-0.4) {$-x_0$};
	\node at (0.35,2.9) {$t_0$};
	\node at (7.9,3) {$t=\epsilon_0 x$};
	\node at (-7.9,3) {$t=-\epsilon_0 x$};
	\end{tikzpicture}
	\caption{The region $\mathbb{R}\times\mathbb{R}^+$ is divided into smaller domains: $R_1\cap\{(x,t)\in\mathbb{R}\times\mathbb{R}^+:t\geq t_0\}$ (in green), $R_2\cap\{(x,t)\in\mathbb{R}\times\mathbb{R}^+:\vert x\vert \geq x_0\}$ (in yellow), and $R_3$ (in red).}
	\label{fig: RxR^+}
\end{figure}

\begin{lemma}\label{lem: alpha quasisymmetric}
	Under Assumptions~\ref{assumption: monotonic}--\ref{assumption: alpha'(z)}, $\alpha(z)$ is a quasisymmetric function from $\mathbb{R}$ onto itself.
\end{lemma}

\begin{proof}
	Without loss of generality, we omit the constant coefficients of $\alpha(z)$ in \eqref{eq: alpha(z) at infty} and $\alpha' (z)$ in Assumption~\ref{assumption: alpha'(z)}. We rewrite $\alpha(z)$ as
	\begin{equation*}
		\alpha(z) = P(z) + r(z),
	\end{equation*}
	where $P(z)=\sign(z)\vert z\vert ^p$, $p:=\theta_2/\theta_1$ and $r(x) = o(\vert x\vert ^p)$ as $x\to\pm\infty$. We now prove that there exist $k_1\in (0,1)$ and $k_2>1$ such that
	\begin{equation}\label{eq: quasisymmetric operator}
		F(\alpha)(x,t,k_2)\leq 0 \leq F(\alpha)(x,t,k_1),
	\end{equation}
	for all $(x,t)\in\mathbb{R}\times\mathbb{R}^+$, where $F(f)$ is an operator acting on any function $f:\mathbb{R}\to\mathbb{R}$ and defined as follows:
	\begin{equation}
		F(f)(x,t,k):= f(x+t)-f(x) - k[f(x)-f(x-t)].
	\end{equation}
	One can rewrite $F(\alpha)$ as
	\begin{equation}\label{eq: F(alpha)(x,t,k)}
		F(\alpha)(x,t,k) = F(P)(x,t,k) + F(r)(x,t,k)= t^p\left[ F(P)(\frac{x}{t},1,k)+ \frac{F(r)(x,t,k)}{t^p}\right].
	\end{equation}
	
	Fix $\epsilon_0 \in (0,1)$ and set
	\begin{equation*}
		R_1:=\{ (x,t)\in\mathbb{R}\times\mathbb{R}^+: -\frac{1}{\epsilon_0} \leq \frac{x}{t} \leq \frac{1}{\epsilon_0}\},
	\end{equation*}
	(see Figure~\ref{fig: RxR^+}). Since $x/t$ is bounded on $R_1$, then for any $k$ fixed, $F(P)(x/t,1,k)$ possesses a maximum and a minimum on $R_1$. Further, since $P(z)$ is quasisymmetric on $\mathbb{R}$ (Lemma~\ref{lem: quasisymmetric-power-func}), then there exist $k_{1,1}\in (0,1)$ and $k_{2,1}>0$ such that
	\begin{equation*}
		F(P)(x/t,1,k_{2,1})< 0< F(P)(x/t,1,k_{1,1}),
	\end{equation*}
	for all $(x,t)\in R_1$. And thus,
	\begin{equation}\label{eq: max min F(P)}
		\max_{(x,t)\in R_1}F(P)(x/t,1,k_{2,1})< 0<\min_{(x,t)\in R_1} F(P)(x/t,1,k_{1,1}).
	\end{equation}
	
	We now prove the following assertion:
	For any $k$ fixed and $\epsilon>0$, there exists $t_0>0$ such that
	\begin{equation}\label{eq: F(r)(x,t,k)/t^p}
		\left| \frac{F(r)(x,t,k)}{t^p}\right| < \epsilon,
	\end{equation}
	for all $(x,t)\in R_1\cap\{(x,t)\in\mathbb{R}\times\mathbb{R}^+: t\geq t_0\}$. We have:
	\begin{equation*}
		\frac{F(r)(x,t,k)}{t^p} = \frac{r(x+t)}{t^p} - (1+k)\frac{r(x)}{t^p} + k\frac{r(x-t)}{t^p}.
	\end{equation*}
	Looking at the first term $r(x+t)/t^p$, we notice that for $\vert x+t\vert$ large, one can write:
	\begin{equation*}
		\left| \frac{r(x+t)}{t^p}\right| =  \frac{\vert r(x+t)\vert }{\vert x+t\vert ^p}\left| \frac{x}{t}+1\right|^p.
	\end{equation*}
	Since $x/t$ is bounded on $R_1$ and $r(x)=o(\vert x\vert^p)$ as $x\to\infty$, then there exists $n_1$ such that for any $(x,t)\in R_1$, $\vert x+t\vert > n_1$, one has:
	\begin{equation}\label{eq: r(x+t)/t^p}
		\left| \frac{r(x+t)}{t^p}\right| < \frac{\epsilon}{3}.
	\end{equation}
	Moreover, there exists $t_1>0$ such that \eqref{eq: r(x+t)/t^p} holds true for all $(x,t)\in R_1$, $\vert x+t\vert \leq n_1$, $t\geq t_1$. This implies that \eqref{eq: r(x+t)/t^p} also holds true for all $(x,t)\in R_1$, $t\geq t_1$. By similar arguments applied to the terms $r(x)/t^p$ and $r(x-t)/t^p$, we conclude that for any $k$ and $\epsilon>0$, there exists $t_0>0$ such that \eqref{eq: F(r)(x,t,k)/t^p} holds true for all $(x,t)\in R_1$, $t\geq t_0$. Eq.\,\eqref{eq: max min F(P)} and \eqref{eq: F(r)(x,t,k)/t^p} imply that there exists $t_0>0$ such that \eqref{eq: quasisymmetric operator} holds for $(k_1,k_2)=(k_{1,1},k_{2,1})$ and for all $(x,t)\in R_1$, $t\geq t_0$.

	Now set
	\begin{equation*}
		R_2:=\{(x,t)\in\mathbb{R}\times\mathbb{R}^+: 0\leq t\leq \epsilon_0 \vert x\vert\},
	\end{equation*}
	(see Figure~\ref{fig: RxR^+}). We first have:
	\begin{equation*}
		\partial_t F(\alpha)(x,t,k) = \alpha'(x+t) - k\alpha'(x-t).
	\end{equation*}
	By Assumption~\ref{assumption: alpha'(z)}, $\alpha'(z) = \vert z\vert^q + u(z)$, where $u(z)=o(\vert z\vert^q)$ as $z\to\pm\infty$, and thus,
	\begin{equation*}
		\frac{\alpha'(x+t)}{\alpha'(x-t)} = \frac{\vert x+t\vert^q+ u( x+t)}{\vert x-t\vert ^q + u( x-t)} = \left|\frac{x+t}{x-t}\right|^q \frac{1+\frac{u( x+t)}{\vert x+t\vert ^q}}{1+\frac{u( x-t)}{\vert x-t\vert ^q}}.
	\end{equation*}
	Notice that on $R_2$, $\big|(x+t)/(x-t)\big|$ is bounded by $(1-\epsilon_0)/(1+\epsilon_0)$ and $(1+\epsilon_0)/(1-\epsilon_0)$, and $(1+\frac{u( x+t)}{\vert x+t\vert ^q})/(1+\frac{u( x-t)}{\vert x-t\vert ^q})$ is close to $1$ as $\vert x\vert $ large enough. Hence, there exist $k_{1,2}\in (0,1)$, $k_{2,2}>1$, and $x_0$ large enough such that
	\begin{equation*}
		k_{1,2} < \frac{\alpha'(x+t)}{\alpha'(x-t)} < k_{2,2},\quad\text{or equivalently,}\quad \partial_tF(\alpha)(x,t,k_{2,2})< 0 < \partial_tF(\alpha)(x,t,k_{1,2}),
	\end{equation*}
	for all $(x,t)\in R_2$, $\vert x\vert \geq x_0$. This implies that for any $\vert x\vert \geq x_0$, $F(\alpha)(x,t,k_{1,2})$ and $F(\alpha)(x,t,k_{2,2})$, as functions of $t$, are strictly monotonic on $[0,\epsilon_0\vert x\vert]$. Since $F(\alpha)(x,0,k_{1,2})=F(\alpha)(x,0,k_{2,2})=0$, then \eqref{eq: quasisymmetric operator} holds true for $(k_1,k_2) = (k_{1,2},k_{2,2})$ and $(x,t)\in R_2, \vert x\vert \geq x_0$.
	
	We now consider the set
	\begin{equation*}
		(\mathbb{R}\times\mathbb{R}^+)\setminus [(R_1\cap\{(x,t)\in\mathbb{R}\times\mathbb{R}^+: t\geq t_0\}) \cup (R_2\cap\{(x,t)\in\mathbb{R}\times\mathbb{R}^+: \vert x\vert\geq x_0\})],
	\end{equation*}
	and let $R_3$ denote its closure (see Figure~\ref{fig: RxR^+}). Let
	\begin{equation*}
		k_{1,3}:= \min_{(x,t)\in R_3} \frac{\alpha(x+t)-\alpha(x)}{\alpha(x)-\alpha(x-t)}
		\quad\text{and}\quad
		k_{2,3}:=\max_{(x,t)\in R_3} \frac{\alpha(x+t)-\alpha(x)}{\alpha(x)-\alpha(x-t)},
	\end{equation*}
	which exist and are positive since $R_3$ is bounded and $\alpha(z)$ is differentiable on $\mathbb{R}$. Thus \eqref{eq: quasisymmetric operator} holds true for $(k_1,k_2) = (k_{1,3},k_{2,3})$ and $(x,t)\in R_3$.
	
	In conclusion, by choosing $k_1 = \min_{1\leq n\leq 3} k_{1,n}$ and $k_2 = \max_{1\leq n\leq 3} k_{1,n}$, we deduce that \eqref{eq: quasisymmetric operator} holds for $(x,t)\in\mathbb{R}\times\mathbb{R}^+$, i.e. $\alpha(z)$ is quasisymmetric on $\mathbb{R}$.
\end{proof}

The following lemma presents the existence of the BVP's solutions introduced at the beginning of Subsection~\ref{subsec: conformal welding}.

\begin{lemma}\label{lem: chi1-chi2-Jordan curve}
	Under Assumptions~\ref{assumption: monotonic}--\ref{assumption: alpha'(z)}, there exist a pair of conformal mappings $\chi_1$, $\chi_2$ and a Jordan curve $\mathcal{G}$ such that
	\begin{enumerate}[label=\textnormal{(\roman{*})},ref=\textnormal{(\roman{*})}]
		\item\label{item: chi1 conformal} $\chi_1$ maps conformally $\mathcal{H}^+$ onto $\mathcal{G}^+$;
		\item\label{item: chi2 conformal} $\chi_2$ maps conformally $\mathcal{H}^-$ onto $\mathcal{G}^-$;
		\item\label{item: BVP with shift} For all $z\in\mathbb{R}$,
		\begin{equation}\label{eq: boundary condition conformal welding}
			\chi_1^+(\alpha(z)) = \chi_2^-(z).
		\end{equation}
	\end{enumerate}
	We can further assume that
	\begin{enumerate}[label=\textnormal{(\roman{*})},ref=\textnormal{(\roman{*})},resume]
		\item\label{item: chi1(infty)=1} $\lim_{z\to \infty}\chi_1(z)=1$;
		\item\label{item: G symmetric} $\mathcal{G}$ is symmetric with respect to the real axis;
		\item\label{item: chi1-chi2-symmetric} $\chi_1(-\overline{z})=\overline{\chi_1(z)}$ for all $z\in \mathcal{H}^+$, $\chi_2(-\overline{z})=\overline{\chi_2(z)}$ for all $z\in \mathcal{H}^-$.
	\end{enumerate}
	In particular, the curve $\mathcal{G}$ is smooth everywhere except at $1$, where it admits a corner point with the angle
	\begin{equation*}
		\frac{\theta_1}{\theta_1+\theta_2}2\pi.
	\end{equation*}
\end{lemma}

\begin{proof}
	Since the shift function $\alpha(z)$ is quasisymmetric (Lemma~\ref{lem: alpha quasisymmetric}), then there exist $\widetilde{\chi}_1$, $\widetilde{\chi}_2$ and $\widetilde{\mathcal{G}}$ satisfying Items \ref{item: chi1 conformal}--\ref{item: BVP with shift} by \cite[Chap.\,II, Sec.\,7.5]{LeVi-73}.
	
	We now construct conformal mappings $\chi_1$ and $\chi_2$ satisfying Items~\ref{item: chi1 conformal}--\ref{item: chi1-chi2-symmetric} from $\widetilde{\chi}_1$ and $\widetilde{\chi}_2$. By the construction of $\widetilde{\chi}_1$ and $\widetilde{\chi}_2$, the union of $\widetilde{\chi}_1(\{a+bi:a,b>0\})$, $\widetilde{\chi}_2(\{a-bi:a,b>0\})$ and $\widetilde{\chi}_1^+(\mathbb{R}_+)$ is indeed a domain and denoted by $\mathcal{D}^+$. Let $a=\widetilde{\chi}_1^+(0)$, $b=\widetilde{\chi}_1^+(\infty)$, and $\zeta$ be a conformal mapping from $\mathcal{D}^+$ onto $\mathcal{H}^+$ such that $\zeta^+(a)\not=\infty$ and $\zeta^+(b)=1$. Such a mapping $\zeta$ exists and can be constructed as follows:
	\begin{itemize}
		\item Let $\delta$ be a conformal mapping from $\mathcal{D}^+$ onto $\mathcal{C}^+$ such that $\delta^+(a),\delta^+(b)\not=1$ and $1\in \delta^+\circ\widetilde{\chi}_2^+(i\mathbb{R}_-)$;
		\item Recall that $\phi(z)$ defined in \eqref{eq: phi} maps conformally $\mathcal{C}^+$ onto $\mathcal{H}^+$;
		\item Define $\zeta(z):=\phi\circ\delta(z)-\phi\circ\delta^+(b)+1$.
	\end{itemize}
	
	Now define
	\begin{align*}
		\chi_1(z):=\begin{cases}
		\zeta\circ\widetilde{\chi}_1(z),&\quad\text{if } z\in\{a+bi:a,b>0\},\\
		\zeta^+\circ\widetilde{\chi}_1(z),&\quad\text{if } z\in i\mathbb{R}_+,\\[1.5pt]
		\overline{\zeta\circ\widetilde{\chi}_1(-\overline{z})},&\quad\text{if } z\in\{-a+bi:a,b>0\},
		\end{cases}
		\\
		\chi_2(z):=\begin{cases}
		\zeta\circ\widetilde{\chi}_2(z),&\quad\text{if } z\in\{a-bi:a,b>0\},\\
		\zeta^+\circ\widetilde{\chi}_2(z),&\quad\text{if } z\in i\mathbb{R}_-,\\[1.5pt]
		\overline{\zeta\circ\widetilde{\chi}_2(-\overline{z})},&\quad\text{if } z\in\{-a-bi:a,b>0\}.
		\end{cases}
	\end{align*} 
	Put $\mathcal{G}=\chi_1^+(\mathbb{R})$. By Schwarz reflection principle, $\chi_1$ and $\chi_2$ are analytic respectively on $\mathcal{H}^+$ and $\mathcal{H}^-$, and $\mathcal{G}$ is a Jordan curve. It is easy to check that $\chi_1$, $\chi_2$ and $\mathcal{G}$ satisfy Items~\ref{item: chi1 conformal}--\ref{item: chi1-chi2-symmetric}. 
	
	We now show the smoothness of $\mathcal{G}$.
	Let $z_0$ be any point on $\mathbb{R}\cup\{\infty\}$ and $\lambda$ denote the angle of $\mathcal{G}$ at $\chi_2^-(z)$. Consider first the case where $z_0\in\mathbb{R}$. Recall that $\alpha(z)$ is analytic and possesses non-vanishing derivatives on $\mathbb{R}$, then by \cite[Thm.\,3.11]{Po-92}, one has:
	\begin{align*}
		&\chi_1^+(\alpha(z)) = \chi_1^+(\alpha(z_0)) +\text{const}\cdot(z-z_0)^{\lambda/\pi}+ o\big( (z-z_0)^{\lambda/\pi} \big),\\
		&\chi_2^-(z) = \chi_2^-(z_0) +\text{const}\cdot(z-z_0)^{(2\pi-\lambda)/\pi} + o\big( (z-z_0)^{(2\pi-\lambda)/\pi} \big),
	\end{align*}
	as $z\in\mathbb{R},z\to z_0$. Item~\ref{item: BVP with shift} implies that $\lambda/\pi = (2\pi-\lambda)/\pi$, and thus $\lambda = \pi$, i.e., $\mathcal{G}$ is smooth at $\chi_2^-(z_0)$.
	
	We now consider the case where $z_0=\infty$. Notice that $\alpha(z)= z^{\theta_2/\theta_1} + o(z^{\theta_2/\theta_1})$ as $z\to\infty$, then also by \cite[Thm.\,3.11]{Po-92}, one has:
	\begin{align*}
		&\chi_1^+(\alpha(z)) = 1 +\frac{\text{const}}{z^{(\lambda/\pi)(\theta_2/\theta_1)}} + o\left(\frac{1}{z^{(\lambda/\pi)(\theta_2/\theta_1)}}\right),\\
		&\chi_2^-(z) = 1 +\frac{\text{const}}{z^{(2\pi-\lambda)/\pi}} + o\left(\frac{1}{ z^{(2\pi-\lambda)/\pi} }\right),
	\end{align*}
	as $z\in\mathbb{R},z\to\infty$. Item~\ref{item: BVP with shift} implies that
	\begin{equation*}
		\frac{\lambda}{\pi}\frac{\theta_2}{\theta_1} = \frac{2\pi-\lambda}{\pi}, \text{ i.e. } \lambda = \frac{\theta_1}{\theta_1+\theta_2} 2\pi.
	\end{equation*}
	The proof is then complete.
\end{proof}

At this point, we already have enough material to build a BVP concerning the generating function $H(x,y)$ in \eqref{eq: H(x,y)-def}, with a boundary condition on the Jordan curve $\mathcal{G}$. However, the corner point $1$ of $\mathcal{G}$ may provoke some difficulties on the regularity of solutions around this point. To overcome such problems, we want to transform the problem to a new one with a boundary condition on an infinite curve, where the non-regularity problem may happen at infinity. Therefore, we introduce the following mappings:
\begin{equation}
	\psi_1:=-i\phi\circ\chi_1\circ\phi\circ\pi_1 \quad\text{and}\quad \psi_2:=-i\phi\circ\chi_2\circ(-\phi )\circ\pi_2,
\end{equation}
where $\phi$ is defined in \eqref{eq: phi}. Let $\mathcal{L}$, $\mathcal{L}^+$ and $\mathcal{L}^-$ respectively denote $\phi(\mathcal{G})$, $\phi(\mathcal{G}^+)$ and $\phi(\mathcal{G}^-)$. By this construction, $\mathcal{L}$ is an open infinite curve, and $\psi_1$ (resp.~$\psi_2$) maps conformally $\mathcal{S}_1^+$ (resp.~$\mathcal{S}_2^+$) onto $\mathcal{L}^+$ (resp.~$\mathcal{L}^-$).

The following lemma presents some crucial properties of $\psi_1$ and $\psi_2$.
\begin{lemma}\label{lem: psi1-psi2}
	We have:
	\begin{enumerate}[label=\textnormal{(\roman{*})},ref=\textnormal{(\roman{*})}]
		\item\label{item: psi_1, psi_2 symmetric} $\psi_1(\overline{z}) = \overline{\psi_1(z)}$ for all $z\in\mathcal{S}_1^+$, and $\psi_2(\overline{z}) = \overline{\psi_2(z)}$ for all $z\in\mathcal{S}_2^+$;
		
		\item\label{item: psi1 x psi2}
		\begin{equation*}
			(\psi_1^{-1})^+\times (\psi_2^{-1})^- (\mathcal{L}) =\begin{cases}
			\mathcal{K}\setminus\{(0,0)\},&\quad\text{if } p_{1,1}=0,p_{0,1}\not=p_{1,0},\\
			\mathcal{K},&\quad\text{otherwise};
			\end{cases}
		\end{equation*}
		
		\item\label{item: asymptotic psi1 psi2} The asymptotic behaviors of $\psi_1$ and $\psi_2$ around $1$ are
		\begin{equation*}
			\psi_1(x) \sim \frac{\text{const}}{(1-x)^{2\pi/(\theta_1+\theta_2)}} \quad\text{and}\quad
			\psi_2(y) \sim \frac{\text{const}}{(1-y)^{2\pi/(\theta_1+\theta_2)}}.
		\end{equation*}
	\end{enumerate}
\end{lemma}

\begin{proof}
	We first prove Item~\ref{item: psi_1, psi_2 symmetric}. For any $z\in\mathcal{S}_1^+$, we have:
	\begin{align*}
		\psi_1(\overline{z}) &= -i\phi\circ\chi_1\circ\phi\circ\pi_1(\overline{z})=-i\phi\circ\chi_1\circ\phi\circ\overline{\pi_1(z)} = -i\phi\circ\chi_1\circ\overline{(-\phi)\circ\pi_1(z)} \\
		&= -i\phi\circ\overline{\chi_1\circ\phi\circ\pi_1(z)} = \overline{-i\phi\circ \chi_1\circ\phi\circ\pi_1(z)} = \overline{\psi_1(z)}.
	\end{align*}
	Similarly, $\psi_2(\overline{z}) = \overline{\psi_2(z)}$ for any $z\in\mathcal{S}_2^+$.
	
	Item~\ref{item: psi1 x psi2} is a direct consequence of Eq.~\eqref{eq: boundary condition conformal welding}. We move to the proof of Item~\ref{item: asymptotic psi1 psi2}. Recall that
	\begin{align*}
		& \pi_1(z) = 1+ \text{const}\cdot(z-1)^{\pi/\theta_1} + o\big( (z-1)^{\pi/\theta_1} \big),\quad\text{as}\quad z\to 1 ,\\
		& \chi_1(z) =  1 + \frac{\text{const}}{z^{2\theta_1/(\theta_1+\theta_2)}} +o\left(\frac{1}{z^{2\theta_1/(\theta_1+\theta_2)}}\right), \quad \text{as}\quad z\to\infty,\\
		& \phi(z) = -i\frac{z+1}{z-1}.
	\end{align*}
	We then have:
	\begin{align*}
		& \phi\circ \pi_1 (z) = -i + \frac{\text{const}}{(z-1)^{\pi/\theta_1}} + o\left(\frac{1}{(z-1)^{\pi/\theta_1}}\right),\\
		& \chi_1\circ\phi\circ \pi_1(z) = 1+\text{const}\cdot (z-1)^{2\pi/(\theta_1+\theta_2)} + o\left( (z-1)^{2\pi/(\theta_1+\theta_2)} \right),\\
		& \psi_1(z)=-i\phi\circ\chi_1\circ\phi\circ\pi_1(z) = \frac{\text{const}}{(z-1)^{2\pi/(\theta_1+\theta_2)}} + o\left( \frac{1}{(z-1)^{2\pi/(\theta_1+\theta_2)}}  \right),
	\end{align*}
	as $z\to 1$. The asymptotic behavior of $\psi_2$ around $1$ is deduced similarly. The proof is then complete.
\end{proof}

\subsection*{On the regularity and analyticity of $\psi_1$ and $\psi_2$ around zero}
Since the extensions of $\psi_1$ and $\psi_2$ around zero will play a major role in defining the bivariate series $H(x,y)$, we spend a small part in the article to discuss it.

We first recap some information about the point zero. In the case $p_{1,1}\neq 0$, since $0\in\mathcal{S}_1^+,\mathcal{S}_2^+$ (see Lemma~\ref{lem: S1-S2}), then $\psi_1$ and $\psi_2$ are obviously analytic around $0$. However, in the case $p_{1,1}=0$ and $p_{1,0}=p_{0,1}$, because $0\in\mathcal{S}_1,\mathcal{S}_2$, we only define the limit values $\psi_1^+(0)$ and $\psi_2^+(0)$. And in the case $p_{1,1}=0$ and $p_{0,1}>p_{1,0}$, $0\in\mathcal{S}_1^+$ but $0\notin\mathcal{S}_2^+\cup\mathcal{S}_2$, $\psi_1$ is analytic around $0$ but $\psi_2(0)$ has not been defined yet.

The following lemma presents analytic continuations of $\psi_1$ and $\psi_2$ around $0$.

\begin{proposition}\label{prop: psi_1,2(0)}
	In the case $p_{1,1}=0$ and $p_{1,0}=p_{0,1}$, $\psi_1$ and $\psi_2$ can be extended analytically around $0$, such that $\psi_1(0)=\psi_2(0)\in\mathcal{L}\cap\mathbb{R}$ and $\psi_1'(0),\psi_2'(0)\neq 0$.
	
	In the case $p_{1,1}=0$ and $p_{0,1}>p_{1,0}$, $\psi_2$ can be extended analytically to $0$ along the real line, such that $\psi_1(0)=\psi_2(0)\in\mathcal{L}^+\cap\mathbb{R}$ and $\psi_2'(0)\neq 0$.
\end{proposition}

\begin{proof}
	Consider first the case $p_{1,1}=0$ and $p_{1,0}=p_{0,1}$. Let $V$ be an open and small enough neighborhood of $0$ such that all the assertions in Lemma~\ref{lem: solution-neigborhood-S2} hold true for $Y(x)$ on $V$. We then introduce the function
	\begin{equation*}
		\omega(x):=\begin{cases}
			\psi_1(x), & x\in V\cap\mathcal{S}_1^+,\\
			\psi_2(Y(x)), & x\in V\cap\mathcal{S}_1^-.
		\end{cases}
	\end{equation*}
	Since $\psi_1$ is well defined in $\mathcal{S}_1^+$, $\psi_2$ is well defined in $\mathcal{S}_2^+$, and $Y(V\cap\mathcal{S}_1^-)$ is a subset of $ \mathcal{S}_2^+$ (Lemma~\ref{lem: solution-neigborhood-S2}), then $\omega(x)$ is also well defined. Further, $\omega(x)$ is sectionally analytic on $V\cap\mathcal{S}_1^+$, $V\cap\mathcal{S}_2^+$ and continuous on $V$ (since $\psi^+_1(x)=\psi_2^+(Y(x))$ for all $x\in V\cap\mathcal{S}_1$), then by Morera's theorem, $\omega(x)$ is analytic on $V$. In other words, $\psi_2(Y(x))$ is an analytic continuation of $\psi_1(x)$ around $0$. Since $\mathcal{S}_1$ and $\mathcal{L}$ are smooth respectively at $0$ and $\psi_1(0)$, then $\psi_1'(0)\neq 0$ by \cite[Thm.\,3.9]{Po-92}.  We have analogous arguments for the analyticity of $\psi_2$ around $0$.
	
	We move to the case $p_{1,1}=0$ and $p_{0,1}>p_{1,0}$, where $0\in\mathcal{S}_1^+,\mathcal{S}_2^-$. Let $y_0$ denote the intersection of $\mathbb{R}$ and $\mathcal{S}_2\setminus\{1\}$. By the same arguments as in the previous case, $\psi_1(X(y))$ is an analytic continuation of $\psi_2(y)$ around $y_0$. Since the Taylor expansion of $X(y)$ at $y_0$ possesses all real coefficients, then the extension of $X(y)$ along the segment $[0,y_0]$ is also real and cannot escape $\mathcal{S}_1^+$ (because $\vert X(y)\vert <\vert y\vert $ for all $y\in(0,y_0)$ close enough to $y_0$ by Lemma~\ref{lem: X(y) around 0}). Thus, $\psi_1(X(y))$ is well defined on $[0,y_0]$ and forms an analytic continuation of $\psi_2(y)$. Hence, $\psi_2(0) = \psi_1(X(0))=\psi_1(0)\in\mathcal{L}^+$ and $\psi_2'(0) = \psi_1'(X(0))X'(0)\neq 0$.
\end{proof}

\section{Proof of the main theorems}\label{sec: proof-main-thms}

In this section, we will outline the proof of the main theorems, but not present them in detail since the strategy is similar as in the symmetric case of \cite{HoRaTa-20}. From the functional equation~\eqref{eq: functional-eq} and the conformal mappings in Section~\ref{sec: conformal welding}, one can construct a boundary value problem (Lemma~\ref{lem: BVP}) whose polynomial solutions form the class of harmonic functions $(h_n)_{n\in\mathbb{N}}$ in Theorem~\ref{thm: main_intro-1}. Such harmonic functions are shown to satisfy the following features: first, at any point $(i,j)\in\mathbb{N}^2$, for all $n$ large enough, $h_n(i,j)=0$ (Lemma~\ref{lem: h_n(i,j)=0}); second, $\{h_n(\cdot,1)\}_{n\geq 1}$ forms a basis of the space $\{h(\cdot,1): h \text{ is harmonic}\}$ (Lemma~\ref{lem: unique sequence}). Consequently, any infinite linear combination $\sum_{n\geq 1} a_n h_n$ with $\{a_n\}_{n\geq 1}\subset\mathbb{R}$, when evaluated at any point $(i,j)\in\mathbb{N}^2$, possesses a finite value. On the other hand, any discrete harmonic function can be expressed uniquely as an infinite sum $\sum_{n\geq 1} a_n h_n$ with $\{a_n\}_{n\geq 1}\subset\mathbb{R}$ (Theorem~\ref{thm: main_intro-2}).

We first set
\begin{equation}\label{eq: F(t) sectinal}
	F(t):=\begin{cases}
		K{\cdot} H(\psi_1^{-1}(t),0) - \frac{1}{2}K{\cdot} H(0,0), & \text{if } t\in\mathcal{L}^+,\\
		-K{\cdot} H(0,\psi_2^{-1}(t)) + \frac{1}{2}K{\cdot} H(0,0), & \text{if } t\in\mathcal{L}^-.\\
	\end{cases}
\end{equation}

\begin{figure}[t]
	\centering
	\begin{tikzpicture}
	\node at (0,0) {\includegraphics[scale=0.2]{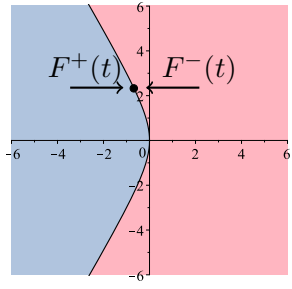}};
	
	\draw[->,thick] (-1,0.7) -- (-0.3,0.7);
	\draw[->,thick] (0.7,0.7) -- (0,0.7);
	\node at (-0.8,0.95) {$F^+(t)$};
	\node at (0.7,0.95) {$F^-(t)$};
	\end{tikzpicture}
	\caption{Description of the BVP in Lemma~\ref{lem: BVP}: If $F$ is analytic on $\mathcal{L}^+$ (blue domain) and $\mathcal{L}^-$ (red domain), and continuous on $\mathcal{L}$ (i.e., $F^+(t)=F^-(t)$ for all $t\in\mathcal{L}$), then $F$ is analytic on $\mathbb{C}$.}
	\label{fig: BVP}
\end{figure}

The following lemma is a direct consequence of the functional equation \eqref{eq: functional-eq}.
\begin{lemma}\label{lem: BVP}
	Assume \ref{assumption: H1}--\ref{assumption: H5}, \ref{assumption: monotonic}--\ref{assumption: alpha'(z)}, and that $H(x,0)$ and $H(0,y)$ have radii of convergence equal to or greater than $1$. Then $F$ defined in \eqref{eq: F(t) sectinal} satisfies the following BVP (see Figure~\ref{fig: BVP}):
	\begin{enumerate}[label=\textnormal{(\roman{*})},ref=\textnormal{(\roman{*})}]

	\item $F$ is analytic on $\mathcal{L}^+$ and admits a continuous extension $F^+$ to $\mathcal{L}$;
	
	\item $F$ is analytic on $\mathcal{L}^-$ and admits a continuous extension $F^-$ to $\mathcal{L}$;
	
	\item For all $t\in\mathcal{L}$,
	\begin{equation}\label{eq: F^+ = F^-}
		F^+(t)-F^-(t)=0.
	\end{equation}
	\end{enumerate}

	Consequently, if $F$ is bounded at infinity, then $F$ is the zero function and so is the associated harmonic function. On the other hand, if $F$ has a pole of order $n>0$ at infinity, then $F$ is a polynomial of degree $n$ satisfying:
	\begin{equation}\label{eq: condition of F(t)}
		F(\psi_1(0))=-F(\psi_2(0))=\frac{K\cdot H(0,0)}{2} = -\frac{1}{2}p_{1,1}h(1,1).
	\end{equation}
\end{lemma}

\begin{proof}
	The general solutions of the BVP are in fact entire functions, and under the conditions on the growth at infinity, one can infer these solutions as constant functions or polynomials (see more details in \cite[Lem.\,9 and Cor.\,10]{HoRaTa-20}). If $F$ is a constant function (resp.~a polynomial), evaluating $F$ at $\psi_1(0)$ and $\psi_2(0)$ yields that $F$ is the zero function (resp.~$F$ satisfies Eq.~\eqref{eq: condition of F(t)}).
\end{proof}

\begin{proof}[Proof of Theorem~\ref{thm: main_intro-1}]
	As one can easily verify, the polynomials introduced in \eqref{eq: family-polynomials} are solutions of the BVP in Lemma~\ref{lem: BVP} and form a basis of $\mathbb{R}[X]$. Consequently, the functions $H_n(x,y)$ defined in \eqref{eq: H_n(x,y)} satisfy the functional equation \eqref{eq: functional-eq}. The remaining work is to prove that these $H_n(x,y)$ are bivariate power series around $(0,0)$, which always holds true in the case $p_{1,1}\not= 0$ (because $K(0,0)\not=0$ in this case). In the case $p_{1,1}=0$, since $\psi_1(0)=\psi_2(0)$ by Prop.~\ref{prop: psi_1,2(0)}, then $H_n(x,y)$ can be rewritten as
	\begin{equation*}
		H_n(x,y) = \frac{\psi_1(x)-\psi_2(y)}{K(x,y)}\sum_{k=0}^{n-1}(\psi_1(x)-\psi_1(0))^k (\psi_2(y)-\psi_1(0))^{n-1-k}.
	\end{equation*}
	Since $p_{0,1}$ and $p_{1,0}$ cannot simultaneously vanish, we can assume further $0\not= p_{0,1}\geq p_{1,0}$. Hence, on a neighborhood of $(0,0)$, $\partial_x K(x,y)\not=0$, which implies that $X(y)$ is the unique solution of $K(x,y)$ (as a function of $x$) around $0$ by the implicit function theorem. Moreover, since $\partial_x(\psi_1(x)-\psi_2(y))\not=0$ on a neighborhood of $(0,0)$, and $\psi_1(X(y))=\psi_2(y)$ around $0$ (recall from the proof of Prop.~\ref{prop: psi_1,2(0)} that $\psi_1(X(y))$ is an analytic continuation of $\psi_2(y)$ around $0$), then $X(y)$ is the unique solution of $\big(\psi_1(x)-\psi_2(y)\big)$ (as functions of $x$) around $0$ by the implicit function theorem. Thanks to the Weierstrass preparation theorem for analytic functions in several variables (see \cite[Chap.~2, Sec.~B, Thm.~2]{GuRo-65}), we can write:
	\begin{equation*}
		\frac{\psi_1(x)-\psi_2(y)}{K(x,y)} = \frac{u(x,y)(x-X(y))}{v(x,y)(x-X(y))} = \frac{u(x,y)}{v(x,y)},
	\end{equation*}
	where $u(x,y)$ and $v(x,y)$ are analytic around $(0,0)$ and not vanishing at $(0,0)$. Thus, $(\psi_1(x)-\psi_2(y))/K(x,y)$ is analytic around $(0,0)$, and so is $H_n(x,y)$.
	
	We move to the properties of $H_n(x,y)$. Since $K(x,y)$ has no solution in $\mathcal{S}_1^+\times\mathcal{S}_2^+$ (Prop.\,\ref{prop: extend-solutions-intro}\ref{item: Prop1-item2}), then $H_n(x,y)$ is analytic in this domain. And lastly, one can compare the Laplace transform of $h_{n,\sigma}(x,y)$, which can be explicitly computed, with the Laplace transform of $h_n(x,y)$, which can be approximated through $H_n(x,y)$ (see \cite[Sec.~4.2, Prop.~11]{HoRaTa-20}).
\end{proof}

Although the proof of Theorem~\ref{thm: main_intro-1} shares the same framework as that  of \cite[Thm.\,1]{HoRaTa-20} (the symmetric case), there are some major differences that we want to emphasize:
\begin{itemize}
	\item In the symmetric case of \cite{HoRaTa-20}, the conformal mappings $\psi_1$ and $\psi_2$ can always be constructed. Further, the shift function $\alpha$ is the identity, we thus do not need the step of solving the conformal welding problem with quasisymmetric shift (Subsection~\ref{subsec: conformal welding}). In the non-symmetric case, we have to restrict the analysis under Assumptions~\ref{assumption: monotonic}--\ref{assumption: alpha'(z)} to ensure the existence of such mappings;
	
	\item As presented above, the proof is mostly based on the behaviors of $\psi_1$, $\psi_2$, $K(x,y)$ around $(0,0)$. In the case $p_{1,1}=0$, $p_{0,1}>p_{1,0}$, we have $0\in\mathcal{S}_2^-$, which does not appear in the symmetric case. We therefore need analytic continuation arguments (Prop.\,\ref{prop: extend-solutions-intro}\ref{item: Prop1-item1}, Prop.\,\ref{prop: psi_1,2(0)}) to show that $\psi_2(z)$ is well defined and analytic around $0$, and $K(x,y)$ has no solution in $\mathcal{S}_1^+\times\mathcal{S}_2^+$ (Prop.\,\ref{prop: extend-solutions-intro}\ref{item: Prop1-item2}).
\end{itemize}

In order to prove Theorem~\ref{thm: main_intro-2}, we recall from \cite{HoRaTa-20} the following two lemmas presenting important properties of the harmonic functions $\{h_n\}_{n\geq 1}$ in Theorem~\ref{thm: main_intro-1}.

\begin{lemma}\label{lem: h_n(i,j)=0}
	The harmonic function $h_n$ defined in Theorem~\ref{thm: main_intro-1} satisfies the following assertions:
	\begin{itemize}
		\item In the case $p_{1,1}=0$,
			\begin{enumerate}[label=\textnormal{(\roman{*})},ref=\textnormal{(\roman{*})}]
				\item For all $i,j\geq 1$ such that $i+j\leq n$, we have $h_n(i,j)=0$;
				\item For all $i,j\geq 1$ such that $i+j= n+1$, we have $h_n(i,j)\not= 0$.
			\end{enumerate} 
		\item In the case $p_{1,1}\not= 0$,
			\begin{enumerate}[label=\textnormal{(\roman{*})},ref=\textnormal{(\roman{*})}]
				\item For all $1\leq i,j\leq \lfloor n/2 \rfloor$, we have $h_n(i,j)=0$;
				\item $h_n(\lfloor n/2 +1 \rfloor,1 ), h_n(1,\lfloor n/2 +1 \rfloor )\neq 0$.
			\end{enumerate}
	\end{itemize}
\end{lemma}

\begin{lemma}\label{lem: unique sequence}
	For any discrete harmonic function $h(i,j)$, we have: 
	\begin{itemize}
		\item In the case $p_{1,1}=0$, there exist unique sequences $\{a_n\}_{n\geq 1},\{b_n\}_{n\geq 1}\subset\mathbb{R}$ such that:
		\begin{equation*}\label{eq: unique sequence p_11=0}
			h(i,1) = \sum_{n\geq 1} a_nh_n(i,1)\quad\text{and}\quad h(1,i) = \sum_{n\geq 1}b_n h_n(1,i),\quad\text{for all } i\in \mathbb{N};
		\end{equation*}
		\item  In the case $p_{1,1}\not= 0$, there exists a unique sequence $\{a_n\}_{n\geq 1}\subset\mathbb{R}$ such that:
		\begin{equation*}\label{eq: unique sequence p_11not0}
			h(i,1) = \sum_{n\geq 1} a_nh_n(i,1)\quad\text{and}\quad h(1,i) = \sum_{n\geq 1}a_n h_n(1,i),\quad\text{for all } i\in \mathbb{N},
		\end{equation*}
	\end{itemize}
	where $\{h_n\}_{n\geq 1}$ are defined in Theorem~\ref{thm: main_intro-1}.
\end{lemma}

By Lemma~\ref{lem: h_n(i,j)=0}, the linear systems of equations in Lemma~\ref{lem: unique sequence} are solvable and have unique solutions. We do not mention its detailed proof since it is similar as in \cite[Lem.\,14, 15]{HoRaTa-20}.

\begin{proof}[Proof of Theorem~\ref{thm: main_intro-2}]
	Observe first that any infinite sum $\sum_{n\geq 1} a_n h_n$, with $\{a_n\}_{n\geq 1}\subset \mathbb{R}$, evaluated at any point $(i,j)\in\mathbb{N}^2$, has a finite value, since $h_n(i,j)=0$ for all $n\geq i+j$ by Lemma~\ref{lem: h_n(i,j)=0}. Thus, the mapping $\Phi$ in Theorem~\ref{thm: main_intro-2} is well defined. The injectivity of $\Phi$ is implied by Lemma~\ref{lem: unique sequence}.
	
	Now let $h(i,j)$ be any harmonic function and $H(x,y)$ be its generating function. To show the surjectivity of $\Phi$, we prove that there exists a sequence $\{a_n\}_{n\geq 1}\subset\mathbb{R}$ such that $h(i,j)=\sum_{n\geq 1}a_nh_n(i,j)$ for all $(i,j)\in\mathbb{N}^2$.
	
	In the case $p_{1,1}\not=0$, let $\{a_n\}_{n\geq 1}$ be the sequence satisfying Lemma~\ref{lem: unique sequence} and $\widetilde{H}(x,y)$ be the generating function of $\sum_{n\geq 1}a_nh_n(i,j)$. We recall that $H(x,y)$ and $\widetilde{H}(x,y)$ are determined by Eq.~\eqref{eq: functional-eq}:
	\begin{align*}
		&H(x,y)=\frac{K\cdot H (x,0) + K\cdot H(0,y)- K\cdot H(0,0)}{K(x,y)},\\
		&\widetilde{H}(x,y)=\frac{K\cdot \widetilde{H} (x,0) + K\cdot \widetilde{H}(0,y)- K\cdot \widetilde{H}(0,0)}{K(x,y)}.
	\end{align*}
	Moreover, the quotients are well defined since $K(0,0)\not=0$. Lemma~\ref{lem: unique sequence} implies that $H(x,0)$ and $\widetilde{H}(x,0)$ coincide and so do $H(0,y)$ and $\widetilde{H}(0,y)$. Thus, $H(x,y)$ and $\widetilde{H}(x,y)$ also coincide and $h\equiv\sum_{n\geq 1}a_nh_n$.
	
	We move to the case $p_{1,1}=0$, and further assume $0\not=p_{0,1}\geq p_{1,0}$. Then the function $X(y)$ is well defined around $0$ (Prop.~\ref{prop: extend-solutions-intro}\ref{item: Prop1-item1}) and by the Weierstrass preparation theorem in the ring of formal power series $\mathbb{R}[[x,y]]$, one can write:
	\begin{equation*}
		K(x,y) = u(x,y)(x-X(y)),
	\end{equation*}
	where $u$ is invertible in $\mathbb{R}[[x,y]]$. By Lemma~\ref{lem: unique sequence}, there exists a unique sequence $\{a_n\}_{n\geq 1}$ such that $h(i,1)=\sum_{n\geq 1}a_nh_n(i,1)$ for all $i\in\mathbb{N}$. Let $\widetilde{H}(x,y)$ denote the generating function of $\sum_{n\geq 1}a_nh_n(i,1)$, we then have $H(x,0)=\widetilde{H}(x,0)$. Consider the division of $K\cdot H(x,0)$ (resp.~$K\cdot \widetilde{H}(x,0)$) by $(x-X(y))$. Applying the Weierstrass division theorem to $K(x,y)$ in the ring $\mathbb{R}[[x,y]]$, there exists a unique formal series $G\in\mathbb{R}[[y]]$ (resp.~$\widetilde{G}\in\mathbb{R}[[y]]$) such that $ K\cdot H(x,0) + G(y)$ (resp.~$ K\cdot \widetilde{H}(x,0) + \widetilde{G}(y)$) is divisible by $x-X(y)$. Since $H(x,0)$ and $\widetilde{H}(x,0)$ coincide, then so do $G(y)$ and $\widetilde{G}(y)$. And thus, the quotients of $ K\cdot H(x,0) + G(y)$ and $ K\cdot \widetilde{H}(x,0) + \widetilde{G}(y)$ by $K(x,y)$, which are respectively $H(x,y)$ and $\widetilde{H}(x,y)$, also coincide. This shows the surjectivity of $\Phi$ in the case $p_{1,1}=0$.
\end{proof}

\section{Small jump random walks}
In this section, we take a closer look to models associated with small jump random walks, namely, $p_{i,j}=0$ if $\vert i\vert \geq 2$ or $\vert j\vert \geq 2$. Our goal is to derive explicit expressions of the conformal mappings $\psi_1$ and $\psi_2$ in Theorem~\ref{thm: main_intro-1}. Such models have been carefully studied in \cite[Sec.\,6.5]{FaIaMa-17}, where the kernel's zero set can be parametrized through a rational uniformization. We first recall some facts about the uniformization, then construct conformal mappings. At the end of the section, we also give a concrete example where the shift function and  all the conformal mappings in Section~\ref{sec: conformal welding} can be explicitly expressed.

\begin{figure}[t]
	\centering
	\begin{tikzpicture}
		\node at (0,2) {\includegraphics[scale=0.2]{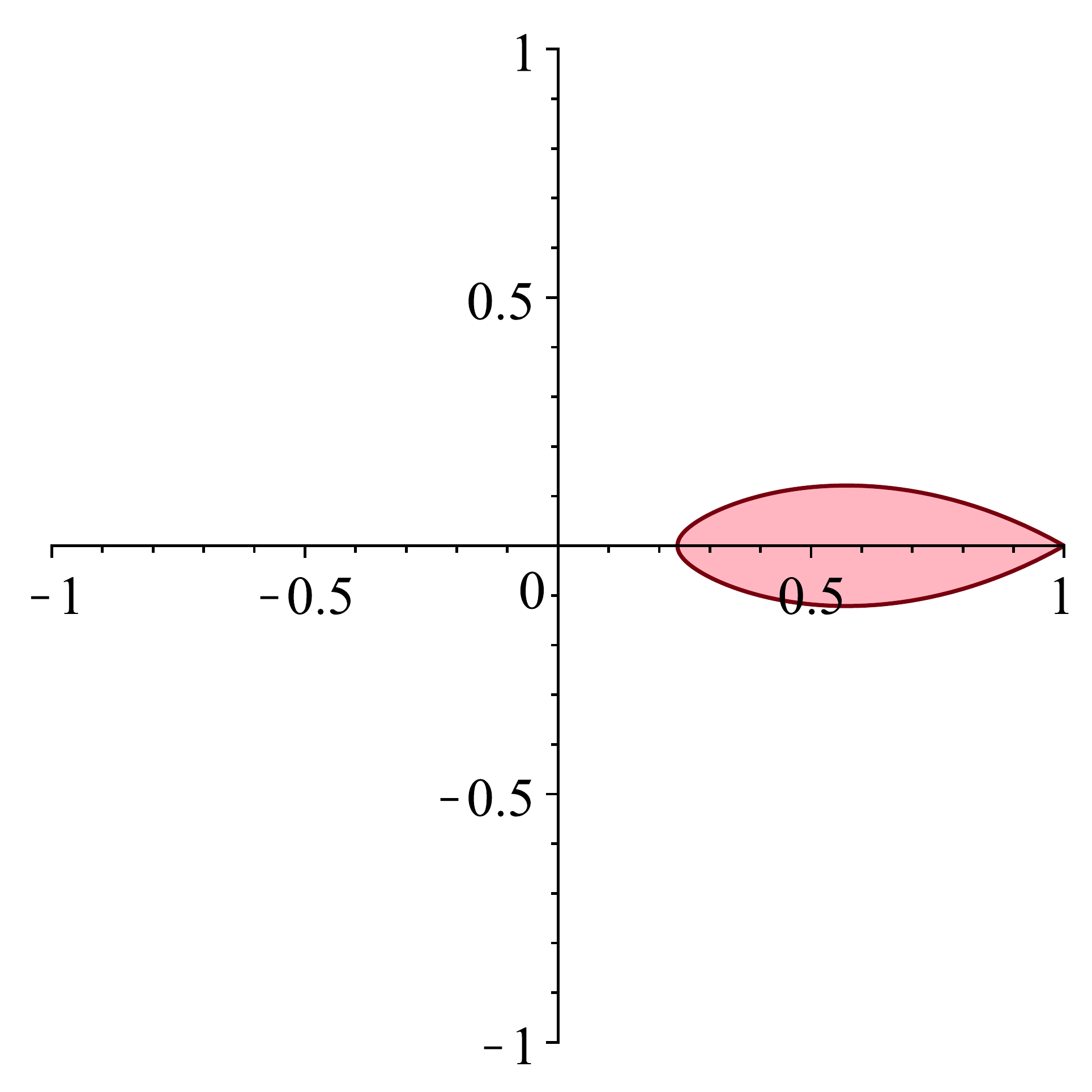}};
		\node at (0,-2) {\includegraphics[scale=0.2]{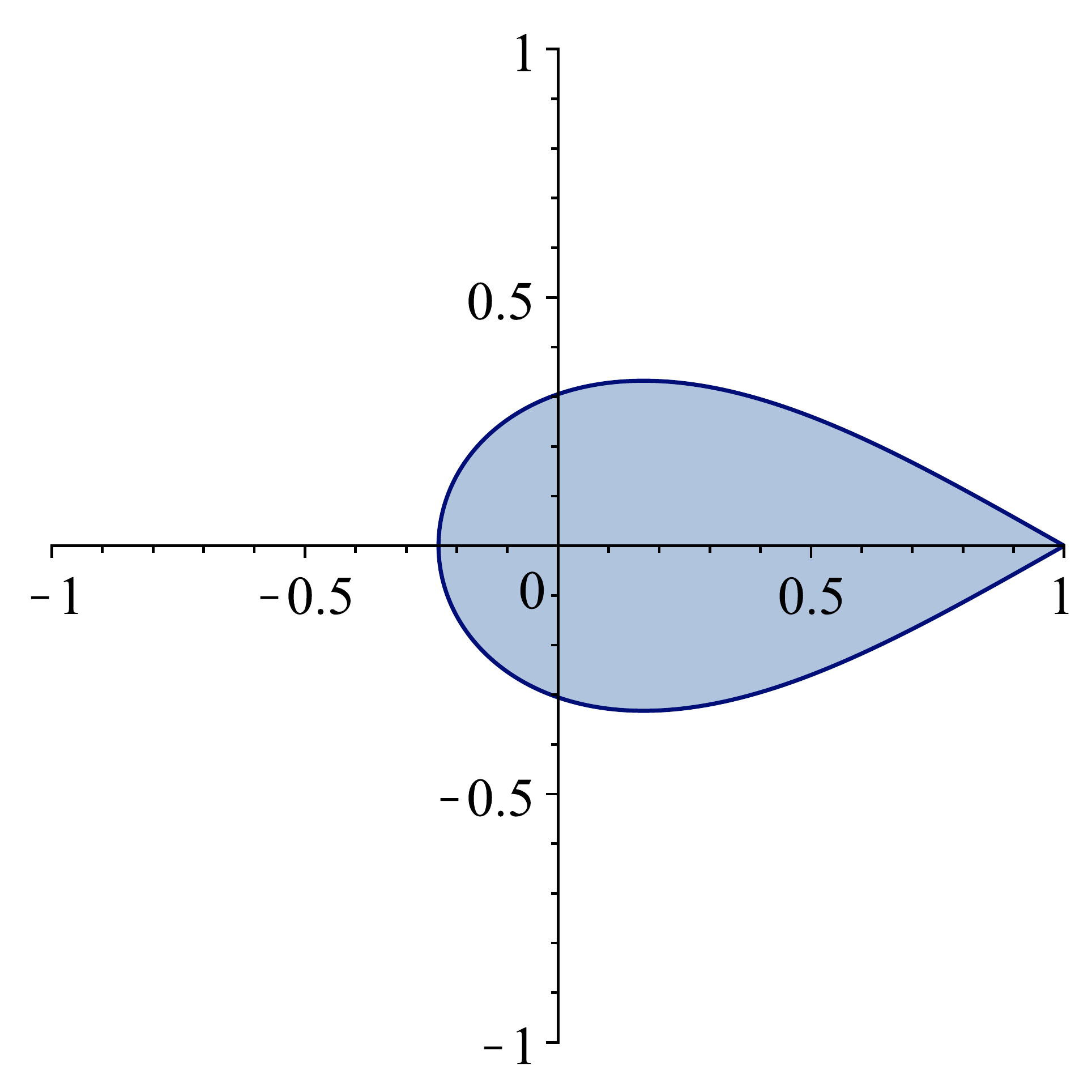}};
		\node at (5,0) {\includegraphics[scale=0.2]{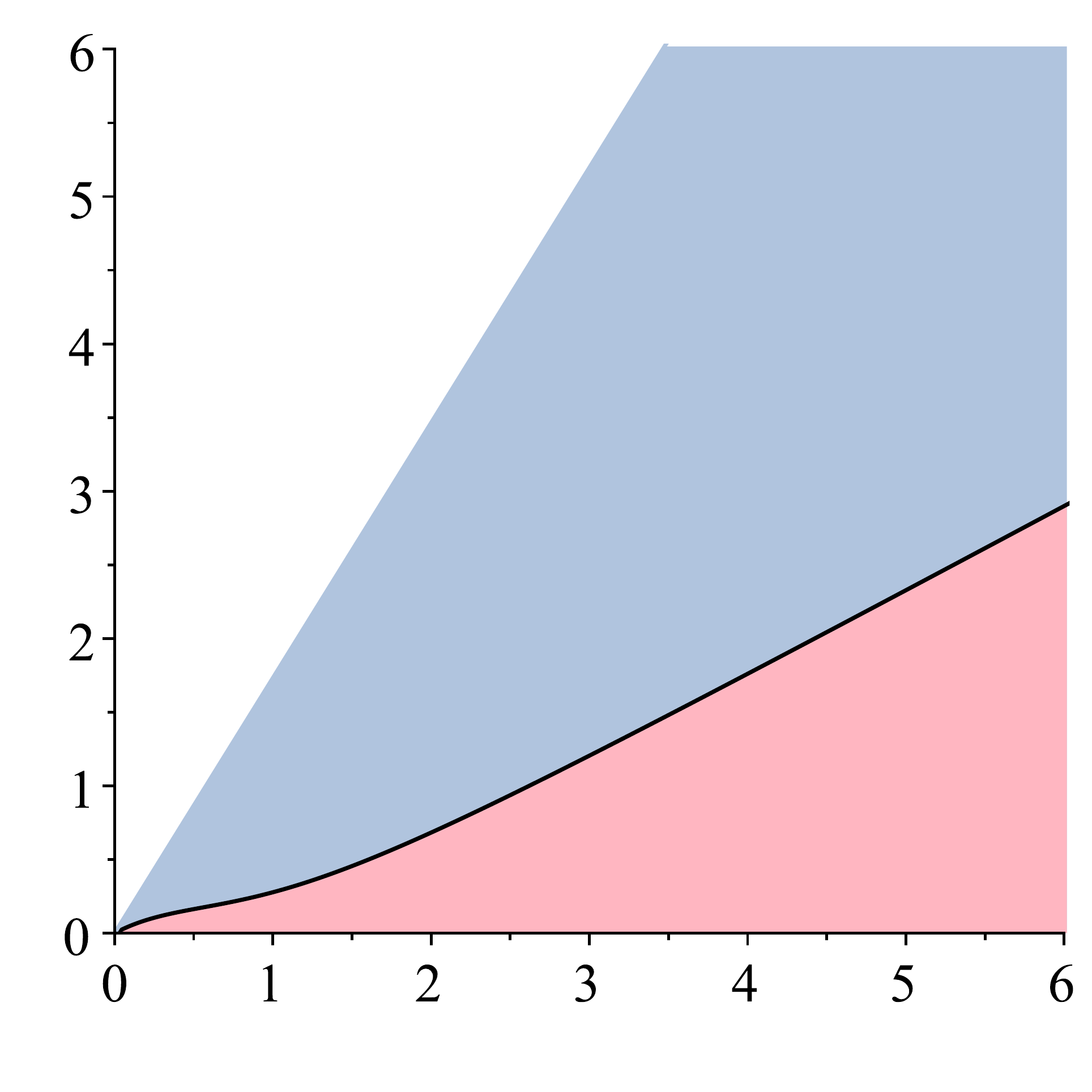}};
		\node at (10.3,0) {\includegraphics[scale=0.2]{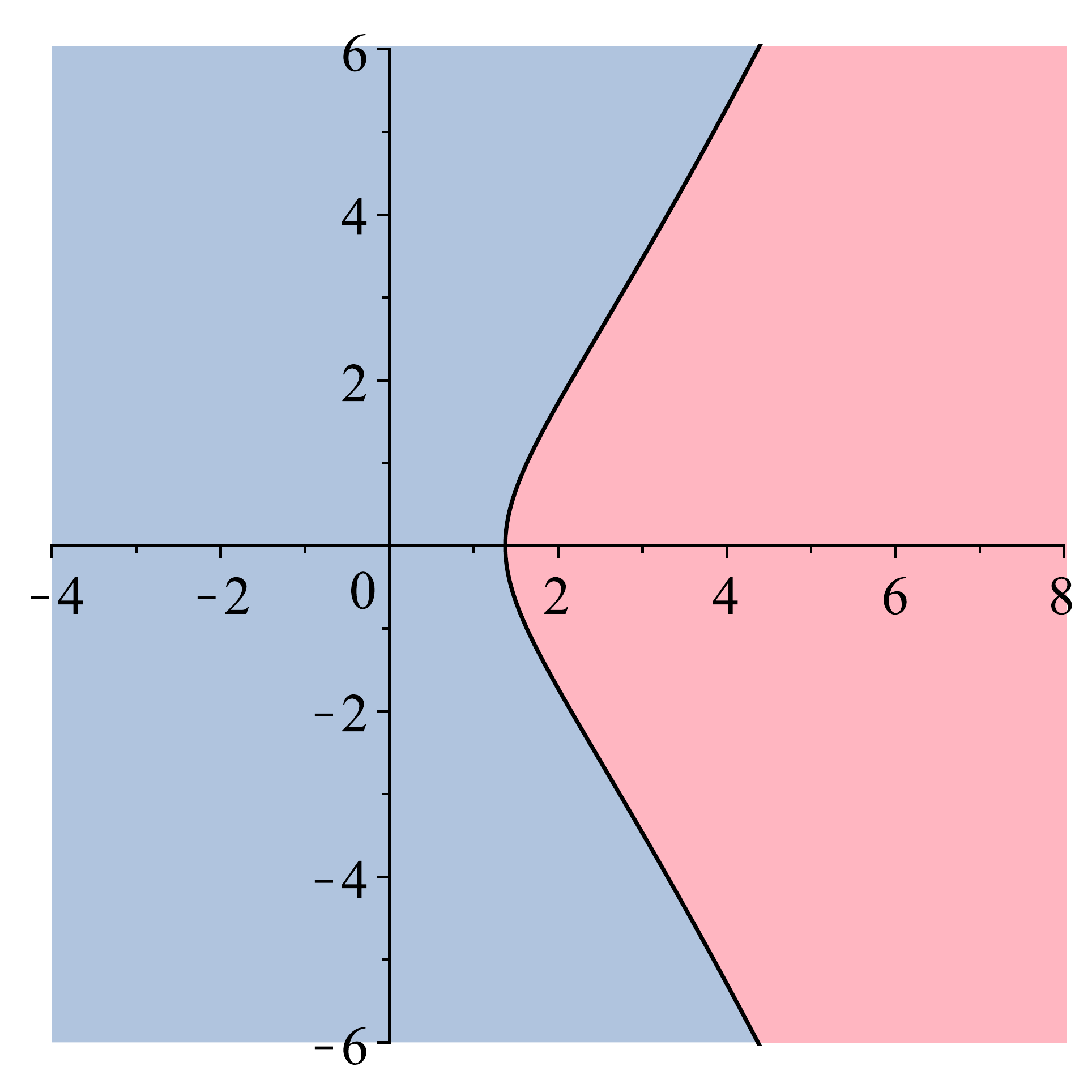}};
		
		\draw[->,thick] (1,2.1) .. controls (2.7,3) and (5.5,2) ..(6.2,-1);
		\draw[->,thick] (0.5,-2.4) .. controls (2,-3) and (4,-3) .. (4.5,-0.5);
		\draw[->,thick] (7,0)--(8.3,0);
		
		\node at (4,2.5) {$x^{-1}(s)$};
		\node at (4,-2.5) {$y^{-1}(s)$};
		\node at (7.7,0.3) {$z^3+z^{-3}$};
	\end{tikzpicture}
	\caption{The model $p_{1,0}=p_{0,-1}=p_{-1,1}=1/3$ (also known as the tandem walk, due to its links with queuing theory): the inverses $x^{-1}(s)$ and $y^{-1}(s)$ map respectively and conformally $\mathcal{S}_1^+\setminus[x_1,1]$ (red domain in the top left figure) and $\mathcal{S}_2^+\setminus[x_1,1]$ (blue domain in the bottom left figure) onto lower (in red) and upper (in blue) domains in the cone; the mapping $\omega(z):=z^3+z^{-3}$ maps the cone onto the plane $\mathbb{C}$ cut along some segments; $\psi_1(z)$ and $\psi_2(z)$ in Theorem~\ref{thm: main_intro-1} can be chosen as $\psi_1=\omega\circ x^{-1}$ and $\psi_2=\omega\circ y^{-1}$.}
	\label{fig: tandem}
\end{figure}
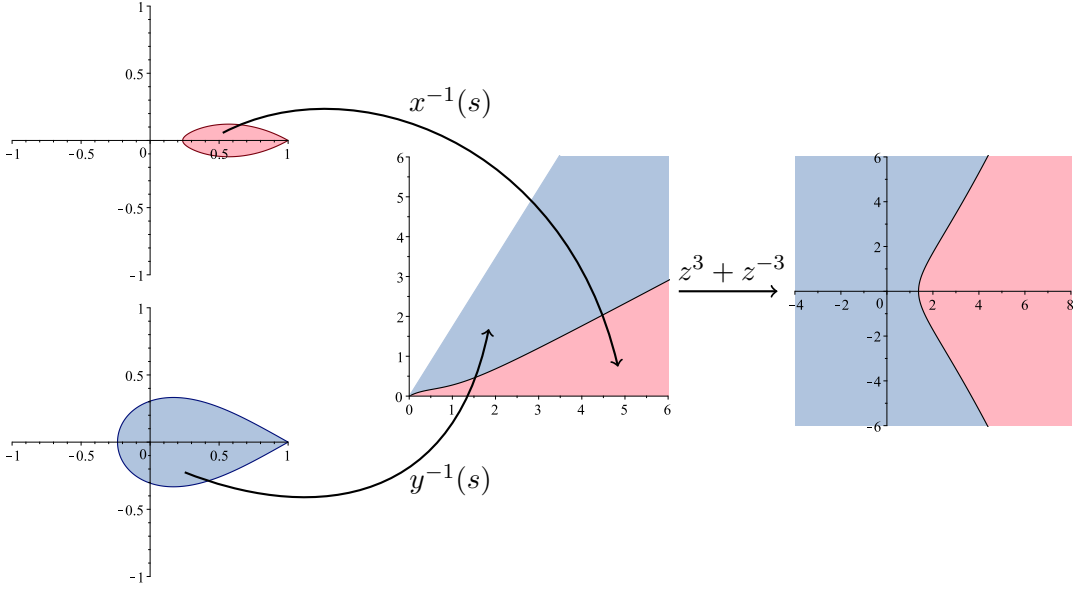

Consider the problem under Assumptions~\ref{assumption: H1}--\ref{assumption: H3} with the additional hypothesis $p_{i,j}=0$ if $\vert i\vert \geq 2$ or $\vert j\vert \geq 2$. Notice that the problem of any model $\{p_{k,\ell}\}_{k,\ell}$ with $p_{0,0}\not=0$ is equivalent to the problem of the model $\{p'_{k,\ell}\}_{k,\ell}$ with $p'_{0,0}=0$ and $p'_{k,\ell} = p_{k,\ell}/(1-p_{0,0})$ for all $k,\ell$. Then without loss of generality, we assume further $p_{0,0}=0$. We can write the kernel under the form
\begin{align*}
	K(x,y) = a(x)y^2 + b(x)y + c(x),
\end{align*}
where
\begin{align*}
	& a(x) = -(p_{-1,-1}x^2 + p_{0,-1}x + p_{1,-1}),\\
	& b(x) = -(p_{-1,0}x^2 - x +p_{1,0}),\\
	& c(x) = -(p_{-1,1}x^2 + p_{0,1}x + p_{1,1}).
\end{align*}

The corresponding discriminant $d(x) = b(x)^2-4a(x)c(x)$ is a polynomial of order $3$ or $4$, and in either case, $d(x)$ always has one root $x_1\in[-1,1)$, a double root $1$. If $d(x)$ has order $4$, then the remaining root, denoted as $x_4$, is in $ (1,\infty)\cup(-\infty,-1]$. If $d(x)$ has order $3$, we denote conventionally $x_4=\infty$. We also obtain $y_1\in [-1,1)$ and $y_4\in (-\infty,-1]\cup(1,\infty]$ similarly.

Put
\begin{align*}
	&s_0 =\frac{2-(x_1+x_4)+2\sqrt{(1-x_1)(1-x_4)}}{x_4-x_1},\\
	&s_1 = \frac{x_1+x_4-2x_1x_4+2\sqrt{x_1x_4(1-x_1)(1-x_4)}}{x_4-x_1},\\
	&s_2 =\frac{2-(y_1+y_4)+2\sqrt{(1-y_1)(1-y_4)}}{y_4-y_1},\\
	&s_3 = \frac{y_1+y_4-2y_1y_4+2\sqrt{y_1y_4(1-y_1)(1-y_4)}}{y_4-y_1},\\	
	&\rho = e^{-i\theta},
\end{align*}
where $\theta:=(\theta_1+\theta_2)/2$ and $\theta_1$, $\theta_2$ are defined in \eqref{eq: theta_1 theta_2}.

We recall an important result about the rational uniformization. We refer to \cite[Sec.~2.3]{FaRa-11} for the proof.
\begin{lemma}
	\label{lem: unif}
	Assuming \ref{assumption: H1}--\ref{assumption: H3} and that $p_{i,j}=0$ if $\vert i\vert\geq 2$ or $\vert j\vert \geq 2$, one has
	\begin{equation*}
	\{(x,y)\in(\mathbb C\cup\{\infty\})^2: K(x,y)=0\}=\{(x(s),y(s)) : s\in\mathbb C\cup\{\infty\}\},
	\end{equation*}
	where
	\begin{equation*}
	x(s) = \frac{(s-s_1)(s-\frac{1}{s_1})}{(s-s_0)(s-\frac{1}{s_0})} \quad \text{and} \quad
	y(s) = \frac{(\rho s-s_3)(\rho s-\frac{1}{s_3})}{(\rho s-s_2)(\rho s-\frac{1}{s_2})}.
	\end{equation*}
	Moreover, the above rational functions admit the involutions $x(s)=x(1/s)$ and $y(s)=y(1/(\rho^2s))$.
\end{lemma}

Put $\mathcal{P}:= x^{-1}(\mathcal{S}_1)\cap y^{-1}(\mathcal{S}_2)$. 

\begin{lemma}
	Under Assumptions \ref{assumption: H1}--\ref{assumption: H3} and \ref{assumption: monotonic}--\ref{assumption: alpha'(z)}, $\mathcal{P}$ is a non-self-intersecting curve with two endpoints $0$ and infinity, lying in the cone $\mathcal{E}:=\{re^{i\phi}:r\geq 0,\phi\in [0,\theta]\}$. Consequently, $\mathcal{P}$ divides $\mathcal{E}$ into two domains, denoted as $\mathcal{P}^+$ and $\mathcal{P}^-$ (see Figure~\ref{fig: tandem}), such that:
	\begin{enumerate}[label=\textnormal{(\roman{*})},ref=\textnormal{(\roman{*})}]
		\item $x(s)$ maps conformally $\mathcal{P}^+$ onto $\mathcal{S}_1^+\setminus[x_1,1]$;
		\item $y(s)$ maps conformally $\mathcal{P}^-$ onto $\mathcal{S}_2^+\setminus[y_1,1]$.
	\end{enumerate}
\end{lemma}

\begin{proof}
	We first rewrite $x(s)$ and $y(s)$ as
	\begin{equation*}
		x(s) = 1+\frac{\left(s_0+\frac{1}{s_0}\right)- \left(s_1+\frac{1}{s_1}\right)}{\left(s+\frac{1}{s}\right)-\left(s_0+\frac{1}{s_0}\right)}\quad\text{and}\quad y(s) = 1+\frac{\left(s_2+\frac{1}{s_2}\right)- \left(s_3+\frac{1}{s_3}\right)}{\left(\rho s+\frac{1}{\rho s}\right)-\left(s_2+\frac{1}{s_2}\right)}.
	\end{equation*}
	It is worth mentioning that $\{s_i+1/s_i\}_{0\leq i\leq 3}$ are real, and $x(s)$ and $y(s)$ are composed by M\"obius transformations (which have the form $(az+b)/(cz+d)$ for $a,b,c,d\in\mathbb{C}$) and Joukowsky transformation (which has the form $z+1/z$).
	
	Let us specify the image of $\mathcal{E}$ under the maps $x(s)$ and $y(s)$. It is easily seen that
	\begin{equation}
		s_0+\frac{1}{s_0} = \frac{4-2(x_1+x_4)}{x_4-x_1}<2\leq s+\frac{1}{s}
	\end{equation}
	for all $s\in[0,\infty]$, and $x'(s)$ does not change the sign on $(0,1)$ and $(1,\infty)$. Moreover, $x(0)=x(\infty)=1$ and $x(1)=x_0$. Hence, $x(s)$ is a one-to-one mapping from $[0,1]$ onto $[x_1,1]$, and from $[1,\infty]$ onto $[x_1,1]$. Moreover, since $x'(s)$ never vanishes on $e^{i\theta}\mathbb{R}_+$, then $x(s)$ is a one-to-one mapping from $e^{i\theta}\mathbb{R}_+$ onto a Jordan curve $\mathcal{J}_1$ passing through $1$, since $x(0)=x(\infty)=1$. Similarly, $y(s)$ is a one-to-one mapping from $e^{i\theta}[0,1]$ and $e^{i\theta}[1,\infty]$ onto $[y_1,1]$, and $\mathbb{R}_+$ onto a Jordan curve $\mathcal{J}_2$ passing through $1$. This implies that $x(s)$ (resp.~$y(s)$) maps conformally $\mathcal{E}$ onto $\mathcal{J}_1^+\setminus[x_1,1]$ (resp.~$\mathcal{J}_2^+\setminus[y_1,1]$).
	
	We now prove that $\mathcal{S}_1\subset \mathcal{J}_1^+\setminus[x_1,1]$ and $\mathcal{S}_2\subset \mathcal{J}_2^+\setminus[y_1,1]$. We first show that $(x_1,1)\subset \mathcal{S}_1^+$ . Reasoning by contradiction, if $\mathcal{S}_1$ cuts $(x_1,1)$ at any point, then there exists $s'\in(0,\infty)\setminus\{1\}$ such that $x(s')$ is the intersection of $\mathcal{S}_1$ and $(x_1,1)$, and $y(s')$ is the intersection of $\mathcal{S}_2$ and $\mathbb{R}$. In other words, $y(s')$ is real, which cannot not hold, since $\left(\rho s+\frac{1}{\rho s}\right)$ is not real for any $s\in(0,\infty)\setminus\{1\}$. Hence, $(x_1,1)\subset \mathcal{S}_1^+$.
	
	Similarly, $(y_1,1)\subset\mathcal{S}_2^+$. By Prop.~\ref{prop: extend-solutions-intro}\ref{item: Prop1-item2}, we then know that $\mathcal{J}_1$ and $\mathcal{J}_2$ are respectively outside of $\mathcal{S}_1^+$ and $\mathcal{S}_2^+$. Hence,  $\mathcal{S}_1\subset \mathcal{J}_1^+\setminus[x_1,1]$ and $\mathcal{S}_2\subset \mathcal{J}_2^+\setminus[y_1,1]$.
	
	Now consider the set $x^{-1}(\mathcal{S}_1)$. By the position of $\mathcal{S}_1$ and the involution of $x(s)$ in Lemma~\ref{lem: unif}, we know that $x^{-1}(\mathcal{S}_1)$ (resp.~ $y^{-1}(\mathcal{S}_2)$) is a union of two curves, both of which have two endpoints $0$ and infinity, but one is in $\mathcal{E}$ and the other one is in $e^{-i\theta}\mathcal{E}$ (resp.~$e^{i\theta}\mathcal{E}$). This implies that $\mathcal{P}\subset\mathcal{E}$. The other statement in the lemma also follows.
\end{proof}

Now put
\begin{equation*}
	\omega(z):= z^{\pi/\theta} + z^{-\pi/\theta}.
\end{equation*}
It is seen that $\omega(z)$ is the composition of $z\mapsto z+1/z$ and $z\mapsto z^{\pi/\theta}$. Recall $z\mapsto z^{\pi/\theta}$ maps the cone $\mathcal{E}$ onto the upper half plane $\mathcal{H}^+$. The mapping $z\mapsto z+1/z$ is the Joukowsky transform, which maps conformally the upper unit disk onto the upper half plane $\mathcal{H}^+$, with the following one-to-one correspondence on the boundary:
\begin{itemize}
	\item $[-1,1]$ corresponds to $[-\infty,-2]\cup[2,\infty]$;
	\item The upper semicircle $\{e^{i\phi}:\phi\in[0,\pi]\}$ corresponds to $[-2,2]$.
\end{itemize}
$z\mapsto z+1/z$ also maps conformally the set $\{re^{i\phi}:r>1,\phi\in(0,\pi)\}$ onto the lower half plane $\mathcal{H}^-$, with the following one-to-one correspondence on the boundary:
\begin{itemize}
	\item $[-\infty,-1]\cup[1,\infty]$ corresponds to $[-\infty,-2]\cup[2,\infty]$;
	\item The upper semicircle $\{e^{i\phi}:\phi\in[0,\pi]\}$ corresponds to $[-2,2]$.
\end{itemize}
Hence, $z\mapsto z+1/z$ maps conformally the upper half plane onto the plan cut $\mathbb{C}\setminus((-\infty,-2]\cup[2,\infty))$. Thus, $\omega(z)$ maps conformally $\mathcal{E}$ onto $\mathbb{C}\setminus((-\infty,-2]\cup[2,\infty))$.

Let $\mathcal{L}$ denote $\omega(\mathcal{P})$, which is an infinite curve, and let $\mathcal{L}^+$, $\mathcal{L}^-$ respectively denote $\omega(\mathcal{P}^+)$, $\omega(\mathcal{P}^-)$. One can verify that $\omega\circ x^{-1}$ (resp.~$\omega\circ y^{-1}$) maps conformally $\mathcal{S}_1^+$ (resp.~$\mathcal{S}_2$) onto $\mathcal{L}^+$ (resp.~$\mathcal{L}^-$). Now we put
\begin{align*}
	& \psi_1(z) = \omega\circ x^{-1}(z)=2T_{\pi/\theta}\left(\frac{(s_0+1/s_0)z-(s_1+1/s_1)}{2(z-1)}\right),\\
	& \psi_2(z) = \omega\circ y^{-1}(z)=-2T_{\pi/\theta}\left(\frac{(s_2+1/s_2)z-(s_3+1/s_3)}{2(z-1)}\right),
\end{align*}
where $T_n(z)$ is a generalization of  Chebyshev polynomial of the first kind to non-integer order $n$, and is defined by:
\begin{equation*}
	T_{\pi/\theta}(z) = \frac{1}{2}\left(  (z+\sqrt{z^2-1})^{\pi/\theta} + (z-\sqrt{z^2-1})^{\pi/\theta}\right),\quad z\in\mathbb{C}\setminus(-\infty,-1),
\end{equation*}
(see Figure~\ref{fig: tandem}). It can be verified that the mappings $\psi_1$ and $\psi_2$ satisfy Lemma~\ref{lem: psi1-psi2} and thus can be chosen to construct harmonic functions in Theorem~\ref{thm: main_intro-1}.

\subsection*{Weighted simple random walk}
We now give a concrete example where all the conformal mappings in Section~\ref{sec: conformal welding} can be explicitly expressed.

Consider the random walk with the transition probabilities:
\begin{equation*}
	p_{0,1}=p_{0,-1}=\frac{3}{8},\,
	p_{1,0}=p_{-1,0}=\frac{1}{8}.
\end{equation*}
The kernel then takes the form:
\begin{equation*}
K(x,y) = xy - \frac{1}{8}\big( 3x + y + 3xy^2 + x^2y\big).
\end{equation*}
By solving the equation $K(\eta s,\eta s^{-1})=0$, one obtains:
\begin{align*}
	&\mathcal{S}_1 = \big\{ \frac{4s + {i}\sqrt{3}(s^2-1)}{s^2+3}s : s={e}^{{i}t},t\in[0,\pi) \big\},\\
	&\mathcal{S}_2 = \big\{ \frac{4s + {i}\sqrt{3}(s^2-1)}{s^2+3}s^{-1} : s={e}^{{i}t},t\in[0,\pi) \big\}.
\end{align*}

Using the uniformization constructed above, we have:
\begin{equation*}
	x(s) = \frac{(s-e^{i\pi/6})(s-e^{-i\pi/6})}{(s-e^{i5\pi/6})(s-e^{-i5\pi/6})}
	\quad\text{and}\quad y(s) = \frac{(e^{-i\pi/3}s-e^{i\pi/3})(e^{-i\pi/3}s-e^{-i\pi/3})}{(e^{-i\pi/3}s-e^{i2\pi/3})(e^{-i\pi/3}s-e^{-i2\pi/3})}.
\end{equation*}
In particular, $x(s)$ maps conformally the cone $\{re^{i\phi}:r>0,\phi\in (0,\pi/3)\}$ onto $\mathcal{S}_1^+\setminus [7-4\sqrt{3},1]$, and $y(s)$ maps conformally the cone $\{re^{i\phi}:r>0,\phi\in (\pi/3,\pi/2)\}$ onto $\mathcal{S}_2^+\setminus [1/3,1]$. We consider some conformal mappings between half-planes, disks and cones:
\begin{align*}
	&\mu:\{a+bi:a>0\}\to\mathcal{C}^+,\quad z\mapsto \frac{z-1}{z+1},\\
	&\pi_1:\{re^{i\phi}:r>0,\phi\in (0,\pi/3)\}\to  \{a+bi:a>0\},\quad z\mapsto z^{3/2}+z^{-3/2},\\
	&\pi_2:\{re^{i\phi}:r>0,\phi\in (\pi/3,\pi/2)\}\to  \{a+bi:a>0\},\quad z\mapsto  - (ze^{-i\pi/6})^3 - (ze^{-i\pi/6})^{-3}.
\end{align*}
Now put
\begin{align*}
	&\pi_1(z) =  \mu\circ\omega_1\circ x^{-1}(z)= \frac{2T_{3/2}(-\frac{\sqrt{3}}{2}\frac{z+1}{z-1})-1}{2T_{3/2}(-\frac{\sqrt{3}}{2}\frac{z+1}{z-1})+1},\\
	&\pi_2(z) = \mu\circ\omega_2\circ y^{-1}(z)= \frac{2T_{3}(-\frac{1}{2}\frac{z+1}{z-1})-1}{2T_{3}(-\frac{1}{2}\frac{z+1}{z-1})+1}.
\end{align*}
It can be verified that $\pi_1$ and $\pi_2$ are conformal mappings respectively from $\mathcal{S}_1^+$ and $\mathcal{S}_2^+$ onto $\mathcal{C}^+$, and satisfy Lemma~\ref{lem: pi_1}. Accordingly, the shift function $\alpha:\mathbb{R}\to\mathbb{R}$ in \eqref{eq: alpha(z)} has the form
\begin{equation*}
	\alpha(z) = \sign(z)\sqrt{\sqrt{z^2+4}-2},
\end{equation*}
and $\chi_1$ and $\chi_2$ in Lemma~\ref{lem: chi1-chi2-Jordan curve} admit the expressions
\begin{equation*}
	\chi_1(z) = \frac{2T_{4/3}(-iz/2)-1}{2T_{4/3}(-iz/2)+1}
	\quad\text{and}\quad
	\chi_2(z) = \frac{-2T_{2/3}(iz/2)-1}{-2T_{2/3}(iz/2)+1}.
\end{equation*}
The conformal mappings $\psi_1$ and $\psi_2$ in Lemma~\ref{lem: psi1-psi2} then follow:
\begin{equation*}
	\psi_1(z) = 2T_2(-\frac{\sqrt{3}}{2}\frac{z+1}{z-1})=\frac{z^2+10z+1}{(z-1)^2}
	\quad\text{and}\quad
	\psi_2(z) = -2T_2(-\frac{1}{2}\frac{z+1}{z-1})=\frac{z^2-6z+1}{(z-1)^2}.
\end{equation*}
With the family of polynomials $\{P_n(x)\}_{n\geq 1} = \{(x-1)^n\}_{n\geq 1}$, the harmonic functions $h_n(i,j)$ and theirs generating functions $H_n(x,y)$ in Theorem~\ref{thm: main_intro-1} can be explicitly expressed, for example, we have:
\begin{align*}
	&H_1(x,y) = \frac{-32}{(x-1)^2(y-1)^2}=-32\sum_{i,j\geq 1} ij x^{i-1}y^{j-1},\\
	&H_2(x,y) = \frac{128(x^2y-3xy^2+4xy-3x+y)}{(x-1)^4(y-1)^4}=128\sum_{i,j\geq 1}\frac{ij(-3i^2+2j^2+2)}{6}  x^{i-1}y^{j-1},\\
	& ...
\end{align*}
We finally remark that $-h_1(i,j)/32=ij$ is the unique positive harmonic function (up to multiplicative factors) for this example.

\bibliographystyle{plain}

\end{document}